\documentclass[12pt,a4paper]{article}
\usepackage{amsmath}
\usepackage{amssymb}
\usepackage{amsthm}
\usepackage{amsfonts}
\usepackage{latexsym}

\theoremstyle{plain}
\newtheorem{thm}{Theorem}[section]
\newtheorem{cor}{Corollary}[section]
\newtheorem{lem}{Lemma}[section]
\newtheorem{prop}{Proposition}[section]
\newtheorem{claim}{Claim}[section]

\theoremstyle{remark}

\theoremstyle{definition}
\newtheorem{defn}{Definition}[section]

\newtheorem{rem}{Remark}[section]
\newtheorem{asmpt}{Basic Assumption}[section]

\title{The symplectic structure\\ of a toric conic transform}
\author{Roberto Paoletti\footnote{\noindent{\bf Address:}
Dipartimento di Matematica e Applicazioni, Universit\`a degli Studi
di Milano-Bicocca, Via R. Cozzi 55, 20126 Milano,
Italy; {\bf e-mail}: roberto.paoletti@unimib.it }}
\date{}

\begin{document}
\maketitle

\begin{abstract}
Suppose that a compact $r$-dimensional torus $T^r$ acts in a holomorphic and Hamiltonian
manner on polarized complex $d$-dimensional projective manifold $M$, 
with nowhere vanishing moment map $\Phi$.
Assuming that $\Phi$ is transverse to the ray through a given weight $\boldsymbol{\nu}$, associated
to these data there is a complex $(d-r+1)$-dimensional polarized projective orbifold
$\widehat{M}_{\boldsymbol{\nu}}$ (referred to as the $\boldsymbol{\nu}$-th \textit{conic transform} of $M$). Namely, $\widehat{M}_{\boldsymbol{\nu}}$ is a suitable
quotient of the inverse image of the ray in the unit circle bundle of the polarization of $M$.
With the aim to clarify the geometric significance of this construction, we consider the
special case where $M$ is toric, and show that $\widehat{M}_{\boldsymbol{\nu}}$ is itself a 
K\"{a}hler toric
obifold, whose moment polytope is obtained from the one of $M$ by 
a certain \lq transform\rq\, operation (depending on $\Phi$ and $\boldsymbol{\nu}$).
\end{abstract}

\section{Introduction}

Consider a $d$-dimensional connected projective manifold $M$, with complex structure $J$, 
and let $(A,h)$ be a
positive holomorphic line bundle on $M$. 
Thus $A$ is ample, $h$ is a Hermitian metric on it, and 
the unique covariant derivative $\nabla$ on 
$A$ compatible with both the metric and the complex structure has
curvature $\Theta=-2\,\imath\,\omega$, with $\omega\in \Omega^2(M)$ a K\"{a}hler form
on $(M,J)$.

We shall denote by $A^\vee$ the dual line bundle to $A$, 
and by $X\subset A^\vee$ the unit circle bundle, with bundle projection
$\pi:X\rightarrow M$. 
Then $\nabla$ corresponds to a connection 1-form $\alpha\in \Omega^1(X)$, which is a contact form on $X$ and satisfies 
\begin{equation}
\label{eqn:contactalpha}
\mathrm{d}\alpha=2\,\pi^*(\omega),\quad .
\end{equation}

Let $T^r$ be an $r$-dimensional compact torus, 
with Lie algebra $\mathfrak{t}^r$ and coalgebra ${\mathfrak{t}^r}^\vee$.
Furthermore, let $\mu^M:T^r\times M\rightarrow M$ a holomorphic
and Hamiltonian action of $T^r$ on $(M,2\,\omega,J)$, with moment map
$\Phi:M\rightarrow {\mathfrak{t}^r}^\vee$ (see e.g. \cite{gs-stp} for general
background on Hamiltonian actions and moment maps). 

Any $\boldsymbol{\xi}\in {\mathfrak{t}^r}$ thus determines a Hamiltonian vector field $\boldsymbol{\xi}_M$ on $M$. 
As is well-known \cite{ko}, $\Phi$ determines a natural lift of 
$\boldsymbol{\xi}_M$ 
to a contact vector
field $\boldsymbol{\xi}_X=\boldsymbol{\xi}_X^\Phi$ on $(X,\alpha)$, given by
\begin{equation}
\label{eqn:contact lift}
\boldsymbol{\xi}_X:=\boldsymbol{\xi}_M^\sharp-\langle\Phi,\boldsymbol{\xi}\rangle\,
\partial_\theta;
\end{equation}
here $V^\sharp$ is the horizontal lift to $X$, with respect to $\alpha$, of a vector field $V$ on $M$, and $\partial_\theta$ is the generator of the structure 
$S^1$-action on $X$, given by counterclockwise fiber rotation.
The flow of $\boldsymbol{\xi}_X$ preserves the contact and CR structures of $X$,
and the flows of $\boldsymbol{\xi}_X$ and $\boldsymbol{\xi}_X'$ commute, for any
$\boldsymbol{\xi},\,\boldsymbol{\xi}'\in {\mathfrak{t}^r}$.

We shall make the stronger hypothesis that
$\mu^M$ lifts to a metric preserving line bundle action on $A$, 
and the induced action $\mu^X:T^r\times X\rightarrow X$ has the correspondence
$\boldsymbol{\xi}\mapsto \boldsymbol{\xi}_X$ in (\ref{eqn:contact lift}) as its differential.
We shall say that $\mu^X$ is the (contact and CR) lift of the homolomorphic Hamiltonian action
$(\mu^M,\Phi)$.

For example, when $r=1$ and $\mu^M$ is trivial, $\Phi:M\rightarrow {\mathfrak{t}^1}^\vee$
is constant; choosing $\Phi=\imath$ in (\ref{eqn:contact lift}) yields the circle action $\rho^X$ 
generated by $-\partial_\theta$, thus
given by clockwise fiber rotation.
If $\partial_\theta^{S^1}$ is the standard generator of the Lie algebra
of $S^1$, $\partial_\theta=(\partial_\theta^{S^1})_X$ in (\ref{eqn:contact lift}) is the vector field on $X$ generating
the structure $S^1$-action given by counter-clockise fiber rotation, while $-\partial_\theta$
is the vector field generating $\rho^X$ (clockwise fiber rotation); we parametrize $S^1$ by
$\theta\mapsto e^{\imath\,\theta}$.

Let us fix a non-zero weight $\boldsymbol{\nu}\in {\mathfrak{t}^r}^\vee$.
The results below rest on the following Basic Assumption
on $(\Phi,\boldsymbol{\nu})$, henceforth referred to a
BA \ref{asmpt:basic}.

\begin{asmpt} \label{asmpt:basic}
The following holds:

\begin{enumerate}
\item $\boldsymbol{\nu}$ is primitive (or coprime);
\item $\Phi$ is nowhere vanishing, that is, $\mathbf{0}\not\in \Phi(M)$;
\item $\Phi$ is transverse to the ray $\mathbb{R}_+\cdot \boldsymbol{\nu}$.
\end{enumerate}

\end{asmpt}

Under these circumstances, a polarized K\"{a}hler orbifold
$(\widehat{M}_{\boldsymbol{\nu}},\widehat{\omega}_{\boldsymbol{\nu}},\widehat{J}_{\boldsymbol{\nu}})$
can be constructed from the previous data,
by taking a suitable quotient by a locally free action of 
$T^r$ of a locus in $X$ defined
by $(\Phi,\boldsymbol{\nu})$; here $\widehat{\omega}_{\boldsymbol{\nu}}$ and $\widehat{J}_{\boldsymbol{\nu}}$ 
denote the (orbifold) symplectic and complex structures on $\widehat{M}_{\boldsymbol{\nu}}$ \cite{pao-JGP}.
We refer to \cite{pao-JGP} 
(where $\widehat{M}_{\boldsymbol{\nu}}$ is denoted $N_{\boldsymbol{\nu}}$ and 
$\widehat{\omega}_{\boldsymbol{\nu}}$ by $\eta_{\boldsymbol{\nu}}$)
for a discussion of the relevance 
of this geometric construction in geometric quantization; it generalizes the 
one of weighted projective spaces as quotients of an odd-dimensional sphere.
Here we aim to clarify the relation between the symplectic
structures of $M$ and $\widehat{M}_{\boldsymbol{\nu}}$ in the toric setting: as we shall see, 
assuming
that $M$ is a toric manifold, 
$(\widehat{M}_{\boldsymbol{\nu}},\,2\,\widehat{\omega}_{\boldsymbol{\nu}})$ 
is a toric symplectic orbifold, and its marked moment polytope
$\widehat{\Delta}_{\boldsymbol{\nu}}$ can be explicitly 
recovered from the moment polytope $\Delta$ of $M$
(by \cite{lt} toric symplectic orbifolds are classified by marked
convex rational simple polytopes).

Before stating the result precisely, let us briefly recall the geometric construction in
point, referring to \cite{pao-IJM} and \cite{pao-JGP} for details.
Let us set 
\begin{equation}
\label{eqn:Mnudefn}
M_{\boldsymbol{\nu}}:=\Phi^{-1}(\mathbb{R}_+\,\boldsymbol{\nu})\subseteq M,\quad
X_{\boldsymbol{\nu}}:=\pi^{-1}(M_{\boldsymbol{\nu}})\subseteq X.
\end{equation}
Then, assuming BA \ref{asmpt:basic}, the following holds:
\begin{enumerate}
\item $M_{\boldsymbol{\nu}}\subseteq M$ is a connected and compact (real) submanifold, 
of codimension $r-1$;
\item $\mu^X$ is locally free on $X_{\boldsymbol{\nu}}$.
\end{enumerate}
We may and will assume without loss that $\mu^X$ is generically free
on $X$ (and $X_{\boldsymbol{\nu}}$).
Then the quotient 
\begin{equation}
\label{eqn:defn Mnu}
\widehat{M}_{\boldsymbol{\nu}}:=X_{\boldsymbol{\nu}}/T^r
\end{equation}
(denoted $N_{\boldsymbol{\nu}}$
in \cite{pao-JGP}) is naturally a $(d-r+1)$-dimensional complex orbifold,
and comes equipped with a K\"{a}hler structure $(\widehat{M}_{\boldsymbol{\nu}},\,
\widehat{\omega}_{\boldsymbol{\nu}},
\widehat{J}_{\boldsymbol{\nu}})$ induced by $(M,\omega,J)$;
here $T^r$ acts on $X_{\boldsymbol{\nu}}$ by the restriction of $\mu^X$
to $X_{\boldsymbol{\nu}}$.
We shall call $\widehat{M}_{\boldsymbol{\nu}}$ the $\boldsymbol{\nu}$-th \textit{conic transform} of $M$;
it depends on $\mu^X$, hence on $\Phi$.

We are interested in clarifying the geometry of $\widehat{M}_{\boldsymbol{\nu}}$ in the toric setting,
thus assuming that $M$ be toric, with structure action
$\gamma^M:T^d\times M\rightarrow M$ and moment polytope $\Delta\subseteq \mathfrak{t}^\vee$.

Let us briefly recall the Delzant construction of $M$ from $\Delta$; obviously
with no pretense of exhaustiveness, we refer
to \cite{del}, \cite{g-toric}, and \cite{lt} for more complete discussions.
Since $M$ is smooth, $\Delta$ is a Delzant polytope \cite{g-toric}.
We shall denote by $\mathcal{F}(\Delta)$ the collection 
of all faces of $\Delta$, by $\mathcal{F}_l(\Delta)\subseteq \mathcal{F}(\Delta)$ 
the subset of codimension-$l$ faces, and specifically
by $\mathcal{G}(\Delta)=\mathcal{F}_1(\Delta)$ the subset of facets.
If $\mathcal{G}(\Delta)=\{F_1,\ldots,F_k\}$, then 
for every $j=1,\ldots,k$ there exists unique 
\begin{equation}
\label{eqn:upsilonjlambdaj}
\boldsymbol{\upsilon}_j\in L(T^d):=\ker\big(\exp_{T^d}(2\pi\cdot )\big)\subset \mathfrak{t}^d
=\mathrm{Lie}(T^d),\quad \lambda_j\in \mathbb{R},
\end{equation}
with $\boldsymbol{\upsilon}_j$ primitive, such that
\begin{equation}
\label{eqn:Delta inters}
\Delta=\bigcap_{j=1}^k\{\ell \in \mathfrak{t}^\vee\,:\,\ell (\boldsymbol{\upsilon}_j)\ge \lambda_j\},
\end{equation}
and for every $j=1,\ldots,k$ the relative interior of $F_j$ (open facet) is
\begin{equation}
\label{eqn:jth face}
F_j^0=\{\ell \in \mathfrak{t}^\vee\,:\,\ell (\boldsymbol{\upsilon}_j)= \lambda_j,\,
\ell (\boldsymbol{\upsilon}_{j'})> \lambda_{j'}\,\forall\,j'\neq j\}.
\end{equation}

Let us set
\begin{equation}
\label{eqn:standardsimp}
\omega_0:=\frac{\imath}{2}\,\sum_{j=1}^k\,\mathrm{d}z_j\wedge \mathrm{d}\overline{z}_j=
\sum_{j=1}^k\,\mathrm{d}x_j\wedge \mathrm{d}y_j.
\end{equation}
Then $M$ can be regarded as symplectic reduction of $(\mathbb{C}^k,2\,\omega_0)$ under the standard action of $T^k$, as follows. 
Denote the general element of $T^k$ by 
$e^{\imath\,\boldsymbol{\vartheta}}=\left(e^{\imath\,\vartheta_1},\ldots,e^{\imath\,\vartheta_k}\right)$;
for $e^{\imath\,\boldsymbol{\vartheta}}\in T^k$ and 
$\mathbf{z}=(z_j)_{j=1}^k\in \mathbb{C}^k$, let us set
$e^{\imath\,\boldsymbol{\vartheta}}\bullet \mathbf{z}:=\left( e^{\imath\,\vartheta_j}\,z_j \right)_{j=1}^k$.
For any choice of $\boldsymbol{\lambda}=(\lambda_j)_{j=1}^k\in \mathbb{R}^k$, the action 
\begin{equation}
\label{eqn:GammaCk}
\Gamma^{\mathbb{C}^k}:\left( e^{\imath\,\boldsymbol{\vartheta}},\mathbf{z}  \right)\in T^k\times \mathbb{C}^k
\mapsto e^{-\imath\,\boldsymbol{\vartheta}}\bullet \mathbf{z}\in \mathbb{C}^k
\end{equation}
is then Hamiltonian on $\left(\mathbb{C}^k,2\,\omega_0 \right)$, with moment map
\begin{equation}
\label{eqn:psilambda}
\Psi_{\boldsymbol{\lambda}}:\mathbf{z}\in \mathbb{C}^k\mapsto \imath\,\sum_{j=1}^k\left(|z_j|^2+\lambda_j\right)\,\mathbf{e}_j^*,
\end{equation}
where $(\mathbf{e}_j)_{j=1}^k$ is the canonical basis of $\mathbb{R}^k$ and $(\mathbf{e}_j^*)_{j=1}^k$ is the dual basis.
The linear map $\mathbb{R}^k\rightarrow \mathbb{R}^d$ such that
$\mathbf{e}_j\mapsto \mathbf{u}_j$ induces a short exact sequence of tori
$0\rightarrow N\rightarrow T^k\rightarrow T^d\rightarrow 0$;
hence $\Gamma$ restricts to a Hamiltonian action of $N$ 
on $\left(\mathbb{C}^k,2\,\omega_0 \right)$, with a moment map
$\Psi^N_{\boldsymbol{\lambda}}$ naturally induced from $\Psi_{\boldsymbol{\lambda}}$. 
Given that $\Delta$ is Delzant, $N$ acts freely on 
$Z_{\Delta}:={\Psi^N_{\boldsymbol{\lambda}}}^{-1}(\mathbf{0})$;
then $M=Z_{\Delta}/N$, with its symplectic structure
$2\,\omega$,
is the Marsden-Weinstein reduction of $\left(\mathbb{C}^k,2\,\omega_0 \right)$ for the action of $N$.
By the arguments of \cite{gs-gq}, the standard complex structure 
$J_0$ of $\mathbb{C}^k$ descends to
a compatible complex structure $J$ on $M$, whence 
$(M,\omega,J)$ is a K\"{a}hler manifold.

Furthermore, $\Gamma$ descends to a holomorphic and Hamiltonian action of $T^d=T^k/N$ on 
$(M,2\,\omega,J)$, $\gamma^M:T^d\times M\rightarrow M$; the 
moment map $\Psi:M\rightarrow {\mathfrak{t}^d}^\vee$
is obtained by descending to the quotient the restriction
$$
\left.\Psi_{\boldsymbol{\lambda}}\right|_Z:Z_{\Delta}\rightarrow \mathfrak{n}^0\cong {\mathfrak{t}^d}^\vee.
$$

In addition, if $\boldsymbol{\lambda}\in \mathbb{Z}^k$ this
construction can be extended by the arguments of \cite{gs-gq} 
so as to obtain an induced toric positive line bundle
$(A,h)$ on $M$, with curvature $\Theta=-2\,\imath\,\omega$ (\S \label{sctn:prel}); that $(A,h)$ is
toric means that $\gamma^M$ lifts to a metric preserving line bundle action of $T^d$ on
$A$. Hence by restriction we obtain a contact CR action $\gamma^X:T^d\times X\rightarrow X$
lifting $\gamma^M$, where
$X\subset A^\vee$ is the unit circle bundle.

In addition, we suppose given an effective holomorphic and Hamiltonian action 
$\mu^M:T^r\times M\rightarrow M$ of an $r$-dimensional
compact torus $T^r$ on $M$, with moment map $\Phi:M\rightarrow {\mathfrak{t}^r}^\vee$
satisfying BA \ref{asmpt:basic} for a certain $\boldsymbol{\nu}\in {\mathfrak{t}^r}^\vee$, 
and commuting with $\gamma^M$. 
Thus $\mu^M$ factors through an injective group homomorphism
$T^r\rightarrow T^k$, hence we may assume wihout loss of generality
that $T^r\leqslant T^d$ and that $\mu^M$ is the retriction of
$\gamma^M$ to $T^r$; 
therefore, letting $\iota:\mathfrak{t}^r\hookrightarrow
\mathfrak{t}^d$ be the Lie algebra inclusion, 
\begin{equation}
\label{eqn:Trpsi}
\Phi=\iota^t\circ \Psi+\boldsymbol{\delta}
\end{equation}
for some constant 
$\boldsymbol{\delta}\in {\mathfrak{t}^r}^\vee$. Equivalently, given $\tilde{\boldsymbol{\delta}}\in 
{\mathfrak{t}^d}^\vee$ such that $\boldsymbol{\delta}=\iota^t( \tilde{\boldsymbol{\delta}} )$,
\begin{equation}
\label{eqn:Trpsi1}
\Phi=\iota^t\circ (\Psi_{\tilde{\boldsymbol{\delta}}}),\qquad\text{where}\qquad
\Psi_{\tilde{\boldsymbol{\delta}}}:=\Psi+\tilde{\boldsymbol{\delta}}.
\end{equation}

Let us assume that $(\mu^M,\Phi)$ lifts to $\mu^X:T^r\times X\rightarrow X$ according to the previous procedure. While $\mu^M$ is the restriction of $\gamma^M$ to $T^r$, 
$\mu^X$ is the restriction of $\gamma^X$ only if $\boldsymbol{\delta}=\mathbf{0}$
in (\ref{eqn:Trpsi}).

If $\mu^X$ exists, we can consider the
conic transform $\widehat{M}_{\boldsymbol{\nu}}$ with respect to $\mu^X$; as mentioned,
$(\widehat{M}_{\boldsymbol{\nu}}, 2\,\widehat{\omega}_{\boldsymbol{\nu}})$ 
turns out to be a symplectic toric orbifold. 
Furthermore, its associated marked convex rational simple polytope 
$(\widehat{\Delta}_{\boldsymbol{\nu}},\mathbf{s}_{\boldsymbol{\nu}})$
is obtained by
applying a suitable \lq transform\rq\, to $\Delta$ (depending on $\boldsymbol{\nu}$).

Since the situation is at its simplest when $r=1$,
we shall describe this case first. 
Thus $\mu^M:T^1\times M\rightarrow M$ is a Hamiltonian action
on $(M,2\,\omega)$, 
with a nowhere vanishing moment map $\Phi:M\rightarrow{\mathfrak{t}^1}^\vee$;
the primitive integral weight $\boldsymbol{\nu}\in {\mathfrak{t}^1}^\vee$ is uniquely determined by the condition
that $M_{\boldsymbol{\nu}}\neq \emptyset$.
Then $M_{\boldsymbol{\nu}}=M$, $X_{\boldsymbol{\nu}}=X$, $\mu^X$ is locally free, 
and 
$\widehat{M}_{\boldsymbol{\nu}}=X/T^1$ (the quotient is with respect to $\mu^X$).


Let us choose a complementary torus $T^{d-1}_c$ to $T^1$ in $T^d$, that is,
$T^d\cong T^{d-1}_c\times T^1$. If $\mathfrak{t}^{d-1}_c\leqslant
\mathfrak{t}^d$
is the Lie algebra of $T^{d-1}_c$, the corresponding lattices $L(T^{d-1}_c)\subset \mathfrak{t}^{d-1}_c$ and 
$L(T^1)\subset \mathfrak{t}^1$ are complementary in $L(T^d)$ (see
\S \ref{sctn:proofT1} on how $\widehat{\Delta}_{\boldsymbol{\nu}}$ depends 
on $T^{d-1}_c$). 

Let $\widetilde{\boldsymbol{\nu}}\in L(T^1)$
be the unique primitive lattice vector such that $\boldsymbol{\nu}
\left(\widetilde{\boldsymbol{\nu}}\right)>0$; since the weight lattice is the dual
lattice to $L(T)$, primitivity implies $\boldsymbol{\nu}
\left(\widetilde{\boldsymbol{\nu}}\right)=1$, that is,
$\boldsymbol{\nu}=\widetilde{\boldsymbol{\nu}}^*\in L(T^1)^\vee$ is the dual vector to 
$\widetilde{\boldsymbol{\nu}}$. Then $\boldsymbol{\delta}=\delta\,\boldsymbol{\nu}\in {\mathfrak{t}^1}^\vee$, where $\delta=\boldsymbol{\delta}(\widetilde{\boldsymbol{\nu}})
\in \mathbb{Z}$
(notation as in (\ref{eqn:Trpsi})).


With $\Delta$ and $\boldsymbol{\upsilon}_j$ as in (\ref{eqn:Delta inters}), for each $j=1,\ldots,k$ there are unique
$\boldsymbol{\upsilon}_j'\in \mathfrak{t}_c^{d-1}$ and $\rho_{j}\in \mathbb{Z}$ such that
\begin{equation}
\label{eqn:upsilonj decmp}
\boldsymbol{\upsilon}_j=\boldsymbol{\upsilon}_j'+\rho_{j}\tilde{\boldsymbol{\nu}}.
\end{equation}
For every $j=1,\ldots,k$ let us define 
\begin{equation}
\label{eqn:hatupsilonj}
\widehat{\boldsymbol{\upsilon}}_j:=\boldsymbol{\upsilon}_j'-(\lambda_j+\rho_{j}\,\delta)\,\,\tilde{\boldsymbol{\nu}}
\end{equation}
and 
\begin{equation}
\label{eqn:Deltanu}
\widehat{\Delta}_{\boldsymbol{\nu}}:=
\bigcap_{j=1}^k\{\ell \in \mathfrak{t}^\vee\,:\,\ell \big(\widehat{\boldsymbol{\upsilon}}_j\big)\ge -\rho_j\}.
\end{equation}
%
Thus $\widehat{\Delta}_{\boldsymbol{\nu}}$ is obtained from $\Delta$ 
by replacing each pair
$(\rho_j,\lambda_j)$ by the pair $(-(\lambda_j+\rho_j\,\delta),-\rho_j)$.

We shall see that $(\widehat{M}_{\boldsymbol{\nu}},\,2\,\widehat{\omega}_{\boldsymbol{\nu}})$ is a toric symplectic orbifold, and that $\widehat{\Delta}_{\boldsymbol{\nu}}$ is its moment polytope.
To complete the combinatorial description of $(\widehat{M}_{\boldsymbol{\nu}},\,2\,\widehat{\omega}_{\boldsymbol{\nu}})$ following \cite{lt},
we need to specify the corresponding marking of $\widehat{\Delta}_{\boldsymbol{\nu}}$, 
that is, the assignment to each of its facets $\widehat{F}_j$ of an appropriate integer $s_j\ge 1$.
We shall denote the marking by $\mathbf{s}_{\boldsymbol{\nu}}=(s_j)_{j=1}^k\in \mathbb{N}^k$, 
and the marked polytope by the pair $(\widehat{\Delta}_{\boldsymbol{\nu}},\,\mathbf{s}_{\boldsymbol{\nu}})$.

We premise a further piece of notation. 
Given a rank-$r$ integral lattice $L\subset V$ in a real vector space, and a basis
$(\ell_1,\ldots,\ell_r)$ of $L$, if $\ell\in L$ we shall denote by $(l)$ the greatest common divisor of
the coefficients of $\ell$ in the given basis, that is, 
\begin{equation}
\label{eqn:gcd lattice}
(\ell):=\mathrm{G.C.D.}(\lambda_1,\ldots,\lambda_r)\quad \text{if}\quad
 \ell=\sum_{j=1}^r \lambda_j\,\ell_j.
\end{equation}
The definition is well-posed, since $(\ell)$ is independent of the choice 
of a basis of $L$.
Furthermore, the following holds:
\begin{enumerate}
\item $\ell$ is primitive in $L$ if and only if $(\ell)=1$;
\item if $T$ is a (real) torus and $\boldsymbol{\xi}\neq \mathbf{0}\in L=L(T)$, then
$e^{\vartheta\,\boldsymbol{\xi}}=1\in T$ if and only if
$e^{\imath\,\vartheta}$ is a $(\boldsymbol{\xi})$-th root of unity.
\end{enumerate}




Let us define $\mathbf{s}_{\boldsymbol{\nu}}=(s_j)\in \mathbb{N}^k$ by setting 
$$
s_j:=\left(\widehat{\boldsymbol{\upsilon}}_j\right)\qquad
(j=1,\ldots,k).
$$

\begin{thm}
\label{thm:case r=1}
Under the above assumptions, thus with $r=1$, $\big(\widehat{M}_{\boldsymbol{\nu}},\,2\,\widehat{\omega}_{\boldsymbol{\nu}}    \big)$
is the symplectic toric orbifold with associated marked polytope 
$\big(\widehat{\Delta}_{\boldsymbol{\nu}},\,\mathbf{s}_{\boldsymbol{\nu}}\big)$.

\end{thm}

The following consequence generalizes to conic transforms a well-known property of weighted projective spaces
\cite{kaw}.

\begin{cor}
\label{cor:equalcohomology/Q}
Under the previous assumptions (thus with $r=1$), 
$$
H^l(M,\mathbb{Q})\cong H^l\big(\widehat{M}_{\boldsymbol{\nu}},\,\mathbb{Q} \big)
\qquad (l=0,1,\ldots).
$$
\end{cor}

Let us now consider a general $r\le d$. 

Let $\boldsymbol{\nu}^\perp\leqslant \mathfrak{t}^r$ 
be the kernel of $\boldsymbol{\nu}$, and $T^{r-1}_{\boldsymbol{\nu}^\perp}\leqslant T^r$
the corresponding subtorus. 
Under Basic Assumption \ref{asmpt:basic}, $T^{r-1}_{\boldsymbol{\nu}^\perp}$ acts locally
freely on $M_{\boldsymbol{\nu}}$; then
$\overline{M}_{\boldsymbol{\nu}}:=M_{\boldsymbol{\nu}}/T^{r-1}_{\boldsymbol{\nu}^\perp}$,
the Marsden-Weinstein reduction of $M$ with respect 
$T^{r-1}_{\boldsymbol{\nu}^\perp}$, is a K\"{a}hler orbifold. The transversality requirement 
in Basic Assumption \ref{asmpt:basic} can be conveniently reformulated in a transversality 
condition between $\Delta+\tilde{\boldsymbol{\delta}}$ and 
${\boldsymbol{\nu}^\perp}^0\subseteq {\mathfrak{t}^d}^\vee$
(the annhilator of $\boldsymbol{\nu}^\perp$), see \S \ref{sctn:momentpolytope}.
We shall for simplicity require that $T^{r-1}_{\boldsymbol{\nu}^\perp}$
acts freely on $M_{\boldsymbol{\nu}}$, which amounts to 
$\overline{\Delta}_{\boldsymbol{\nu}}':=(\Delta+\tilde{\boldsymbol{\delta}})\cap 
{\boldsymbol{\nu}^\perp}^0$
being a Delzant polytope (see \S \ref{sctn:smoothnessDelta}). 
Then $\overline{M}_{\boldsymbol{\nu}}$ is naturally a toric K\"{a}hler manifold,
acted upon by the quotient torus $T^{d-r+1}_q:=T^d/T^{r-1}_{\boldsymbol{\nu}^\perp}$;
the associated moment polytope $\overline{\Delta}_{\boldsymbol{\nu}}$ can be identified with
$\overline{\Delta}_{\boldsymbol{\nu}}'$ under the natural isomorphism between $\mathfrak{t}^{d-r+1}_q$
(the Lie algebra of $T^{d-r+1}_q$) and ${\boldsymbol{\nu}^\perp}^0$.
The general case can then be reduced to the case $r=1$,
with $M$ replaced by $\overline{M}_{\boldsymbol{\nu}}$.

Let us choose:
\begin{enumerate}
\item a complementary subtorus
$\widehat{T}^{1}_{\boldsymbol{\nu}}\leqslant T^r$ to $T^{r-1}_{\boldsymbol{\nu}^\perp}$,
so that exists a unique 
primitive $\widetilde{\boldsymbol{\nu}}\in L(\widehat{T}^{1}_{\boldsymbol{\nu}})$ with
$\boldsymbol{\nu}(\widetilde{\boldsymbol{\nu}})=1$;
\item a complementary subtorus 
$T^{d-r}_c\leqslant T^d$ to $T^r$, with Lie algebra 
$\mathfrak{t}^{d-r}_c\leqslant \mathfrak{t}^d$,
so that
\begin{equation}
\label{eqn:factorizationT}
T^d\cong T^{d-r}_c\times {T}^{r}
\cong T^{d-r}_c\times \widehat{T}^{1}_{\boldsymbol{\nu}}\times  T^{r-1}_{\boldsymbol{\nu}^\perp},
\end{equation}
\begin{equation}
\label{eqn:factorizationfrakT}
\mathfrak{t}^d\cong \mathfrak{t}^{d-r}_c\times \widehat{\mathfrak{t}}^{1}_{\boldsymbol{\nu}}
\times  \mathfrak{t}^{r-1}_{\boldsymbol{\nu}^\perp}.
\end{equation}
\end{enumerate} 

With notation as in
(\ref{eqn:Delta inters}) and (\ref{eqn:jth face}),
suppose that the $k$ facets of $\Delta$ have been so numbered that 
$\mathcal{G}_{ \boldsymbol{\nu} }(\Delta):=\{ F_1, \ldots ,F_l  \}\subseteq \mathcal{G}(\Delta)$
is the subset of those facets of $\Delta$ such that
$(F_j+\tilde{\boldsymbol{\delta}})\cap {\boldsymbol{\nu}^\perp}^0\neq \emptyset$
(it then follows that 
$(F_j^0+\tilde{\boldsymbol{\delta}})\cap {\boldsymbol{\nu}^\perp}^0\neq \emptyset$,
see \S \ref{sctn:transverse polytopes}).
For every $j=1,\ldots,l$, let us decompose $\boldsymbol{\upsilon}_j$ according to 
(\ref{eqn:factorizationfrakT}):
\begin{equation}
\label{eqn:decompj0}
\boldsymbol{\upsilon}_j=\boldsymbol{\upsilon}_j'+\rho_j\,\widetilde{\boldsymbol{\nu}}+
\boldsymbol{\upsilon}_j'',
\end{equation} 
for unique $\boldsymbol{\upsilon}_j'\in L(T^{d-r}_c)$, $\rho_j\in \mathbb{Z}$,
$\boldsymbol{\upsilon}_j''\in L(T^{r-1}_{\boldsymbol{\nu}^\perp})$.
If $\delta_j:=\tilde{\boldsymbol{\delta}}(\boldsymbol{\upsilon}_j)$, 
$\overline{\Delta}_{\boldsymbol{\nu}}$ is canonically identifiable 
under the natural isomorphism 
$\mathfrak{t}^{d-r+1}_q\cong{\boldsymbol{\nu}^\perp}^0$
with the Delzant polytope 
\begin{equation}
\label{eqn:Deltanu'}
{\boldsymbol{\nu}^\perp}^0\supseteq \overline{\Delta}'_{\boldsymbol{\nu}}:=
\bigcap_{j=1}^l\{\gamma \in {\boldsymbol{\nu}^\perp}^0\,:\,
\gamma (\boldsymbol{\upsilon}_j'+\rho_j\,\widetilde{\boldsymbol{\nu}})\ge \lambda_j+\delta_j\}.
\end{equation}

Let us set
\begin{equation}
\label{eqn:hatupsilonjnur}
\widehat{\boldsymbol{\upsilon}}_j:=
\boldsymbol{\upsilon}_j'-(\lambda_j+\delta_{j})\,\,\tilde{\boldsymbol{\nu}},\qquad
s_j:=\left(\widehat{\boldsymbol{\upsilon}}_j\right)\qquad
(j=1,\ldots,l).
\end{equation}
\begin{equation}
\label{eqn:Delta'nur}
\widehat{\Delta}_{\boldsymbol{\nu}}':=
\bigcap_{j=1}^l\{\ell \in {\boldsymbol{\nu}^\perp}^0\,:\,\ell \big(\widehat{\boldsymbol{\upsilon}}_j\big)\ge -\rho_j\},
\end{equation}

Finally, let $\widehat{\Delta}_{\boldsymbol{\nu}}\subset \mathfrak{t}^{d-r+1}_q$ be the polytope 
corresponding to $\widehat{\Delta}'_{\boldsymbol{\nu}}\subset{\boldsymbol{\nu}^\perp}^0$, and let
$\mathbf{s}_{\boldsymbol{\nu}}:=(s_j)\in \mathbb{N}^l$.

\begin{thm}
\label{thm:case r general}
Under Basic ssumption \ref{asmpt:basic}, suppose in addition that $\Delta+\tilde{\boldsymbol{\delta}}$
and ${\boldsymbol{\nu}^\perp}^0$ are transverse and that the intersection is Delzant.
Then $\big(\widehat{M}_{\boldsymbol{\nu}},\,2\,\widehat{\omega}_{\boldsymbol{\nu}}    \big)$
is the symplectic toric orbifold with associated marked polytope 
$\big(\widehat{\Delta}_{\boldsymbol{\nu}},\,\mathbf{s}_{\boldsymbol{\nu}}\big)$.

\end{thm}

We have an analogue of Corollary \ref{cor:equalcohomology/Q},
linking the cohomology groups of the symplectic reduction 
$\overline{M}_{\boldsymbol{\nu}}$ and of the conic reduction 
$\widehat{M}_{\boldsymbol{\nu}}$. 
By the theory of \cite{kir}, $H^l(\overline{M}_{\boldsymbol{\nu}},\mathbb{Q})$ is tightly related to
the equivariant cohomology of $M$ for the action of $T^{r-1}_{\boldsymbol{\nu}^\perp}$.

\begin{cor}
\label{cor:equal cohomology r general}
Under the hypothesis of Theorem \ref{thm:case r general},
$$
H^l(\overline{M}_{\boldsymbol{\nu}},\mathbb{Q})\cong H^l\big(\widehat{M}_{\boldsymbol{\nu}},\,\mathbb{Q} \big)
\qquad (l=0,1,\ldots).
$$
\end{cor}

\section{The case $r=1$}

\subsection{Preliminaries}
Before embarking on the proof of Theorem \ref{thm:case r=1}, we need to recall some basic
constructions from toric geometry, referring to \cite{del}, \cite{g-toric}
and \cite{g-JDG} for details. We premise a digression on the geometric 
relation between $\Delta$ and $\widehat{\Delta}_{\boldsymbol{\nu}}$.

\subsubsection{The transform of a polytope}

Although not logically necessary, it is suggestive to describe the passage
from $\Delta$ to $\widehat{\Delta}_{\boldsymbol{\nu}}$ in terms of a general
\lq transform\rq\, operation on rational polytopes in a finite-dimensional 
real vector space
with a full-rank lattice $L$, depending on the datum of a decomposition of
$L$ as the product of an oriented rank-$1$ sublattice and a complementary sublattice.

Let $V$ be a $d$-dimensional real vector space, $L\subset V$ a full-rank lattice, $V^\vee$
the dual vector space, and $L^\vee$ the dual lattice. Suppose that $\Delta\subset V^\vee$
is $d$-dimensional rational simple convex polytope (terminology as in \cite{lt}).
This means that there exist primitive $\mathbf{v}_i\in L$ and $\lambda_i\in \mathbb{R}$,
$i=1,\ldots,k$, such that
\begin{equation}
\label{eqn:rational polytope}
\Delta=\bigcap_{j=1}^k\left\{\ell\in V^\vee\,:\,
\ell (\mathbf{v}_j)\ge \lambda_j\right\},
\end{equation}
and that exactly $d$ facets of $\Delta$ meet at each of its vertexes. 
In addition, we shall say that $\Delta$ is
\textit{integral} if $\lambda_j\in \mathbb{Z}$ for every $j$.

Suppose given:
\begin{enumerate}

\item a primitive lattice vector $\mathbf{v}\neq \mathbf{0}\in L$;
\item $\boldsymbol{\delta}\in \mathrm{span}(\mathbf{v})^\vee$ such that
$\boldsymbol{\delta}(\mathbf{v})\in \mathbb{Z}$ and
\begin{equation}
\label{eqn:positive reqt ell}
\ell(\mathbf{v})+\boldsymbol{\delta}(\mathbf{v})>0\qquad \forall\,\ell\in \Delta;
\end{equation}

\item a complementary sublattice $L'\subset L$ to $\mathbb{Z}\cdot \mathbf{v}$, so that setting
$V':=L'\otimes \mathbb{R}$ we have $V=V'\oplus \mathrm{span}(\mathbf{v})$ and 
dually $V^\vee={V'}^\vee\oplus
\mathrm{span}(\mathbf{v}^*)$, where $\mathbf{v}^*\in \mathrm{span}(\mathbf{v})^\vee$ 
is dual to $\mathbf{v}$.

\end{enumerate}

Then we may uniquely extend 
$\boldsymbol{\delta}$ to
$\tilde{\boldsymbol{\delta}}\in L^\vee\cap\mathrm{span}(\mathbf{v}^*)\subseteq V^\vee$ so that
$\tilde{\boldsymbol{\delta}}=\delta\,\mathbf{v}^*$ with $\delta\in \mathbb{Z}$
(a different choice of $\tilde{\boldsymbol{\delta}}$ 
would result in a translation of the transormed polytope).
By (\ref{eqn:positive reqt ell}),
$\Delta+\tilde{\boldsymbol{\delta}}$ lies in the open half-space 
$V^\vee_+\subset V^\vee$
where pairing with $\mathbf{v}$
is positive.

Any $\ell\in V^\vee_+$ can be written uniquely as
$\ell=\ell'+\ell (\mathbf{v})\,\mathbf{v}^*$, where $\ell'\in {V'}^\vee$ and 
$\ell (\mathbf{v})>0$. Let us define an involution $\varrho:V^\vee_+\rightarrow V^\vee_+$
by setting
\begin{equation}
\label{eqn:varrho}
\varrho (\ell):=\frac{1}{\ell (\mathbf{v})}\,\ell'+   \frac{1}{\ell (\mathbf{v})}\,\mathbf{v}^*.
\end{equation}

Let us determine 
$\rho(\Delta +\tilde{\boldsymbol{\delta}})$.
For each $j$, we can write uniquely
$\mathbf{v}_j=\mathbf{v}_j'+\rho_j\,\mathbf{v}$ where $\mathbf{v}_j'\in L'$ and $\rho_j\in \mathbb{Z}$.
Hence $\tilde{\boldsymbol{\delta}}(\mathbf{v}_j)=\delta\,\rho_j$.
We have
\begin{equation}
\label{eqn:rational polytope transl}
\Delta +\tilde{\boldsymbol{\delta}}=\bigcap_{j=1}^k\left\{\ell\in V^\vee\,:\,
\ell (\mathbf{v}_j)=\ell'(\mathbf{v}_j')+\rho_j\,\ell(\mathbf{v})\ge \lambda_j+\delta\,\rho_j\right\},
\end{equation}
Since $\varrho=\varrho^{-1}$, by (\ref{eqn:varrho}) and (\ref{eqn:rational polytope transl}) we have
\begin{eqnarray}
\label{eqn:rational polytope tranformed}
\widehat{\Delta}:=\varrho(\Delta +\tilde{\boldsymbol{\delta}})&=&\bigcap_{j=1}^k\left\{\ell\in V^\vee_+\,:\,
\varrho(\ell) (\mathbf{v}_j)=\frac{1}{\ell(\mathbf{v})}\,\left[
\ell'(\mathbf{v}_j')+\rho_j\right]\ge \lambda_j+\delta\,\rho_j\right\}\nonumber\\
&=&\bigcap_{j=1}^k\left\{\ell\in V^\vee_+\,:\,
\ell'(\mathbf{v}_j')-(\lambda_j+\delta\,\rho_j)\,\ell(\mathbf{v})\ge -\rho_j \right\}.
\end{eqnarray}

Thus $\widehat{\Delta}$ is the convex polytope obtained from $\Delta$
by replacing each primitive normal vector $\mathbf{v}_j=\mathbf{v}_j'+\rho_j\,\mathbf{v}$ with 
the integral vector $\widehat{\mathbf{v}}_j:=
\mathbf{v}_j'-(\lambda_j+\delta\,\rho_j)\,\mathbf{v}$, and each $\lambda_j$ with $-\rho_j$.
Clearly, $\widehat{\Delta}$ is rational; it is not claimed that each $\widehat{\mathbf{v}}_j$ be primitive, hence neither that 
$\hat{\Delta}$ be integral.

Furthemore, (\ref{eqn:rational polytope tranformed}) shows that, if $F_j$ is the facet
of $\Delta$ normal to $\mathbf{v}_j$, then $\widehat{F}_j:=\rho (F_j+\boldsymbol{\delta})$
is the facet of $\widehat{\Delta}$ normal to $\widehat{\mathbf{v}}_j$; this correspondence passes
to intersection of facets, i.e. faces. Thus we have a bijection between the set of faces of each
given dimension of $\Delta$ and $\widehat{\Delta}$, hence in particular between the families of
their respective vertexes.
In particular, 
the vertexes of $\widehat{\Delta}$ are the images
by $\varrho$ of the vertexes of $\Delta$,
and furthermore $\widehat{\Delta}$ is simple since so is $\Delta$.

\subsubsection{The toric line bundle $A$ and its circle bundle}
Let us review 
the construction of the positive toric line bundle $(A,h)$ 
on $M$ from the Delzant polytope $\Delta_{\boldsymbol{\lambda}}$, 
for $\boldsymbol{\lambda}\in \mathbb{Z}^k$, based on 
pairing the Delzant construction of $M$ 
as a symplectic quotient of $\mathbb{C}^k$
with the construction of a polarization on the quotient in \cite{gs-gq}.
Consider the trivial line bundle 
$L:=\mathbb{C}^k\times \mathbb{C}$, and define a Hermitian metric $\kappa$ on $L$ by setting 
$$
\kappa_{\mathbf{z}}\big((\mathbf{z},w),\,(\mathbf{z},v)\big)
:=w\,\overline{v}\,e^{-\|\mathbf{z}\|^2} \quad (\mathbf{z}\in \mathbb{C}^k,\,
w,v\in \mathbb{C}).
$$
The unit circle bundle $Y\subset L^\vee$ (that is, in $L$ with the dual metric)
is then 
\begin{eqnarray}
\label{eqn:unitcircleL}
Y&:=&\left\{ (\mathbf{z},w)\in \mathbb{C}^k\times \mathbb{C}  \,:\,
|w|=e^{-\frac{1}{2}\,\|\mathbf{z}\|^2}    \right\}\nonumber\\
&=&\left\{ \left(\mathbf{z},e^{-\frac{1}{2}\,\|\mathbf{z}\|^2}\,e^{\imath\,\theta}\right)
\,:\,\mathbf{z}\in \mathbb{C}^k,\,
e^{\imath\,\theta}\in S^1 \right\}\cong \mathbb{C}^k\times S^1.
\end{eqnarray}
In the following we shall implicitly identify $Y$ and $\mathbb{C}^k\times S^1$.
In terms of the previous diffeomorphism,
the unique compatible connection 1-form is
\begin{equation}
\label{eqn:betaY}
\beta:=\frac{\imath}{2}\,\sum_{j=1}^k\big[z_j\,\mathrm{d}\overline{z}_j-\overline{z}_j\,\mathrm{d}z_j\big]
+\mathrm{d}\theta .
\end{equation}
Thus, $\beta(\partial_\theta)=1$ and $\ker(\beta)\subset TY$ is the horizontal subspace.

If $f:\mathbb{C}^k\rightarrow \mathbb{C}$ is $\mathcal{C}^\infty$, the corresponding
section $\sigma_f$ of $L$ has pointwise norm
$$
\|\sigma_f(\mathbf{z})\|_\kappa =
|f(\mathbf{z})| \, e^{-\frac{1}{2}\,\|\mathbf{z}\|^2}.
$$
Applying this with $f=1$ we obtain that, letting $\Theta_0$ be the curvature of the unique compatible connection on $L$, 
$$
\Theta_0=\overline{\partial}\,\partial\,\left(-\|\mathbf{z}\|^2\right)
=\sum_{j=1}^k\mathrm{d}z_j\wedge \mathrm{d}\overline{z}_j=
-2\,\imath\,\omega_0,
$$
where $\omega_0=(\imath/2)\,\sum_{j=1}^k\mathrm{d}z_j\wedge \mathrm{d}\overline{z}_j$ is the standard symplectic form on $\mathbb{C}^k$.

Given that $\boldsymbol{\lambda}\in \mathbb{Z}^k$,
the Hamiltonian action $(\Gamma^{\mathbb{C}^k},\Psi_{\boldsymbol{\lambda}})$
(see (\ref{eqn:GammaCk}) and (\ref{eqn:psilambda}))
has the contact CR lift $\Gamma^Y_{\boldsymbol{\lambda}}:T^k\times Y\rightarrow Y$ given by
\begin{equation}
\label{eqn:momentGAMMA_Y}
\Gamma^Y_{\boldsymbol{\lambda}}:\left(e^{\imath\,\boldsymbol{\vartheta}},
(\mathbf{z},w)\right)\mapsto 
\left(\Gamma^{\mathbb{C}^k}(e^{\imath\,\boldsymbol{\vartheta}},\mathbf{z}),  
e^{-\imath\,\langle\boldsymbol{\lambda},\boldsymbol{\vartheta}\rangle}\,w   \right)=
\left(e^{-\imath\,\boldsymbol{\vartheta}}\cdot \mathbf{z},  
e^{-\imath\,\langle\boldsymbol{\lambda},\boldsymbol{\vartheta}\rangle}\,w   \right).
\end{equation}
This is the restriction a similarly defined
metric preserving linearization 
$\Gamma^{L^\vee}_{\boldsymbol{\lambda}}:T^k\times L^\vee\rightarrow L^\vee$;
dually, we also have a linearization
$\Gamma^{L}_{\boldsymbol{\lambda}}:T^k\times L\rightarrow L$.

As in \cite{gs-gq}, we can take the quotient and obtain a positive line bundle
$(A,h)$ on $M=Z_{\Delta}/N$, by setting
\begin{equation}
\label{eqn:positivelb M}
A:=\left.L\right|_{Z_\Delta}/N=(Z_\Delta\times \mathbb{C})/N,
\end{equation}
with associated unit circle bundle $X\subset A^\vee$ given by
\begin{equation}
\label{eqn:Xcircle bundle quot}
X=\left.Y\right|_{Z_\Delta}/N\cong (Z_\Delta\times S^1)/N.
\end{equation}

\subsubsection{The complexification $\widetilde{N}$ and its stable locus}
\label{sctn:complexification}

Besides representing $M$ as a Marsden-Weinstein reduction for the quotient of $N$, 
it is useful to consider 
its parallel description as a GIT quotient for the action of the complexification
$\widetilde{N}$ (\cite{g-toric}, \cite{g-JDG}). 
In the following, for every compact group $T$, $\widetilde{T}$ is the complexification
of $T$.

Every face $F\in \mathcal{F}(\Delta)$ of codimension $c_F$ of $\Delta$
is uniquely an intersection of facets; hence there exists 
a unique increasing multi-index
$I_F:=\{i_1(F),\ldots,i_{c_F}(F)\}\subset \{1,\ldots,k\}^{c_F}$ such that
$F=\bigcap_{j=i_1(F)}^{i_{c_F}(F)}F_j$.
Let us set
\begin{equation}
\label{eqn:OFCk}
\mathbb{O}_F:=\left\{\mathbf{z}=(z_j)\in \mathbb{C}^k\,:\,
z_j=0\,\Leftrightarrow\,j\in I_F\right\}\qquad (F\in \mathcal{F}(\Delta)).
\end{equation}
The following holds:

\begin{enumerate}
\item $\mathbb{O}_F$ is an orbit of $\tilde{T}^k$, and
$\mathbb{O}_F\cong {\mathbb{C}^*}^{k-c_F}$ equivariantly;
\item the stabilizer 
in $T^k$
of every $\mathbf{z}\in \mathbb{O}_F$ is the subtorus
$T^k_F$ with Lie algebra 
$$
\mathfrak{t}^k_F:=\imath\,\mathrm{span}_{\mathbb{R}}\{\mathbf{e}_j\,:\,j\in I_F\},
$$
and the corresponding statement holds for the stabilizer 
in the complexification, $\widetilde{T}^k_F\leqslant \tilde{T}^k$;
\item $\mathbb{C}_\Delta:=\bigcup_{F\in \mathcal{F}(\Delta)}\mathbb{O}_F$ is the open subset
of stable points for the action of $\widetilde{T}^k$ on $\mathbb{C}^k$ with the given linearization or,
equivalently, the $\widetilde{T}^k$-saturation of $Z_\Delta$;
\item $\widetilde{N}$ acts freely and properly on $\mathbb{C}_\Delta$;
\item $M=\mathbb{C}_\Delta/\widetilde{N}$ as a complex manifold;
\item for every $F\in \mathcal{F}(\Delta)$, $M_F^0:=\mathbb{O}_F/\widetilde{N}$ is 
a $\widetilde{T}^d$-orbit and 
a complex submanifold of $M$ of codimension $c_F$; 
\item $M_F^0=\Psi^{-1}(F^0)$, where
$F^0$ is the interior of $F$ 
(recall that $\Psi:M\rightarrow {\mathfrak{t}^d}^\vee$
is the moment map).
\end{enumerate}

We have the following (\cite{del}, \cite{g-toric}, \cite{g-JDG}).

\begin{lem}
\label{lem:stabilizerMF}
For any $F\in \mathcal{F}(\Delta)$, let $T^d_F\leqslant T^d$
be the subtorus with Lie subalgebra
$$
\mathfrak{t}^d_F:=\mathrm{span}_{\mathbb{R}}\{\mathbf{u}_j\,:\,j\in I_J\}.
$$
Then
\begin{enumerate}
\item the isomorphism $T^k/N\cong T^d$ induces an isomorphism
$\rho_F:T^k_F\cong T^d_F$;
\item $T^d_F$ is the stabilizer in $T^k$ of every $m\in M_F^0$.

\item $M_F:=\Psi^{-1}(F)$ is the complex submanifold of fixed points of $T^d_F$;

\item
$M_F^0$ is the dense open subset of $M_F$ of those points whose stabilizer is exactly 
$T^d_F$.
\end{enumerate}
Similar statements hold in the complexifications.

\end{lem}

Let us denote by $P:Z_\Delta\rightarrow M$ and by $\tilde{P}:\mathbb{C}_\Delta\rightarrow M$
the projections. Then $\mathbb{C}_\Delta=\tilde{T}^k\cdot Z_\Delta$ and 
$P=\tilde{P}|_{Z_\Delta}$.

\subsubsection{The lifted action of $T^d$ on $X$}

By passing to the quotient, $\Gamma^{\mathbb{C}^d}$ and 
$\Gamma_{\boldsymbol{\lambda}}^Y$ determine corresponding actions 
$$
\gamma^M:T^d\times M\rightarrow M,\quad
\gamma^X:T^d\times X\rightarrow X;
$$
given that $\Gamma^Y_{\boldsymbol{\lambda}}$ is the contact and CR lift of 
$(\Gamma^{\mathbb{C}^k},\Psi_{\boldsymbol{\lambda}})$,
$\gamma^X$ is the contact and CR lift of $(\gamma^{M},\Psi)$.

Given $m\in M_F^0$, $T^d_F\leqslant T^d$ acts on $X_m=\pi^{-1}(m)\subset X$ by a 
character that we now specify. Let us choose 
$\mathbf{z}\in P^{-1}(m)\subset \mathbb{O}_F\cap Z_\Delta$. Since $N$ acts freely on $Z_\Delta$,
the projection $L|_{Z_\Delta}\rightarrow A$ restricts to an isomorphism 
$L_\mathbf{z}\cong A_m$, which is equivariant with respect to the isomorphism
$\rho_F:T^k_F\cong T^d_F$ in Lemma \ref{lem:stabilizerMF}.

If $\imath\,\boldsymbol{\vartheta}\in \mathfrak{t}^k_F$ and 
$(\mathbf{z},w)\in Y_\mathbf{z}$, then by (\ref{eqn:momentGAMMA_Y})
\begin{equation}
\label{eqn:momentGAMMA_Ygiù}
\Gamma^Y_{\boldsymbol{\lambda}}\left(e^{\imath\,\boldsymbol{\vartheta}},
(\mathbf{z},w)\right)=
\left(\mathbf{z},  
e^{-\imath\,\langle\boldsymbol{\lambda},\boldsymbol{\vartheta}\rangle}\,w   \right).
\end{equation}
Since $\rho_F(e^{\imath\,\mathbf{e}_j})=e^{\boldsymbol{\upsilon}_j}$ for 
$j\in I_F$, for $x=(m,\ell)\in X_m$ we have
\begin{eqnarray}
\label{eqn:gammaX}
\gamma^X\left(e^{\sum_{j\in I_F}\vartheta_j\,\boldsymbol{\upsilon}_j},
(m,\ell)\right)&=&
\left(m,  
e^{-\imath\,\langle\boldsymbol{\lambda},\boldsymbol{\vartheta}\rangle}\,\ell  \right)\nonumber\\
&=&e^{-\imath\,\langle\boldsymbol{\lambda},\boldsymbol{\vartheta}\rangle}\,x
=\chi_{F}\left(e^{\sum_{j\in I_F}\vartheta_j\,\boldsymbol{\upsilon}_j}\right)^{-1}\,x,
\end{eqnarray}
where 
\begin{equation}
\label{eqn:characterChiF}
\chi_F:e^{\sum_{j\in I_F}\vartheta_j\,\boldsymbol{\upsilon}_j}\in 
T^d_F\mapsto e^{\imath\,\sum_{j\in I_F}\vartheta_j\,\lambda_j}\in S^1.
\end{equation}

We can reformulate this as follows.

\begin{lem}
\label{lem:characterFgamma}
Suppose that $F$ is a face of $\Delta$, that $m\in F^0$, and that
$x\in X_m$. Then for every $t\in T^d_m$ we have
$
\gamma^X_t(x)=\rho^X_{\chi_F(t)}(x)
$.

\end{lem}

\subsubsection{The lifted action of $T^1$ on $X$}

We have remarked that $\mu^M:T^1\times M\rightarrow M$ is the restriction
of $\gamma^M:T^d\times M\rightarrow M$ to $T^1$ while, on the other hand,
$\mu^X:T^1\times X\rightarrow X$ won't be the restriction
of $\gamma^X:T^d\times X\rightarrow X$ to $T^1$, unless 
$\boldsymbol{\delta}=\mathbf{0}\in {\mathfrak{t}^1}^\vee$ in (\ref{eqn:Trpsi}).
Since however both $\mu^X$ and the restriction of $\gamma^X$ to $T^1$ lift
$\mu^M$, there is a character $\chi=\chi_{\boldsymbol{\delta}}:T^1\rightarrow S^1$ such that
\begin{equation}
\label{eqn:char mediatra}
\mu^X_h(x)=\gamma^X_h\circ \rho^X_{\chi(h)}(x)\quad (x\in X,\,h\in T^1).
\end{equation}
Let us make $\chi$ explicit. 
Since $\widetilde{\boldsymbol{\nu}}$ is primitive, the map $e^{\imath\,\vartheta}\in S^1\mapsto
e^{\vartheta\,\widetilde{\boldsymbol{\upsilon}}}\in T^1$ is an isomorphism of Lie groups. 

\begin{lem}
\label{lem:deltachi}
We have
$\delta\in \mathbb{Z}$ and $\chi(e^{\vartheta\,\widetilde{\boldsymbol{\nu}}})=e^{\imath\,\delta\,\vartheta}$.
\end{lem}

\begin{proof}
Recall that, by choice of $\widetilde{\boldsymbol{\nu}}$, 
$\boldsymbol{\nu}=\widetilde{\boldsymbol{\nu}}^*\in {\mathfrak{t}^1}^\vee$
is the dual basis to $\widetilde{\boldsymbol{\nu}}$, and so
$\boldsymbol{\delta}=\delta\,\boldsymbol{\nu}$, where $\delta = \boldsymbol{\delta}(\widetilde{\boldsymbol{\nu}})$.
Let us write $\widetilde{\boldsymbol{\nu}}_X^{\Phi}$ and $\widetilde{\boldsymbol{\nu}}_X^{\Psi}$ 
for the vector field on $X$ induced by 
$\widetilde{\boldsymbol{\nu}}$ under $\mu^X$ and $\gamma^X$, respectively. 
In view of (\ref{eqn:contact lift}), we obtain
\begin{eqnarray}
\label{eqn:nuXPhiPsi}
\widetilde{\boldsymbol{\nu}}_X^{\Phi}&=&\widetilde{\boldsymbol{\nu}}_M^\sharp
-\langle\Phi,\widetilde{\boldsymbol{\nu}}\rangle\,\partial_\theta
\nonumber\\
&=&\widetilde{\boldsymbol{\nu}}_M^\sharp-\langle\Psi,\widetilde{\boldsymbol{\nu}}\rangle\,\partial_\theta
-\delta\,\partial_\theta=\widetilde{\boldsymbol{\nu}}_X^{\Psi}-\delta\,\partial_\theta.
\end{eqnarray}
Hence for every $x\in X$ and $e^{\vartheta\,\boldsymbol{\nu}}\in T^1$
\begin{equation}
\label{eqn:chidelta}
\mu^X_{e^{\vartheta\,\widetilde{\boldsymbol{\nu}}}}(x)=
e^{-\imath\,\delta\,\vartheta}\gamma^X_{e^{\vartheta\,\widetilde{\boldsymbol{\nu}}}}(x)=\gamma^X_{e^{\vartheta\,\widetilde{\boldsymbol{\nu}}}}\circ 
\rho^X_{e^{\imath\,\delta\,\vartheta}}(x).
\end{equation}
Given that $\widetilde{\boldsymbol{\nu}}\in L(T^1)$, (\ref{eqn:chidelta})
implies $\rho^X_{e^{2\pi\,\imath\,\delta}}=\mathrm{id}_X$. Since $\rho^X$ is free, 
this implies $\delta\in \mathbb{Z}$. Since (\ref{eqn:chidelta}) holds for any $\vartheta$,
the second claim follows as well.

\end{proof}

\subsubsection{$\widehat{M}_{\boldsymbol{\nu}}$ and its K\"{a}hler structure}

By assumption, $\mu^X:T^1\times X\rightarrow X$ lifts $\mu^M:T^1\times M\rightarrow M$;
since $\Phi^{\tilde{\boldsymbol{\nu}}}>0$, $\mu^X$ is locally free by (\ref{eqn:contact lift}), 
and given that
$\mu^M$ is holomorphic $\mu^X$ preserves the CR structure
of $X$.  
Hence the quotient $\widehat{M}_{\boldsymbol{\nu}}:=X/\mu^X$ is a 
$d$-dimensional complex orbifold with complex structure
$\widehat{J}_{\boldsymbol{\nu}}$ (\cite{pao-loa}, \cite{pao-JGP}).
Furthermore, $\mu^X$ is effective, hence generically free;
therefore the projection $\widehat{\pi}_{\boldsymbol{\nu}}
:X\rightarrow \widehat{M}_{\boldsymbol{\nu}}$ is a principal $V$-bundle with
structure group $T^1$. 

We shall now see that
$(\widehat{M}_{\boldsymbol{\nu}},\,\widehat{J}_{\boldsymbol{\nu}})$ carries a 
K\"{a}hler structure $\widehat{\omega}_{\boldsymbol{\nu}}$, 
naturally induced from $\omega$. Aside from slight changes in notation,
the discussion is close to the ones in \S 2 of \cite{pao-loa}
and \S 5.3 of \cite{pao-JGP}, so we'll be rather sketchy.
To lighten notation, we shall adopt the following conventions.

\begin{enumerate}
\item if $(m,x,\widehat{m})\in M\times X\times \widehat{M}_{\boldsymbol{\nu}}$,
we shall write $m\leftarrow x\rightarrow \widehat{m}$ to mean 
$\pi(x)=m$ and $\widehat{\pi}_{\boldsymbol{\nu}}(x)=\widehat{m}$;
\item if $(m,\widehat{m})\in M\times \widehat{M}_{\boldsymbol{\nu}}$,
we shall write $m\sim \widehat{m}$ to mean that $m\leftarrow x\rightarrow \widehat{m}$
for some $x\in X$;
\item if $U\subseteq M$, we shall set 
$\widehat{U}:=\widehat{\pi}_{\boldsymbol{\nu}}\left( \pi^{-1}(U) \right)$;
\item we shall generally omit symbols of pull-backs for
functions, and denote by the same symbol a function $f:M\rightarrow \mathbb{C}$
and its pull-back $\pi^*(f):X\rightarrow \mathbb{C}$;
\item similarly, if
$f$ is invariant and hence $\pi^*(f)$ descends to $\widehat{M}_{\boldsymbol{\nu}}$, we shall also 
denote by
$f:\widehat{M}_{\boldsymbol{\nu}}\rightarrow \mathbb{C}$ the descended function.
\end{enumerate}

Let the invariant differential $1$-form $\widehat{\alpha}_{\boldsymbol{\nu}}\in \Omega^1(X)$ be
defined by
\begin{equation}
\label{eqn:hatalphanu}
\widehat{\alpha}_{\boldsymbol{\nu}}:=\frac{1}{\Phi^{\tilde{\boldsymbol{\nu}}}}\,\alpha.
\end{equation}
Then $\iota(\tilde{\boldsymbol{\nu}}_X)\,\widehat{\alpha}_{\boldsymbol{\nu}}=-1$, hence 
$\widehat{\alpha}_{\boldsymbol{\nu}}$ is a connection $1$-form for $\widehat{\pi}_{\boldsymbol{\nu}}$.
Hence there is a unique orbifold 2-form $\widehat{\omega}_{\boldsymbol{\nu}}$ on
$\widehat{M}_{\boldsymbol{\nu}}$ such that 
$\mathrm{d}\widehat{\alpha}_{\boldsymbol{\nu}}=2\,\widehat{\pi}_{\boldsymbol{\nu}}^*(\widehat{\omega}_{\boldsymbol{\nu}})$.
Since by (\ref{eqn:hatalphanu}) $\ker(\alpha)=\ker(\widehat{\alpha}_{\boldsymbol{\nu}})$, $\pi$ 
and $\widehat{\pi}_{\boldsymbol{\nu}}$ share the same horizontal bundle, i.e.,
$\mathcal{H}_x(\pi)=\mathcal{H}_x(\widehat{\pi}_{\boldsymbol{\nu}})$ for every $x\in X$. 
On the other hand,
since $\Phi^{\tilde{\boldsymbol{\nu}}}>0$ by (\ref{eqn:contact lift}) we have 
$\tilde{\boldsymbol{\nu}}_X(x)\not\in \mathcal{H}(\pi)_x$ at every $x\in X$. 
Hence we can split the tangent bundle $TX$ of $X$ in the two alternative ways:
\begin{equation}
\label{eqn:two splittings}
TX=\mathcal{H}(\pi)\oplus \mathrm{span}(\partial_\theta)=
\mathcal{H}_x(\widehat{\pi}_{\boldsymbol{\nu}})\oplus \mathrm{span}(\tilde{\boldsymbol{\nu}}_X).
\end{equation}
In particular, if
$m\leftarrow x\rightarrow \widehat{m}$ then there are complex linear isomorphisms 
\begin{equation}
\label{eqn:uniformizing unitaryiso}
T_mM\cong \mathcal{H}_x(\pi)\cong T_{\widehat{m}}(\widehat{M}_{\boldsymbol{\nu}}),
\end{equation}
where the latter denotes the uniformizing tangent space of $\widehat{M}_{\boldsymbol{\nu}}$
at $\tilde{m}$.
Since 
$$
2\,\widehat{\pi}_{\boldsymbol{\nu}}^*(\widehat{\omega}_{\boldsymbol{\nu}})=
\mathrm{d}\widehat{\alpha}_{\boldsymbol{\nu}}=\frac{1}{\Phi^{\tilde{\boldsymbol{\nu}}}}
\,2\,\pi^*(\omega)
-\frac{1}{{\Phi^{\tilde{\boldsymbol{\nu}}}}^2}\,\mathrm{d}{\Phi^{\tilde{\boldsymbol{\nu}}}}
\wedge \alpha,
$$
the triple $\left(T_{\widehat{m}}(\widehat{M}_{\boldsymbol{\nu}}), \widehat{J}_{\boldsymbol{\nu}},
\widehat{\omega}_{\boldsymbol{\nu}}\right)$ is isomorphic to 
$(T_mM,J_m,\,\omega_m/\Phi^{\tilde{\boldsymbol{\nu}}})$,
so that
$\widehat{\omega}_{\boldsymbol{\nu}}$ is a K\"{a}hler form on $\widehat{M}_{\boldsymbol{\nu}}$.

\subsubsection{Horizontal and contact lifts with respect to $\widehat{\pi}_{\boldsymbol{\nu}}$}

Since $\tilde{\boldsymbol{\nu}}$ is primitive, the map
$e^{\imath\,\vartheta}\in S^1\mapsto 
e^{\imath\,\vartheta\,\tilde{\boldsymbol{\nu}}} \in T^1$ is an isomorphism of Lie groups.
Composing the latter with the effective action $\mu^X$, we obtain an effective action of
$S^1$ on $X$, which is free on a dense invariant subset. Therefore, there exists a dense
(and smooth) open subset $\widehat{M}_{\boldsymbol{\nu}}'\subseteq \widehat{M}_{\boldsymbol{\nu}}$
over which $\widehat{\pi}_{\boldsymbol{\nu}}$ restricts to principal $S^1$-bundle.
Let us set $X':=\widehat{\pi}_{\boldsymbol{\nu}}^{-1}(\widehat{M}_{\boldsymbol{\nu}}')$.

Given a smooth orbifold vector field $\boldsymbol{\upsilon}$ on $\widehat{M}_{\boldsymbol{\nu}}$, we
shall say that a (smooth) vector field on $X$ is the horizontal lift of 
$\boldsymbol{\upsilon}$ (with respect to $\widehat{\pi}_{\boldsymbol{\nu}}$) 
if it is horizontal (i.e., tangent to 
$\mathcal{H}(\pi)=\mathcal{H}(\widehat{\pi}_{\boldsymbol{\nu}})$) and 
$\widehat{\pi}_{\boldsymbol{\nu}}$-related to $\boldsymbol{\upsilon}$ over 
$\widehat{M}_{\boldsymbol{\nu}}'$. 

\begin{prop}
\label{prop:horizontal lift orbifold}
Any smooth orbifold vector field $\boldsymbol{\upsilon}$ on $\widehat{M}_{\boldsymbol{\nu}}$
has a unique horizontal lift to $X$.
\end{prop}

We shall denote the horizontal lift in Proposition (\ref{prop:horizontal lift orbifold})
by $\boldsymbol{\upsilon}^\flat$.

\begin{proof}
Any two horizontal lifts of $\boldsymbol{\upsilon}$
clearly coincide on
$X'$, hence everywhere
in $X$.
As to existence, obviously the horizontal lift exists over the smooth locus (i.e. on $X'$), 
so the point is to see that
it has a smooth extension over the singular locus. 

Suppose $\widehat{m}=\widehat{\pi}_{\boldsymbol{\nu}}(x)\in \widehat{M}_{\boldsymbol{\nu}}$,
and let $F_1\subset X$ be a slice for $\mu^X$ through $x$. Thus $F_1$ uniformizes an open neighbourhood
of $\widehat{m}$, and $\boldsymbol{\upsilon}$ corresponds to a vector field
$\mathbf{v}_1$ on $F_1$, invariant under the action of the stabilizer subgroup $T^1_x$ of $x$ in
$T^1$. Furthermore, a suitable invariant tubular neighbourhood $U_1\subseteq X$ of the $T^1$-orbit of
$x$ is equivariantly diffeomorphic to $T^1\times _{T_x^1}F_1$.
Hence we can push forward $\mathbf{v}_1$ (or, more precisely, $(\mathbf{0},\mathbf{v}_1)$) 
under the local diffeomorphism $T^1_x\times F_1\rightarrow U_1$, and obtain a smooth
vector field $\boldsymbol{\upsilon}_1'$ on $U_1$ which is $\widehat{\pi}_{\boldsymbol{\nu}}$-related
to $\boldsymbol{\upsilon}$ on $U_1\cap X'$. Let $\boldsymbol{\upsilon}_1$ denote the horizontal
component of $\boldsymbol{\upsilon}_1'$ with respect to $\widehat{\pi}_{\boldsymbol{\nu}}$,
that is, its projection on 
$\mathcal{H}_x(\widehat{\pi}_{\boldsymbol{\nu}})$ along
$\mathrm{span}(\tilde{\boldsymbol{\nu}}_X)$ in (\ref{eqn:two splittings}).
Then $\boldsymbol{\upsilon}_1$ is a smooth vector field on $U_1$, horizontal and
$\widehat{\pi}_{\boldsymbol{\nu}}$-related
to $\boldsymbol{\upsilon}$ on $U_1\cap X'$.

Another such vector field $\boldsymbol{\upsilon}_2$ similarly constructed on an invariant
open set $U_2$ will necessarily coincide with $\boldsymbol{\upsilon}_1$ on 
$U_1\cap U_2\cap X'$, whence on all of $U_1\cap U_2$ if the latter is non-empty.
Hence by glueing these local constructions we obtain the desired lift.
\end{proof}

Suppose that $f$ is a $\mathcal{C}^\infty$ real function on $\widehat{M}_{\boldsymbol{\nu}}$,
and let $\boldsymbol{\upsilon}_f$ be its Hamiltonian orbifold vector field with respect to $2\,\widehat{\omega}_{\boldsymbol{\nu}}$.
Let us define 
\begin{equation}
\label{eqn:contact upsilon f}
\boldsymbol{\upsilon}_f^c:=\boldsymbol{\upsilon}_f^\flat+f\,\tilde{\boldsymbol{\nu}}_X.
\end{equation}

\begin{prop}
\label{prop:contact lift hat}
$\boldsymbol{\upsilon}_f^c$ is a contact vector field on $(X,\hat{\alpha}_{\boldsymbol{\nu}})$.
If in addition the flow of $\boldsymbol{\upsilon}_f$ is holomorphic on 
$(\widehat{M}_{\boldsymbol{\nu}},\,J_{\boldsymbol{\nu}})$, then the flow of 
$\boldsymbol{\upsilon}_f^c$ preserves the CR structure of $X$.
\end{prop}

\begin{proof}
We have (writing $f$ for $\hat{\pi}^*_{\boldsymbol{\nu}}(f)$)
$$
\iota(\boldsymbol{\upsilon}_f^c)\,\mathrm{d}\hat{\alpha}_{\boldsymbol{\nu}}=
\iota(\boldsymbol{\upsilon}_f^c)\,2\,\hat{\pi}^*_{\boldsymbol{\nu}}\left( \hat{\omega}^*_{\boldsymbol{\nu}}  \right)=\mathrm{d}f,\quad
\mathrm{d}\left( \iota(\boldsymbol{\upsilon}_f^c)\, \hat{\alpha}_{\boldsymbol{\nu}} \right)=
-\mathrm{d}f.
$$
Hence $L_{\boldsymbol{\upsilon}_f^c}\hat{\alpha}_{\boldsymbol{\nu}}=
\iota(\boldsymbol{\upsilon}_f^c)\,\mathrm{d}\hat{\alpha}_{\boldsymbol{\nu}}+
\mathrm{d}\left( \iota(\boldsymbol{\upsilon}_f^c)\, \hat{\alpha}_{\boldsymbol{\nu}} \right)=0$.
This proves the first statement. 

On the other hand, the flow of $\boldsymbol{\upsilon}_f^c$
preserves the horizontal tangent bundle and covers a holomorphic flow on 
$\widehat{M}_{\boldsymbol{\nu}}$; the second statement then 
follows in view of the unitary isomorphisms
(\ref{eqn:uniformizing unitaryiso}).
\end{proof}

By the same principle, we can consider lifts of Hamiltonian actions for 
$\widehat{\pi}_{\boldsymbol{\nu}}$ just as one does for $\pi$.
Suppose given a holomorphic and Hamiltonian action 
$\varsigma^{\widehat{M}_{\boldsymbol{\nu}}}$
of a compact and connected Lie group $G$ on 
$(\widehat{M}_{\boldsymbol{\nu}},\,2\,\widehat{\omega}_{\boldsymbol{\nu}})$ (in the orbifold sense, 
see \cite{lt}), with moment map $\Lambda:\widehat{M}_{\boldsymbol{\nu}}\rightarrow \mathfrak{g}^\vee$. 
Thus any $\boldsymbol{\xi}\in \mathfrak{g}$ determines 
an induced Hamiltonian (orbifold) vector field $\boldsymbol{\xi}_{\widehat{M}_{\boldsymbol{\nu}}}$
on $\widehat{M}_{\boldsymbol{\nu}}$. Applying (\ref{eqn:contact upsilon f})
with $\boldsymbol{\upsilon}=\boldsymbol{\xi}_{\widehat{M}_{\boldsymbol{\nu}}}$,
thus setting $\boldsymbol{\xi}_{X}:=\boldsymbol{\xi}_{\widehat{M}_{\boldsymbol{\nu}}}^c$, 
we associate a contact and CR vector field
on $X$ to each $\boldsymbol{\xi}\in \mathfrak{g}$.
A standard argument shows that this assignment defines an infinitesimal contact
and CR action of $\mathfrak{g}$ on $X$. If this infinitesimal action is the differential
of a Lie group action $\varsigma^{X}$ of $G$ on $X$, we shall call the latter the (contact and
CR) lift of $(\varsigma^{\widehat{M}_{\boldsymbol{\nu}}}, \Lambda)$.

When $G$ acts  on both $M$ and $\widehat{M}_{\boldsymbol{\nu}}$, we have in principle
two lifts in the picture and two different meanings for $\boldsymbol{\xi}_X$. We will
clarify this point in the following section.

\subsubsection{Transfering Hamiltonian actions from $M$ to $\widehat{M}_{\boldsymbol{\nu}}$}

Suppose that $G$ is a connected compact Lie group and let
$\Xi^M:G\times M\rightarrow M$ be a holomorphic and Hamiltonian action, with
moment map $\Upsilon:M\rightarrow\mathfrak{g}^\vee$. Assume the following:
\begin{enumerate}
\item $(\Xi^M,\Upsilon)$
lifts to the contact CR action $\Xi^X:G\times X\rightarrow X$;
\item $\Xi^M$ and $\mu^M$ commute.
\end{enumerate}
Then one can see that $\Xi^X$ commutes with $\mu^X$; therefore
$\Xi^X$ descends to an action
$\Xi^{\widehat{M}_{\boldsymbol{\nu}}}:G\times \widehat{M}_{\boldsymbol{\nu}}\rightarrow
\widehat{M}_{\boldsymbol{\nu}}$.

\begin{prop}
\label{prop:twisted lift hat}
Under the previous assumptions, 
$\Xi^{\widehat{M}_{\boldsymbol{\nu}}}$ is holomorphic and Hamiltonian on the
K\"{a}hler orbifold $(\widehat{M}_{\boldsymbol{\nu}},\,
2\,\widehat{\omega}_{\boldsymbol{\nu}},\,\widehat{J}_{\boldsymbol{\nu}})$,
with moment map 
$$
\hat{\Upsilon}_{\boldsymbol{\nu}}:=\frac{1}{\Phi^{\tilde{\boldsymbol{\nu}}}}\,\Upsilon.
$$
Furthermore, $\Xi^X$ is also the contact and CR lift of 
$\big(\Xi^{\widehat{M}_{\boldsymbol{\nu}}},\,\hat{\Upsilon}_{\boldsymbol{\nu}}\big)$.
\end{prop}

\begin{proof}
Given that $\Xi^M$ commutes with $\mu^M$, it preserves $\Phi^{\tilde{\boldsymbol{\nu}}}$.
Since $\Xi^X$ preserves $\alpha$ and $\Phi^{\tilde{\boldsymbol{\nu}}}$, it generates a
flow of contactomorphisms for $\widehat{\alpha}_{\boldsymbol{\nu}}$.
Therefore, the flow of $\Xi^X$ preserves 
$\widehat{\pi}_{\boldsymbol{\nu}}^*(\widehat{\omega}_{\boldsymbol{\nu}})
=\mathrm{d}\widehat{\alpha}_{\boldsymbol{\nu}}/2$.
Since $\Xi^X$ lifts $\Xi^{\widehat{M}_{\boldsymbol{\nu}}}$ by
$\widehat{\pi}_{\boldsymbol{\nu}}$, we conclude that
$\Xi^{\widehat{M}_{\boldsymbol{\nu}}}$ is a symplectic vector field for 
$\widehat{\omega}_{\boldsymbol{\nu}}$.

Since $\Upsilon:M\rightarrow \mathfrak{g}^\vee$ is $G$-equivariant 
by assumption and $\Phi^{\tilde{\boldsymbol{\nu}}}$
is $G$-invariant because $\Xi^M$ and $\mu^M$ commute, 
$\hat{\Upsilon}_{\boldsymbol{\nu}}:\widehat{M}_{\boldsymbol{\nu}}\rightarrow
\mathfrak{g}^\vee$ is $G$-equivariant.
Thus it suffices to prove that 
$(\Xi^{\widehat{M}_{\boldsymbol{\nu}}},\hat{\Upsilon}_{\boldsymbol{\nu}})$ is weakly Hamiltonian.

Suppose $\hat{m}=\widehat{\pi}_{\boldsymbol{\nu}}(x)\in \widehat{M}_{\boldsymbol{\nu}}$.
Choose a slice $F\subset X$ at $x$ for $\mu^X$, and view it as the uniformizing
open set of an open neighbourhood of $\hat{m}$ in $\widehat{M}_{\boldsymbol{\nu}}$. 
We obtain a local action $\Xi^F$ of $G$ on $F$ as follows. 
For any $y$ in a neighbourhood $F'\subseteq F$ of $x$ 
and $g$ in a neighbourhood $G'\subset G$
of the identity $e_G$,
there exists a unique $s(g,y)\in T^1$ such that 
$\mu^X_{s(g,y)}\circ \Xi^X_g(y)\in F$. Let us set
\begin{equation}
\label{eqn:localactionF}
\Xi^F:(g,y)\in G'\times F'\mapsto\mu^X_{s(g,y)}\circ \Xi^X_g(y)\in F.
\end{equation}
If $g_1,\,g_2\in G'$ are sufficiently close to the identity,
\begin{eqnarray*}
\Xi^F_{g_1}\circ\Xi^F_{g_2}(y)&=&\Xi^F_{g_1}\left( \mu^X_{s(g_2,y)}\circ \Xi^X_{g_2}(y)   \right)\\
&=& \mu^X_{s(g_1,\Xi^F_{g_2}(y))}\circ\Xi^X_{g_1}\circ \mu^X_{s(g_2,y)}\circ \Xi^X_{g_2}(y)\\
&=&\mu^X_{s(g_1,\Xi^F_{g_2}(y))}\circ \mu^X_{s(g_2,y)}\circ\Xi^X_{g_1}\circ \Xi^X_{g_2}(y)\\
&=&\mu^X_{s(g_1\,g_2,y))}\circ \Xi^X_{g_1\,g_2}(y)=\Xi^F_{g_1\,g_2}(y).
\end{eqnarray*}

Given $\boldsymbol{\xi}\in \mathfrak{g}$, the induced vector field $\boldsymbol{\xi}_F$ 
on $F$ may be computed by considering the restricted local action of the $1$-parameter subgroup
$\tau\mapsto e^{\tau\,\boldsymbol{\xi}}\in G$, hence by differentiating at $\tau=0$ the
path $\Xi^F_{e^{\tau\,\boldsymbol{\xi}}}(y)=\mu^X_{s(e^{\tau\,\boldsymbol{\xi}},y)}
\circ \mu^X_{e^{\tau\,\boldsymbol{\xi}}}(y)$. We conclude the following.

\begin{lem}
\label{lem:xiFy}
There exists a $\mathcal{C}^\infty$-function
$\sigma:\mathfrak{g}\times F\rightarrow \mathbb{R}$ such that 
for any $\boldsymbol{\xi}\in \mathfrak{g}$ and $y\in F$
$$
\boldsymbol{\xi}_F(y)=\sigma(\boldsymbol{\xi},y)\,\tilde{\boldsymbol{\nu}}_X(y)+
\boldsymbol{\xi}_X(y).
$$
\end{lem}
Here $\boldsymbol{\xi}_X$ is as in (\ref{eqn:contact lift}), with $\Upsilon$ in place
of $\Phi$.

At any $y\in F$, we have a direct sum decomposition
$T_yX=T_yF\oplus \mathrm{span}\big(\tilde{\boldsymbol{\nu}}_X(y)  \big)$.
Thus Lemma \ref{lem:xiFy} may be reformulated as follows.

\begin{cor}
For any $y\in F$ and 
$\boldsymbol{\xi}\in \mathfrak{g}$,
$\boldsymbol{\xi}_F(y)$ is the projection of $\boldsymbol{\xi}_X(y)$
on $T_yF$ along $\mathrm{span}\big(\tilde{\boldsymbol{\nu}}_X(y)  \big)$.
\end{cor}

By the commutativity of $\mu^X$ and $\Xi^X$, the stabilizer subgroup of $x$
in $T^1$ acts on $F$ preserving the previous direct sum of vector bundles on $F$.
It follows that $\boldsymbol{\xi}_F$ is an invariant vector field on $F$, and the
collection of all such is the induced vector field 
$\boldsymbol{\xi}_{\widehat{M}_{\boldsymbol{\nu}}}$ on $\widehat{M}_{\boldsymbol{\nu}}$.

Letting $\jmath:F\hookrightarrow X$ be the inclusion, let us set
$\alpha_F:=\jmath^*(\widehat{\alpha}_{\boldsymbol{\nu}})$ and $\omega_F:=\mathrm{d}\alpha_F/2$.
The collection of all pairs $(F,\omega_F)$ represents 
$\widehat{\omega}_{\boldsymbol{\nu}}$.

If $y\in F$ as above, $\iota\big(\boldsymbol{\xi}_F(y)  \big)\,(\mathrm{d}\alpha_F)_y$
is the restriction to $T_yF\subset T_yX$ of 
$\iota\big(\boldsymbol{\xi}_F(y)  \big)\,(\mathrm{d}\widehat{\alpha}_{\boldsymbol{\nu}})_y$.
On the other hand, by Lemma \ref{lem:xiFy} we have
\begin{eqnarray*}
\iota\big(\boldsymbol{\xi}_F(y)  \big)\,(\mathrm{d}\widehat{\alpha}_{\boldsymbol{\nu}})_y
&=&\iota\Big( \boldsymbol{\xi}_X(y)
+\sigma(\boldsymbol{\xi},y)\,\tilde{\boldsymbol{\nu}}_X(y)\Big)\,
(\mathrm{d}\widehat{\alpha}_{\boldsymbol{\nu}})_y\\
&=&\left[\iota( \boldsymbol{\xi}_X)\,
(\mathrm{d}\widehat{\alpha}_{\boldsymbol{\nu}})\right]_y
+\sigma(\boldsymbol{\xi},y)\,\left[\iota(\tilde{\boldsymbol{\nu}}_X)\,
(\mathrm{d}\widehat{\alpha}_{\boldsymbol{\nu}})\right]_y\\
&=&-\left[\mathrm{d}\left( \iota( \boldsymbol{\xi}_X)\, \widehat{\alpha}_{\boldsymbol{\nu}})  \right)\right]_y-\sigma(\boldsymbol{\xi},y)\,\left[\mathrm{d}
(\iota(\tilde{\boldsymbol{\nu}}_X)\,\widehat{\alpha}_{\boldsymbol{\nu}})\right]_y\\
&=&\mathrm{d}_y\left( \frac{\Upsilon^{\boldsymbol{\xi}}}{\Phi^{\tilde{\boldsymbol{\nu}}}} \right)
+\sigma(\boldsymbol{\xi},y)\,\mathrm{d}_y(1)\\
&=&\mathrm{d}_y\left( \frac{\Upsilon^{\boldsymbol{\xi}}}{\Phi^{\tilde{\boldsymbol{\nu}}}} \right)
=\mathrm{d}_y\left(\widehat{\Upsilon}^{\boldsymbol{\xi}} \right).
\end{eqnarray*}
We have used that both $\boldsymbol{\xi}_X$ and $\tilde{\boldsymbol{\nu}}_X$ are contact vector
fields for $\widehat{\alpha}_{\boldsymbol{\nu}}$.

To prove the last statement of Proposition \ref{prop:twisted lift hat},
we need to verify that $\boldsymbol{\xi}_X=\boldsymbol{\xi}_{\widehat{M}_{\boldsymbol{\nu}}}^c$
for every $\boldsymbol{\xi}\in \mathfrak{g}$. Since both 
$\boldsymbol{\xi}_X$ and $\boldsymbol{\xi}_{\widehat{M}_{\boldsymbol{\nu}}}^c$
lift $\boldsymbol{\xi}_X=\boldsymbol{\xi}_{\widehat{M}_{\boldsymbol{\nu}}}^c$ under
$\widehat{\pi}_{\boldsymbol{\nu}}$, it suffices to show that
the coefficient of $\boldsymbol{\xi}_X$ along $\tilde{\boldsymbol{\nu}}_X$
is $\widehat{\Upsilon}^{\boldsymbol{\xi}}$. 
Therefore, the equality
\begin{equation}
\label{eqn:twisted lift doppio senso}
\boldsymbol{\xi}_X=\left(\boldsymbol{\xi}_M^\sharp-\widehat{\Upsilon}^{\boldsymbol{\xi}}\,
\tilde{\boldsymbol{\nu}}_M^\sharp\right)+ \widehat{\Upsilon}^{\boldsymbol{\xi}}\,
\tilde{\boldsymbol{\nu}}_X,
\end{equation}
implies that $\boldsymbol{\xi}_M^\sharp-\widehat{\Upsilon}^{\boldsymbol{\xi}}\,
\tilde{\boldsymbol{\nu}}_M^\sharp$ is the horizontal lift (with respect to
$\widehat{\pi}_{\boldsymbol{\nu}}$) of $\boldsymbol{\xi}_{\widehat{M}_{\boldsymbol{\nu}}}$,
and that $\boldsymbol{\xi}_X=\boldsymbol{\xi}_{\widehat{M}_{\boldsymbol{\nu}}}^c$.
\end{proof}

\subsubsection{The torus $\widehat{T}^d$ and its action on $\widehat{M}_{\boldsymbol{\nu}}$}

As in the Introduction, let $T^{d-1}_c\leqslant T^d$ be a complementary subtorus to $T^1$.
Let us define a new torus
\begin{equation}
\label{eqn:hatTd}
\widehat{T}^d:=T^{d-1}_c\times S^1.
\end{equation}
Since $\rho^X$ (the action of $S^1$ on $X$ with generator $-\partial_\theta$)
and $\gamma^X$ (the contact CR action of $T^d$ on $X$)
commute, the restriction of $\gamma^X$ to $T^{d-1}_c$ and $\rho^X$
may be combined to yield a new action 
\begin{equation}
\label{eqn:gammaXhat}
\beta^X:\widehat{T}^d\times X\rightarrow X.
\end{equation}
Since furthermore $\beta^X$ commutes with $\mu^X:T^1\times X\rightarrow X$,
it descends to an action 
\begin{equation}
\label{eqn:beta descended}
\beta^{\widehat{M}_{\boldsymbol{\nu}}}:\widehat{T}^d\times \widehat{M}_{\boldsymbol{\nu}}
\rightarrow \widehat{M}_{\boldsymbol{\nu}}.
\end{equation}

In fact, $(\ref{eqn:gammaXhat})$ is the contact and CR lift of a Hamiltonian action $\beta^M$
of $\widehat{T}^d$ on $(M,2\,\omega)$. 
Given the decomposition $\mathfrak{t}^d=\mathfrak{t}^{d-1}_c\oplus \mathrm{span}(\tilde{\boldsymbol{\nu}})$
the moment map $\Psi$ of $\gamma^M$ may be written $\Psi=\Psi'+\Psi''$, where
$\Psi':M\rightarrow {\mathfrak{t}^{d-1}_c}^\vee$, $\Psi'':M\rightarrow {\mathrm{span}(\tilde{\boldsymbol{\nu}})}^\vee$.
The restriction of $\gamma^M$ to
$T^{d-1}_c$ is Hamiltonian, with moment map 
$\Psi'$. On the other, hand, with the usual identification 
of the Lie algebra and coalgebra of $S^1$ with $\imath\,\mathbb{R}$,
$\rho^X$ is the contact lift of the the trivial action of $S^1$
on $(M,\,2\,\omega)$ with costant moment map $\imath$.
Therefore, 
$\beta^X$ is the contact lift of the Hamiltonian action 
$\beta^M:\widehat{T}^d\times M\rightarrow M$ with moment map
$\Xi=(\Psi',\imath)$. In view of Proposition \ref{prop:twisted lift hat},
we conclude the following.

\begin{prop}
\label{prop:hamiltonian on Mhat}
$\beta^{\widehat{M}_{\boldsymbol{\nu}}}$ in (\ref{eqn:beta descended}) 
is Hamiltonian, with moment map
$$
\hat{\Xi}:=\left( \frac{\Psi'}{\Phi^{\tilde{\boldsymbol{\nu}}}} ,\,\frac{\imath}{\Phi^{\tilde{\boldsymbol{\nu}}}} \right):\widehat{M}_{\boldsymbol{\nu}}\rightarrow
{\widehat{\mathfrak{t}}^{d}}
{}^\vee={\mathfrak{t}^{d-1}_c}^\vee\oplus \imath\,\mathbb{R}.
$$
\end{prop} 

We now argue that $\beta^{\widehat{M}_{\boldsymbol{\nu}}}$ can be
complexified to a holomorphic action of $\widehat{\mathbb{T}}^d$,
the complexification of $\widehat{T}^d$, on $\widehat{M}_{\boldsymbol{\nu}}$.
More generally, for any compact Lie group $G$ we shall denote its complexification by
$\mathbb{G}$.

To this end, we consider  the complement of the zero section
$A^\vee_0\subset A^\vee$, and observe that
all the actions involved on $X$ uniquely extend to complexified actions
on $A^\vee_0$.
Thus $\mu^X$ extends to
$\widetilde{\mu}^{A^\vee_0}:\mathbb{T}^1\times A^\vee_0\rightarrow A^\vee_0$,
$\rho^X$ to $\widetilde{\rho}^{A^\vee_0}:\mathbb{C}^*\times A^\vee_0\rightarrow A^\vee_0$,
$\gamma^X$ to $\widetilde{\gamma}^{A^\vee_0}:\mathbb{T}^d\times A^\vee_0\rightarrow A^\vee_0$,
$\beta^X$ to $\widetilde{\beta}^{A^\vee_0}:\widehat{\mathbb{T}}^d\times A^\vee_0\rightarrow A^\vee_0$;
clearly $\widehat{\mathbb{T}}^d=\mathbb{T}^{d-1}_c\times \mathbb{C}^*$.

In view of the discussion in \S 2, \S 3, and \S 5 of \cite{pao-JGP} (applied with $r=1$), 
under Basic Assuption \ref{asmpt:basic} the following holds:

\begin{enumerate}
\item $A^\vee_0=\mathbb{T}^1\cdot X$ (the $\widetilde{\mu}^{A^\vee_0}$-saturation of $X$);
\item $\widetilde{\mu}^{A^\vee_0}$ is proper and locally free;
\item there is a natural biholomorphism 
\begin{equation}
\label{eqn:AXcomplisomo}
X/T^1\cong A^\vee_0/\mathbb{T}^1
\end{equation}
where the former quotient is taken with respect to $\mu^X$, and the latter with 
respect to $\widetilde{\mu}^{A^\vee_0}$.
\end{enumerate}

Since $\widetilde{\beta}^{A^\vee_0}$ commutes with $\widetilde{\mu}^{A^\vee_0}$,
it descends to the quotient and we conclude the following.

\begin{prop}
\label{prop:complexified action}
$\beta^{\widehat{M}_{\boldsymbol{\nu}}}$ admits a unique holomorphic extension
$\widetilde{\beta}^{\widehat{M}_{\boldsymbol{\nu}}}:
\widehat{\mathbb{T}}^d\times \widehat{M}_{\boldsymbol{\nu}}\rightarrow \widehat{M}_{\boldsymbol{\nu}}$.
\end{prop}

We aim to relate the stabilizer of $m\in M$ under $\gamma^M$ to the stabilizer of
$\hat{m}\in \widehat{M}_{\boldsymbol{\nu}}$ under under $\beta^{\widehat{M}_{\boldsymbol{\nu}}}$
if $m\leftarrow x\rightarrow\hat{m}$. 
More generally, we can consider the same issue for the complexified actions
$\widetilde{\gamma}^M:\mathbb{T}^d\times M\rightarrow M$ and 
$\widetilde{\beta}^{\widehat{M}_{\boldsymbol{\nu}}}$;
by Proposition 1.6
of \cite{sj-holoslice}, the stabilizer of $m$ under $\widetilde{\gamma}^M$ is the complexification
of the stabilizer under $\gamma^M$.

There is a dense open subset $M^0\subset M$  
where $\gamma^M$ is free; then  
$\widetilde{\gamma}^M$ is free and transitive on $M^0$.
Let us consider
the corresponding open set
\begin{equation}
\label{eqn:hatM0}
\widehat{M}^0:=\widehat{M^0}=\widehat{\pi}_{\boldsymbol{\nu}}\big( \pi^{-1}(M^0) \big)\subset 
\widehat{M}_{\boldsymbol{\nu}}.
\end{equation}
Let us set $X^0:=\pi^{-1}(M^0)$.

\begin{prop}
\label{prop:free action}
Under the previous assumptions, the following holds.
\begin{enumerate}
\item $\beta^{\widehat{M}_{\boldsymbol{\nu}}}$ is free on $\widehat{M}^0$.
\item $\widetilde{\beta}^{\widehat{M}_{\boldsymbol{\nu}}}$ is free and transitive
on $\widehat{M}^0$.
\item $\widehat{\pi}_{\boldsymbol{\nu}}$ restricts to a principal $S^1$-bundle 
$X^0\rightarrow \widehat{M}^0$.
\end{enumerate}

In particular, $\widehat{M}^0$ is smooth.
\end{prop}

\begin{proof}
Suppose $m\leftarrow x\rightarrow \hat{m}$ with $m\in M^0$,
and that $(t,e^{\imath\theta})\in \widehat{T}^d$ stabilizes $\hat{m}$.
Hence there exits $h\in T^1$ such that
$$
\gamma^X_t\circ \rho^X_{e^{\imath\theta}}(x)=\mu^X_h(x).
$$
With $\chi$ as in (\ref{eqn:char mediatra}) and
Lemma \ref{lem:deltachi}, we conclude that
$$
\gamma^X_{t\,h^{-1}}\circ \rho^X_{e^{\imath\theta}\,\chi(h)^{-1}}(x)=x\quad
\Rightarrow\quad\gamma^M_{t\,h^{-1}}(m)=m.
$$
Hence $t\,h^{-1}=1\in T^d$, and since $T^{d-1}_c$ and $T^1$ are complementary subtori
we conclude $t=h=1$. Hence $e^{\imath\,\theta}=1$ as well.

This proves the first statement; that
$\widetilde{\beta}^{\widehat{M}_{\boldsymbol{\nu}}}$ is free on
$\widehat{M}_0$ then 
follows either by a similar argument using complexifications, or else by appealing to Proposition 1.6
of \cite{sj-holoslice}.

We prove that $\widetilde{\beta}^{\widehat{M}_{\boldsymbol{\nu}}}$ is transitive 
on $\widehat{M}_0$. 
Suppose $\hat{m}_j\in \widehat{M}_0$, $j=1,2$. Then there exist $m_j\in M^0$ and
$x_j\in \pi^{-1}(m_j)\subseteq X^0$, such that 
$m_j\leftarrow x_j\rightarrow \hat{m}_j$. There exists $t'\in T^d$ such that 
$\gamma^M_{t'}(m_1)=m_2$. We can factor $t'$ uniquely as
$t'=t\,h$, where $t\in T^{d-1}_c$ and $h\in T^1$. Lifting first 
this relation to $X$, and then descending to $\widehat{M}_{\boldsymbol{\nu}}$, 
this means that for some
$e^{\imath\theta}\in S^1$ we have
\begin{eqnarray}
\label{eqn:lifted relX}
\gamma^X_{t}\circ \gamma^X_h(x_1)=\rho^X_{e^{\imath\,\theta}}(x_2)
&
\Rightarrow&
 \gamma^X_{t}\circ  \rho_{\chi(h)\,e^{\imath\,\theta}}^{-1}(x_1)=\mu^X_{h^{-1}}(x_2)\nonumber\\
 &\Rightarrow&\widetilde{\beta}^{\widehat{M}_{\boldsymbol{\nu}}}_{\hat{t}}(\hat{m}_1)=
 \hat{m}_2,
\end{eqnarray}
where $\hat{t}:=\big(t,\chi(h)^{-1}\,e^{-\imath\,\theta}\big)\in \widehat{T}^d$, and $\chi$ is
as in (\ref{eqn:char mediatra}).

Finally, since $\mu^X$ lifts the restriction of $\gamma^M$ to $T^1$, and on the other hand 
$\gamma^M$ is free on $M^0$,
it follows that $\mu^X$ is free on $X^0$; the third statement follows.
\end{proof}

\begin{cor}
\label{cor:toric variety}
$\widehat{M}_{\boldsymbol{\nu}}$ (with the Hamiltonian 
action $(\widetilde{\beta}^{\widehat{M}_{\boldsymbol{\nu}}},\,\hat{\Xi})$) is a symplectic toric orbifold and
a complex toric variety.
\end{cor}

Let us consider a general triple $m\leftarrow x\rightarrow \hat{m}$, 
and denote by $T^d_m\leqslant T^d$ the stabilizer of $m$ for $\gamma^M$, and
by $\widehat{T}^d_{\hat{m}}\leqslant \widehat{T}^d$ the stabilizer of $\hat{m}$ for $\widetilde{\beta}^{\widehat{M}_{\boldsymbol{\nu}}}$.
We want to describe the relation between $T^d_m$ and 
$\widehat{T}^d_{\hat{m}}$.

Let $T^1_x\leqslant T^1$ be the stabilizer of $x$ in $T^1$ under $\mu^X$.
Recall that $\mu^M$ is the restriction of $\gamma^M$, that is,
$\mu^M=\left.\gamma^M\right|_{T^1\times M}$; hence we can unambiguously 
denote by $T^1_m=T^1\cap T^d_m$ the $\mu^M$-stabilizer of $m\in M$.

\begin{lem}
\label{lem:finite subgroupT1}
If $m=\pi(x)$, then $T^1_x$ is a finite subgroup of $T^1_m$.
\end{lem}

\begin{proof}
Since $\mu^X$ is locally free, $T^1_x$ is a discrete subgroup of $T^1$,
hence finite. Furthermore, since 
$\mu^X$ lifts $\mu^M$, which is the restriction of $\gamma^M$ to $T^1$, we also have
$T^1_x\leqslant T^d_m$. In fact, if $t\in T^1_x$ then 
$\mu^X_t(x)=x$ then $$\mu^M_t(m)=\pi\circ \mu^X_t(x)=
m,$$ hence $t\in T^d_m$.
\end{proof}

Thus $T^1_x\leqslant T^d_m$ is finite, and $T^d_m$ is a subtorus of $T^d$ \cite{del}; hence
$T^d_m/T^1_x$ is a torus of the same dimension as $T^d_m$.

\begin{prop}
\label{prop:same dim stabilizers}
If $m\leftarrow x\rightarrow \hat{m}$, there is a natural isomorphism
$T^d_m/T^1_x\cong\widehat{T}^d_{\hat{m}}$. In particular, $T^d_m$ and
$\widehat{T}^d_{\hat{m}}$ are tori of the same dimension.
\end{prop}
\begin{proof}
For every $m\in M$, there is a character $\delta_m:T^d_m\rightarrow S^1$ such that
\begin{equation}
\label{eqn:char stabil}
\gamma^X_k(x)=\rho^X_{\delta_m(k)}(x)\quad (x\in \pi^{-1}(m),\,k\in T^d_m).
\end{equation}
Let us factor $k=t\,h$ with $t\in T^{d-1}_c$, $h\in T^1$. Then (\ref{eqn:char stabil}) implies
\begin{equation}
\label{eqn:char stabi2}
\gamma^X_t\circ \rho^X_{\delta_m(k)^{-1}\,\chi(h)^{-1}}(x)=\mu^X_{h^{-1}}(x)\quad
\Rightarrow\quad \widetilde{\beta}^{\widehat{M}_{\boldsymbol{\nu}}}_{\hat{k}}\left(\hat{m}\right)
=\hat{m},
\end{equation}
where $\hat{k}:=\left(t,\delta_m(k)^{-1}\,\chi(h)^{-1}\right)\in \widehat{T}^d$.
Hence we obtain a Lie group homomorphism
\begin{equation}
\label{eqn:psim}
\psi_m:k=t\,h\in T^d_m\mapsto \hat{k}:=\left(t,\delta_m(k)^{-1}\,\chi(h)^{-1}\right)\in \widehat{T}^d_m.
\end{equation}
Let us set $T^1_m:=T^1\cap T^d_m$. Then
\begin{equation}
\label{eqn:kerpsim}
\ker(\psi_m)=\left\{h\in T^1_m\,:\, \delta_m(k)\,\chi(h)=1\right\}.
\end{equation}

\begin{lem}
\label{lem:psim discrete}
$\ker(\psi_m)=T^1_x$.
\end{lem}

\begin{proof}
By (\ref{eqn:char mediatra}) and (\ref{eqn:char stabil}), we have for $h\in T^1_m$
$$
\gamma^X_h(x)=\rho^X_{\delta_m(h)}(x)\quad \Rightarrow\quad 
\mu^X_h(x)= \rho^X_{\chi(h)\,\delta_m(h)}(x).
$$
In other words, $\ker(\psi_m)=T^1_x$.

\end{proof}

Let us prove that $\psi_m$ is surjective.
Suppose $(t,e^{\imath\theta})\in \widehat{T}^d_{\hat{m}}$, i.e.
$\widetilde{\beta}^{\widehat{M}}_{(t,e^{\imath\theta})}(\hat{m})=\hat{m}$.
There exists $h\in T^1$ such that
\begin{eqnarray*}
\gamma^X_t\circ \rho^X_{e^{\imath\,\theta}}(x)=\mu^X_{h^{-1}}(x)&\Rightarrow&
\gamma^X_{t\,h}\circ \rho^X_{\chi(h)}(x)=\rho^X_{e^{-\imath\,\theta}}(x)
\quad \Rightarrow\quad \gamma^M_{t\,h}(m)=m.
\end{eqnarray*}
Thus $k:=t\,h\in T^d_m$, whence 
$\gamma^X_{k}(x)=\rho^X_{\delta_m(k)}(x)$. We conclude
$\rho^X_{\delta_m(k)\,\chi(h)}(x)=\rho^X_{e^{-\imath\,\theta}}(x)$, so that
$e^{\imath\,\theta}=\delta_m(k)^{-1}\,\chi(h)^{-1}$.
It follows that $(t,e^{\imath\,\theta})=\psi_m(k)$.

\end{proof}

\subsection{The polytope $\widehat{\Delta}_{\boldsymbol{\nu}}$}

By Proposition \ref{prop:hamiltonian on Mhat} and 
Corollary \ref{cor:toric variety}, $\widehat{M}_{\boldsymbol{\nu}}$ is a K\"{a}hler
toric orbifold, and its associated convex rational simple polytope (\cite{del}, \cite{lt}) is
\begin{equation}
\widehat{\Delta}_{\boldsymbol{\nu}}:=
\Xi \left( \widehat{M}_{\boldsymbol{\nu}} \right)\subset
\widehat{\mathfrak{t}}^{d\,\vee}.
\end{equation}

We aim to describe the faces of $\widehat{\Delta}_{\boldsymbol{\nu}}$ in terms of the faces
of $\Delta$; to this end, we premise a few remarks.

\begin{lem}
\label{lem:invariance and transform}
Suppose that $R,\,S\subseteq M$; then 
$\widehat{R\cap S}\subseteq \widehat{R}\cap \widehat{S}$.
If furthermore $R$ and $S$ are $\mu^M$-invariant, then
$\widehat{R\cap S}=\widehat{R}\cap \widehat{S}$.
\end{lem}

\begin{proof}
Suppose $\hat{m}\in \widehat{R\cap S}$. Then there exist
$m\in R\cap S$, $x\in X$ such that $m\leftarrow x\rightarrow \hat{m}$;
hence $\widehat{m}\in \widehat{R}\cap \widehat{S}$,
so that  and $\widehat{R\cap S}\subseteq \widehat{R}\cap \widehat{S}$.

Suppose that $R$ and $S$ are $\widehat{\mu}^M$-invariant, and that
$\hat{m}\in \widehat{R}\cap \widehat{S}$. Hence there exist 
$m_1\in R$, $m_2\in S$ and $x_1,\,x_2\in X$ such that 
$m_1\leftarrow x_1\rightarrow \hat{m}$ and $m_2\leftarrow x_2\rightarrow \hat{m}$.
Hence $x_2\in T^1\cdot x_1$ ($\mu^X$-orbit) and by the equivariance of
$\pi$ this implies $m_2\in T^1\cdot m_1$ ($\mu^M$-orbit). Thus $m_1,\,m_2\in R\cap S$ and
so $\hat{m}\in \widehat{R\cap S}$. It follows that
$\widehat{R\cap S}\supseteq \widehat{R}\cap \widehat{S}$.
\end{proof}

\begin{lem}
\label{lem:closure transform}
If $R\subseteq M$, then $\overline{\widehat{R}}\subseteq \widehat{\overline{R}}$.
If in addition $R\subseteq M$ is $\mu^M$ invariant, then 
$\overline{\widehat{R}}= \widehat{\overline{R}}$.
\end{lem}

\begin{proof}
We have by definition 
$$
\widehat{\overline{R}}=\widehat{\pi}_{\boldsymbol{\nu}}\left( \pi^{-1}\left(\overline{R}\right)  \right)
\supseteq \widehat{\pi}_{\boldsymbol{\nu}}\left( \pi^{-1}\left(R\right)  \right)
=\widehat{R};
$$
hence $\widehat{\overline{R}}$ is closed and contains $\widehat{R}$, i.e. 
$\widehat{\overline{R}}\supseteq \overline{\widehat{R}}$.

Before considering the reverse inclusion, let us premise that
- since $\pi:X\rightarrow M$ is an
$S^1$-bundle - $\pi^{-1}(\overline{R})=\overline{\pi^{-1}(R)}$.

Suppose that $R$ is $\mu^M$-invariant, and let $S\subseteq \widehat{M}_{\boldsymbol{\nu}}$
be closed with $\widehat{R}\subseteq S$. 
Then 
$\widehat{\pi}_{\boldsymbol{\nu}}^{-1}(S)\supseteq \widehat{\pi}_{\boldsymbol{\nu}}^{-1}(\widehat{R})$.

\begin{claim}
\label{claim:invariance-closure}
Given that $R$ is $\mu^M$-invariant, $\widehat{\pi}_{\boldsymbol{\nu}}^{-1}(\widehat{R})=
\pi^{-1}(R)$. 
\end{claim}

\begin{proof}
[Proof of Claim \ref{claim:invariance-closure}]
By construction, $\widehat{\pi}_{\boldsymbol{\nu}}^{-1}(\widehat{R})$ is the union of all
$\mu^X$-orbits passing through points of $\pi^{-1}(R)$. By equivariance, every
$\mu^X$-orbit projects down to $M$ under $\pi$ to a $\mu^M$-orbit, and each such orbit through a
point of $R$ is entirely conained in $R$. Therefore, all $\mu^X$-orbits trhough points of
$\pi^{-1}(R)$ are entirely contained in $\pi^{-1}(R)$, whence the claim. 
\end{proof}

Thus $\widehat{\pi}_{\boldsymbol{\nu}}^{-1}(S)\supseteq \pi^{-1}(R)$,
hence 
$$
\widehat{\pi}_{\boldsymbol{\nu}}^{-1}(S)\supseteq \overline{\pi^{-1}(R)}=
\pi^{-1}\left( \overline{R} \right)\quad\Rightarrow\quad
S\supseteq \widehat{\overline{R}}.
$$
We conclude that  $\widehat{\overline{R}}\subseteq \overline{\widehat{R}}$ when $R$ is $\mu^M$-invariant,
which completes the proof of Lemma \ref{lem:closure transform}. 

\end{proof}

Let as above $\mathcal{G}(\Delta)=\{F_1,\ldots,F_k\}$ be the collection of the
facets of $\Delta$. Recalling that 
$\Psi:M\rightarrow {\mathfrak{t}^d}^\vee$ is the moment map for $\gamma^M$,
for each $j$ let $M_j:=\Psi^{-1}(F_j)$ (see \S \ref{sctn:complexification}).
Then $M_j$ is a complex submanifold of codimension $1$ of $M$.
We shall set
$$
\widehat{M}_j:=\pi_{\boldsymbol{\nu}}\left( \pi^{-1}(M_j) \right),\quad
\widehat{M}_j^0:=\widehat{M^0_j}=\pi_{\boldsymbol{\nu}}\left( \pi^{-1}(M_j^0) \right);
$$
then $\widehat{M}_j$ is a complex suborbifold of $\widehat{M}_{\boldsymbol{\nu}}$ of codimension $1$,
and $\widehat{M}_j^0$ is open and dense in $\widehat{M}_j$.
Furthermore, $M_j$ is the fixed locus of the $1$-parameter subgroup $f_j:\tau\mapsto e^{\tau\,\boldsymbol{\upsilon}_j}$,
where $\boldsymbol{\upsilon}_j$ is as in (\ref{eqn:upsilonjlambdaj}) and (\ref{eqn:Delta inters}),
while
$M_j^0$ is the locus of those points in $M$ whose stabilizer subgroup in $T^d$ is 
exactly $f_j(S^1)$.

Since $M_j$ and $M_j^0$ are $\mu^M$-invariant for every $j$, in light of Proposition 
\ref{prop:same dim stabilizers}, Lemma \ref{lem:invariance and transform},
and Lemma \ref{lem:closure transform}
we conclude the following.

\begin{cor}
\label{cor:loci to loci}
For every $j,l\in \{1,\ldots,k\}$, the following holds:
\begin{enumerate}
\item $\widehat{M_j\cap M_l}=\widehat{M}_j\cap \widehat{M}_l$;
\item $\widehat{M}_j=\overline{\widehat{M_j^0}}$;
\item $\widehat{M}_j^0\cap \widehat{M}_l^0=\emptyset$ if $j\neq l$;
\item for every $j=1,\ldots,k$, 
$\bigcup_{j=1}^k\widehat{M}_j^0\subset \widehat{M}_{\boldsymbol{\nu}}$ is the locus of points with a $1$-dimensional stabilizer subgroup in $\widehat{T}^d$.
\end{enumerate}

\end{cor}

Let us set $\widehat{F}_j:=\Xi(\widehat{M}_j)$
We can use the conclusions of Corollary \ref{cor:loci to loci} to relate the faces
of $\Delta$ and $\widehat{\Delta}_{\boldsymbol{\nu}}$.

\begin{prop}
\label{prop:bijectivefaces}
Let $\mathcal{F}_l(\widehat{\Delta}_{\boldsymbol{\nu}})$ and 
$\mathcal{G}(\widehat{\Delta}_{\boldsymbol{\nu}})
=\mathcal{F}_1(\widehat{\Delta}_{\boldsymbol{\nu}})$
be, respectively, the collections of codimension-$l$ faces and of facets 
of $\widehat{\Delta}_{\boldsymbol{\nu}}$. 
Then there are bijective correspondences
\begin{enumerate}
\item $F_j=\Psi(M_j)\in \mathcal{G}(\Delta)\mapsto \widehat{F}_j:=
\Xi(\widehat{M}_j)\in  \mathcal{G}(\widehat{\Delta}_{\boldsymbol{\nu}})$;
\item $\bigcap_{j=i_1}^{i_l}
F_{i_j}\in \mathcal{F}_l(\Delta)
\mapsto \bigcap_{j=i_1}^{i_l}
\widehat{F}_{i_j}\in \mathcal{F}_l(\widehat{\Delta}_{\boldsymbol{\nu}})$;

\item for every $j=1,\ldots,k$, the relative interior of $F_j$ is 
$F_j^0=\Psi(M_j^0)$;

\item for every $j=1,\ldots,k$, the relative interior of $\widehat{F}_j$ is 
$\widehat{F}_j^0=\Xi(\widehat{M}_j^0)$

\end{enumerate}

\end{prop}

\begin{proof}
By the theory in \cite{del} and \cite{lt}, the moment maps $\Psi$ and $\Xi$ are quotients
by the respective actions, i.e., the fibers are orbits. 
Furthermore, the relative interior of $F_j$ is 
$F_j^0=\Psi(M_j^0)$.

By Corollary \ref{cor:loci to loci}, the orbifolds $\widehat{M}_j^0$
are the connected components of the locus of points in $\hat{M}_{\boldsymbol{\nu}}$
having $1$-dimensional stabiliser subgroup in $\widehat{T}^d$.
Therefore, $\overline{\Xi(\widehat{M}_j^0})=\Xi(\widehat{M}_j)$ 
is a facet of $\widehat{\Delta}_{\boldsymbol{\nu}}$, and these are all the (distinct)
faces of $\widehat{\Delta}_{\boldsymbol{\nu}}$.

One argues similarly for the other faces.

\end{proof}

Let us next determine the normal vectors to the facets
$\widehat{F}_j$ of $\widehat{\Delta}_{\boldsymbol{\nu}}$.
This amounts to determining the stabilizer subgroup in $\widehat{T}^d$ of the points
in each relative interior $\widehat{F}_j^0$ (recall that $\widehat{T}^d$ acts
on $\widehat{M}_{\boldsymbol{\nu}}$ by $\beta^{\widehat{M}_{\boldsymbol{\nu}}}$ in
(\ref{eqn:beta descended})).

\begin{prop}
\label{prop:stabilizerhatFj0}
The stabilizer subgroup for $\beta^{\widehat{M}_{\boldsymbol{\nu}}}$ of any $\hat{m}\in \widehat{F}_j$
is the $1$-parameter subgroup of $\widehat{T}^d$ generated by 
$\widehat{\boldsymbol{\upsilon}}_j:=\boldsymbol{\upsilon}_j'
-(\lambda_j+\rho_{j}\,\delta)\,\partial^{S^1}_\theta\in \widehat{\mathfrak{t}}^d$.

\end{prop}

\begin{rem}
Given Lemma \ref{lem:deltachi}, $\hat{\boldsymbol{\upsilon}}_j\in L(\widehat{T}^d)$; it is not claimed that $\hat{\boldsymbol{\upsilon}}_j$ is primitive.

\end{rem}

\begin{proof}
Assume $m\leftarrow x\rightarrow \widehat{m}$ with $m\in F_j^0$, whence
$\widehat{m}\in \widehat{F}_j^0$. The stabilizer subgroup of $m$ is then the 
$1$-paramenter subgroup
$\tau\mapsto e^{\tau\,\boldsymbol{\upsilon}_j}$, where $\boldsymbol{\upsilon}_j$ is as
in (\ref{eqn:jth face}).
Hence $\mu^M_{e^{\tau\,\boldsymbol{\upsilon}_j}}(m)=m$ for every $\tau\in \mathbb{R}$,
and by Lemma \ref{lem:characterFgamma} 

\begin{equation}
\label{eqn:gammaXstab}
\gamma^X_{ e^{  \tau\,\boldsymbol{\upsilon}_j       }      }(x)=
\rho^X_{e^{\tau\,\lambda_j}}(x)\quad(\tau\in \mathbb{R}).
\end{equation}
In view of (\ref{eqn:upsilonj decmp}), (\ref{eqn:char mediatra}),
and Lemma \ref{lem:deltachi}, (\ref{eqn:gammaXstab}) may be rewritten
as follows
\begin{eqnarray}
\label{eqn:gammaXstab1}
x&=&\gamma^X_{ e^{  \tau\,(\boldsymbol{\upsilon}_j'+\rho_{j}\tilde{\boldsymbol{\nu}})      }      }\circ \rho^X_{e^{-\tau\,\lambda_j}}(x)=\gamma^X_{ e^{  \tau\,\boldsymbol{\upsilon}_j'     }      }\circ \rho^X_{e^{-\tau\,\lambda_j}}
\circ \gamma^X_{\rho_{j}\tilde{\boldsymbol{\nu}} }(x)\nonumber\\
&=&\gamma^X_{ e^{  \tau\,\boldsymbol{\upsilon}_j'     }      }\circ \rho^X_{e^{-\tau\,\lambda_j}}
\circ \rho^X_{e^{-\tau\,\rho_{j}\,\delta}}
\circ \mu^X_{e^{\tau\,\rho_{j}\tilde{\boldsymbol{\nu}}} }(x)\nonumber\\
&=&\gamma^X_{ e^{  \tau\,\boldsymbol{\upsilon}_j'     }      }\circ 
\rho^X_{e^{-\tau\,(\lambda_j+\rho_{j}\,\delta)}}
\circ \mu^X_{e^{\tau\,\rho_{j}\tilde{\boldsymbol{\nu}}} }(x)\nonumber\\
&=&\mu^X_{e^{\tau\,\rho_{j}\tilde{\boldsymbol{\nu}} }}\circ \beta^X_{e^{\tau\,
[\boldsymbol{\upsilon}_j'
-(\lambda_j+\rho_{j}\,\delta)\,\partial^{S^1}_\theta]}}(x)
\quad(\tau\in \mathbb{R}),
\end{eqnarray}
where $\beta^X$ was introduced in (\ref{eqn:gammaXhat}).
Passing to the quotient, we can reformulate (\ref{eqn:gammaXstab1}) 
in terms of $\beta^{\widehat{M}_{\boldsymbol{\nu}}}$:
\begin{eqnarray}
\label{eqn:gammaXstab2}
\widehat{m}&=& \beta^{\widehat{M}_{\boldsymbol{\nu}}}_{e^{\tau\,
[\boldsymbol{\upsilon}_j'
-(\lambda_j+\rho_{j}\,\delta)\,\partial^{S^1}_\theta]}}(\widehat{m})
\quad(\tau\in \mathbb{R}).
\end{eqnarray}

\end{proof}

Given Proposition \ref{prop:stabilizerhatFj0}, the facets of $\widehat{\Delta}_{\boldsymbol{\nu}}$
are defined by equations of the form

\subsection{Proof of Theorem \ref{thm:case r=1}}
\label{sctn:proofT1}
We can now combine the previous results to a proof of Theorem \ref{thm:case r=1}.
Let us premise a piece of notation.
We shall denote by $\mathrm{d}^{S^1}\theta$ 
the dual basis in $\mathrm{Lie}(S^1)^\vee$
to $\partial_\theta^{S^1}$. Assuming $m\leftarrow x\rightarrow \hat{m}$,
by Proposition \ref{prop:hamiltonian on Mhat}
\begin{equation}
\label{eqn:Ximhatm}
\hat{\Xi}(\hat{m})= \frac{\Psi'(m)}{\Phi^{\tilde{\boldsymbol{\nu}}}(m)}+\frac{1}{\Phi^{\tilde{\boldsymbol{\nu}}}(m)}\,\mathrm{d}^{S^1}\theta =
\frac{\Psi'(m)+\mathrm{d}^{S^1}\theta}{\Psi^{\tilde{\boldsymbol{\nu}}}(m)+\delta}.
\end{equation}

\begin{proof}
[Proof of Theorem \ref{thm:case r=1}]
We have 
$\widehat{\Delta}_{\boldsymbol{\nu}}=\hat{\Xi}\big(\widehat{M}_{\boldsymbol{\nu}}\big)$. Assume 
$\rho\in \widehat{\Delta}_{\boldsymbol{\nu}}$, and choose a triple $m\leftarrow x\rightarrow \hat{m}$
with $\rho=\hat{\Xi}(\hat{m})$. 
Then $\langle\Psi(m),\boldsymbol{\upsilon}_j\rangle \ge \lambda_j$ for every $j=1,\ldots,k$.
With $\boldsymbol{\upsilon}_j$ as in (\ref{eqn:upsilonj decmp}) this yields for every $j$
\begin{equation}
\label{eqn:ineq1psi}
\langle \Psi'(m),\boldsymbol{\upsilon}_j'\rangle-(\lambda_j+\rho_j\,\delta)   \ge 
-\rho_j\left(\Psi^{\tilde{\boldsymbol{\nu}}}(m)+\delta\right)=
-\Phi^{\tilde{\boldsymbol{\nu}}}(m)\,\rho_j.
\end{equation}
In view of Proposition \ref{prop:stabilizerhatFj0} and
(\ref{eqn:Ximhatm}),
dividing (\ref{eqn:ineq1psi}) by $\Phi^{\tilde{\boldsymbol{\nu}}}(m)>0$, one gets
\begin{equation}
\label{eqn:ineq1psi1}
\left\langle\Xi'(\hat{m}),\widehat{\boldsymbol{\upsilon}}_j\right\rangle\ge -\rho_j.
\end{equation}
Hence, every $\rho\in \widehat{\Delta}_{\boldsymbol{\nu}}$ satisfies 
$\left\langle\rho,\widehat{\boldsymbol{\upsilon}}_j\right\rangle\ge -\rho_j$ for every $j$. Furthermore, the previous argument also shows that the inequalities are all strict if and only if 
$\rho\in \widehat{M}^0$ (notation as in (\ref{eqn:hatM0})), and that on the other hand equality holds for exactly one $j$ if and only if $\rho$ belongs to the corresponding facet $\hat{F}_j$.

Since we know already that $\widehat{\Delta}_{\boldsymbol{\nu}}$ is a rational convex 
polytope and the $\widehat{F}_j$'s are its facets, we conclude that
$$
\widehat{\Delta}_{\boldsymbol{\nu}}=
\bigcap_{j=1}^k\left\{\rho \in \widehat{\mathfrak{t}}_{\boldsymbol{\nu}}^\vee\,:\,\rho \left(\widehat{\boldsymbol{\boldsymbol{\upsilon}}}
_j\right)\ge -\rho_j\right\},
$$
and in particular that each $\widehat{\boldsymbol{\upsilon}}_j$ is inward-pointing.

The previous discussion completes the proof that shape of $\widehat{\Delta}_{\boldsymbol{\nu}}$ is as claimed in the statement of Theorem \ref{thm:case r=1},
except that $\widehat{\Delta}_{\boldsymbol{\nu}}$ is realized in the Lie coalgebra of
$\widehat{T}^d$ rather than $T^d$. 
To obtain the corresponding statement of Theorem
\ref{thm:case r=1} we need only compose with the isomorphism 
$T^d\cong T^{d-1}_c\times T^1\rightarrow T^{d-1}_c\times S^1$ given by the product of the identity and $e^{\vartheta \,\boldsymbol{\nu}}\mapsto e^{\imath\,\vartheta}$.

It remains to determine the marking of $\widehat{\Delta}_{\boldsymbol{\nu}}$, that is,
the assignment to each facet $\widehat{F}_j$ of the order $s_j\ge 1$ of the structure group
$G_j$ of
an arbitrary $\hat{m} \in \widehat{M}_j^0$. By construction, 
given $m\leftarrow x\rightarrow \hat{m}$,
up to isomorphism $G_j$ may be identified with the stabilizer subgroup
$T^1_x\leqslant T^1$ of $x$
under $\mu^X$. Now if $\mu^X_{e^{\vartheta\,\boldsymbol{\nu}}}(x)=x$ 
then $\mu^M_{e^{\vartheta\,\boldsymbol{\nu}}}(m)=m$ by equivariance of $\pi$. Since $m\in M_j^0$, this means that for some
(unique) $e^{\imath\,\vartheta'}\in S^1$
\begin{equation}
\label{eqn:stabrelXM}
e^{\vartheta\,\widetilde{\boldsymbol{\nu}}}=e^{\vartheta'\,\boldsymbol{\upsilon}_j}=
e^{\vartheta'\,(\boldsymbol{\upsilon}'_j+\rho_j\,\widetilde{\boldsymbol{\nu}})}\quad
\Rightarrow\quad 
e^{\vartheta'\,\boldsymbol{\upsilon}'_j}=e^{(\vartheta-\rho_j\,\vartheta')\,
\widetilde{\boldsymbol{\nu}}}\in T^{d-1}_c\cap T^1=(1)
\end{equation}
(notation as in (\ref{eqn:upsilonj decmp})).

In particular, since $\tilde{\boldsymbol{\nu}}$ is primitive, we see from (\ref{eqn:stabrelXM})
that
$e^{\imath\,\vartheta}=e^{\imath\,\rho_j\,\vartheta'}\in S^1$.
Let us distinguish the following cases, depending on the relation between 
$\widetilde{\boldsymbol{\nu}}$ and $\boldsymbol{\upsilon}_j$.

\textbf{Case 1.}
Suppose $\rho_j=0$, that is, 
$\boldsymbol{\upsilon}_j=\boldsymbol{\upsilon}_j'
=\widehat{\boldsymbol{\upsilon}}_j\in L(T^{d-1}_c)\subset \mathfrak{t}^{d-1}_c$.
Then $e^{\imath\,\vartheta}=1$ (and since $\boldsymbol{\upsilon}_j$ is primitive we obtain from (\ref{eqn:stabrelXM}) that
$e^{\imath\,\vartheta'}=1$ as well). Thus
$G_x$ is trivial in this case, that is, $s_j=1=(\widehat{\upsilon}_j)$.

At the opposite extreme, suppose that $\boldsymbol{\upsilon}_j\wedge \widetilde{\boldsymbol{\nu}}=0$.
Then $\boldsymbol{\upsilon}_j'=0$ and $\boldsymbol{\upsilon}_j=\pm \widetilde{\boldsymbol{\nu}}$
by primitivity.

\textbf{Case 2.}
Assume first $\boldsymbol{\upsilon}_j= \widetilde{\boldsymbol{\nu}}$, that is,
$\rho_j=1$; thus
$\widehat{\boldsymbol{\upsilon}}_j
=-(\lambda_j+\delta)\,\widetilde{\boldsymbol{\nu}}$. 
As $m\in F_j$,
\begin{equation}
\label{eqn:rellambda delta+}
0<\Phi^{\widetilde{\boldsymbol{\nu}}}(m)=\Psi^{\widetilde{\boldsymbol{\nu}}}(m)+\delta
=\Psi^{\boldsymbol{\upsilon}_j}(m)+\delta=\lambda_j+\delta.
\end{equation}
On the other hand, by (\ref{eqn:char mediatra}), Lemma \ref{lem:deltachi}, and
Lemma \ref{lem:characterFgamma}
\begin{equation}
\label{eqn:second case+}
\mu^X_{e^{\vartheta\,\widetilde{\boldsymbol{\nu}}}}(x)=
\gamma^X_{e^{\vartheta\,\widetilde{\boldsymbol{\nu}}}}\circ \rho^X_{e^{\imath\delta\vartheta}}(x)
=
\gamma^X_{e^{\vartheta\,\boldsymbol{\upsilon}_j}}\circ \rho^X_{e^{\imath\delta\vartheta}}(x)=
e^{-\imath\,\vartheta\,(\lambda_j+\delta)}\,x.
\end{equation}
Hence if $\widetilde{\boldsymbol{\nu}}=\boldsymbol{\upsilon}_j$, then $\delta+\lambda_j>0$,
and  $G_x$ is isomorphic to the group of $(\lambda_j+\delta)$-th roots of unity, that is,
$s_j=\lambda_j+\delta=(\widehat{\boldsymbol{\upsilon}}_j)$.

\textbf{Case 3.}
If $\boldsymbol{\upsilon}_j= -\widetilde{\boldsymbol{\nu}}$, then
$\rho_j=-1$; hence $\widehat{\boldsymbol{\upsilon}}_j
=-(\lambda_j-\delta)\,\widetilde{\boldsymbol{\nu}}$.  In place
of (\ref{eqn:rellambda delta+}) we have
\begin{equation}
\label{eqn:rellambda delta-}
0<\Phi^{\widetilde{\boldsymbol{\nu}}}(m)=\Psi^{\widetilde{\boldsymbol{\nu}}}(m)+\delta
=\Psi^{-\boldsymbol{\upsilon}_j}(m)+\delta=-\lambda_j+\delta,
\end{equation}
and in place of (\ref{eqn:second case+}) we obtain
\begin{equation}
\label{eqn:second case-}
\mu^X_{e^{\vartheta\,\widetilde{\boldsymbol{\nu}}}}(x)=
\gamma^X_{e^{\vartheta\,\widetilde{\boldsymbol{\nu}}}}\circ \rho^X_{e^{\imath\delta\vartheta}}(x)
=
\gamma^X_{e^{-\vartheta\,\boldsymbol{\upsilon}_j}}\circ \rho^X_{e^{\imath\delta\vartheta}}(x)=
e^{-\imath\,\vartheta\,(-\lambda_j+\delta)}\,x.
\end{equation}
Hence 
$\delta-\lambda_j>0$
and $G_x$ is isomorphic to the group of $(\delta-\lambda_j)$-th roots of unity. In particular,
$s_j=\delta-\lambda_j=(\widehat{\boldsymbol{\upsilon}}_j)$.

\textbf{Case 4.}
Suppose that $\boldsymbol{\upsilon}_j\not\in L(T^1)\cup L(T^{d-1}_c)$, that is,
$\rho_j\,\boldsymbol{\upsilon}_j'\neq \mathbf{0}$.
By (\ref{eqn:stabrelXM}) $e^{\imath\,\vartheta'}$ is a $(\boldsymbol{\upsilon}_j')$-th root
of unity, and $e^{\imath\,\vartheta}=e^{\imath\,\vartheta'\,\rho_j}$.
We have
\begin{eqnarray}
\label{eqn:intermediate case}
\mu^X_{e^{\vartheta\,\widetilde{\boldsymbol{\nu}}}}(x)&=&\gamma^X_{e^{\vartheta\,\widetilde{\boldsymbol{\nu}}}}
\circ \rho_{e^{\vartheta\,\delta}}^X(x)\nonumber\\
&=&\gamma^X_{e^{\vartheta'\,\boldsymbol{\upsilon}_j}}
\circ \rho_{e^{\vartheta'\,\rho_j\,\delta}}^X(x)
=e^{-\imath\,\vartheta'\,(\lambda_j+\rho_j\,\delta)}\,x.
\end{eqnarray}

Thus, $e^{\vartheta\,\widetilde{\boldsymbol{\nu}}}\in G_x$ if and only if
$e^{\imath\,\vartheta}=e^{\imath\,\rho_j\,\vartheta'}$, where $e^{\imath\,\vartheta'}$ is both
a $(\boldsymbol{\upsilon}_j')$-th root of unity 
(by (\ref{eqn:stabrelXM})
and a $(\lambda_j+\rho_j\,\delta)$-th
root of unity (by (\ref{eqn:intermediate case})), i.e. a $G.C.D.\big( (\boldsymbol{\upsilon}_j'),\lambda_j+\rho_j\,\delta  \big)$-th
root of unity. 

Since $\boldsymbol{\upsilon}_j$ is primitive, 
$G.C.D.\big((\boldsymbol{\upsilon}_j'),\rho_j  \big)=1$; therefore also
$$
G.C.D.\Big(G.C.D.\big( (\boldsymbol{\upsilon}_j'),\lambda_j+\rho_j\,\delta  \big),\rho_j\Big)=1.
$$
Hence we may assume that $e^{\imath\,\vartheta}$ is a primitive 
$G.C.D.\big( (\boldsymbol{\upsilon}_j'),\lambda_j+\rho_j\,\delta  \big)$-th
root of unity. In other words, $G_x$ is isomorphic to the group of $G.C.D.\big( (\boldsymbol{\upsilon}_j'),\lambda_j+\rho_j\,\delta  \big)$-th roots of unity, whence by Proposition 
\ref{prop:stabilizerhatFj0}
$$
s_j=G.C.D.\big( (\boldsymbol{\upsilon}_j'),\lambda_j+\rho_j\,\delta  \big)=
\left(\widehat{\boldsymbol{\upsilon}}_j\right).
$$

The proof of Theorem \ref{thm:case r=1} is complete.
\end{proof}

\begin{proof}
[Proof of Corollary \ref{cor:equalcohomology/Q}]
Since $J$ and $\widehat{J}_{\boldsymbol{\nu}}$ are torus invariant complex
structures on $M$ and $\widehat{M}_{\boldsymbol{\nu}}$, respectively, by
Theorem 9.1 of \cite{lt} both $M$ and $\widehat{M}_{\boldsymbol{\nu}}$ have
structures of complex toric varieties (of course in the case of $M$ this is
our starting assumption); furthermore, the corresponding fans 
$\mathrm{Fan}(M)$ and $\mathrm{Fan}(\widehat{M}_{\boldsymbol{\nu}})$
are defined by their respective polytopes, $\Delta$ and $\widehat{\Delta}_{\boldsymbol{\nu}}$. Since $\Delta$ and $\widehat{\Delta}_{\boldsymbol{\nu}}$
are simple and compact, $\mathrm{Fan}(M)$ and $\mathrm{Fan}(\widehat{M}_{\boldsymbol{\nu}})$ are simplicial and complete.

Hence the Betti numbers $\beta_j$ and $\widehat{\beta}_j$ of  $M$ and $\widehat{M}_{\boldsymbol{\nu}}$ are determined by the collection of the 
all the numbers $d_r$ and $\widehat{d}_r$ of 
$r$-dimensional cones in $\mathrm{Fan}(M)$ and $\mathrm{Fan}(\widehat{M}_{\boldsymbol{\nu}})$, respectively (\S 4.5 of \cite{ful}).
Thus it suffices to prove that $d_r=\widehat{d}_r$ for any $r$. 

On the other hand, in order 
to determine the fan $\mathrm{Fan}_\Gamma$ associated to a polytope $\Gamma$ we may assume 
without loss that
$\Gamma$ contains the origin in its interior; in this case, furthermore, 
the cones in $\mathrm{Fan}_\Gamma$ are the cones over the faces of the polar
polytope $\Gamma^0$ to $\Gamma$ (\S 1.5 of \cite{ful}). Hence we need to show the polar polytopes of
(suitable translates of) $\Delta$ and $\widehat{\Delta}_{\boldsymbol{\nu}}$
share the same number of faces in each dimension.
However, for any $d$-dimensional polytope $\Gamma$
in a $d$-dimensional real vector space, containing the origin in its interior, 
there
is an order-reversing bijection between the faces of $\Gamma$ and those of $\Gamma^0$, with corresponding faces $F$ and $F^*$ having dimensions adding up to $d-1$ (\textit{loc. cit.}). Thus the statement follows from Proposition \ref{prop:bijectivefaces}.

\end{proof}

It is in order to briefly digress on how $\widehat{\Delta}_{\boldsymbol{\nu}}$ in Theorem \ref{thm:case r=1}
depends on the choice of $T^{d-1}_c\leqslant T^d$. 

Suppose first that $\boldsymbol{\delta}=\mathbf{0}$ in (\ref{eqn:Trpsi}), so that
$\mu^X$ is the restriction of $\gamma^X$ to $T^1$.
Let $S^{d-1}_c,\,T^{d-1}_c\leqslant T^d$ be different complementary subtorii to $T^1$, so
that $T^d\cong S^{d-1}_c\times T^1\cong T^{d-1}_c\times T^1$; thus projecting 
onto $T^{d-1}_c$ along $T^1$ yields an isomorphism $P:S^{d-1}_c\cong T^{d-1}_c$.
Let us choose an isomorphism of the standard torus $T_{st}^{d-1}=(S^1)^{d-1}$ with
$S^{d-1}_c$, so that (composing with $P$) $T_{st}^{d-1}\cong S^{d-1}_c\cong T^{d-1}_c$.

We obtain actions $\varphi^X$ and $\psi^X$ of $T^{d-1}_{st}$ on $X$,  
by composing the previous isomorphisms with the restrictions of $\gamma^X$ to
$S^{d-1}_c$ and $T^{d-1}_c$ respectively; then $\varphi^X\neq \psi^X$ (unless
$S^{d-1}_c= T^{d-1}_c$) and, by construction, the two actions differ by the composition of
a character $T^{d-1}_{st}\rightarrow T^1$ with $\mu^X$. 
On the other hand, since $\varphi^X$ and $\psi^X$ commute with
$\mu^X$, they descend to symplectic actions $\varphi^{\widehat{M}_{\boldsymbol{\nu}}}$ and 
$\psi^{\widehat{M}_{\boldsymbol{\nu}}}$ of $T^{d-1}_{st}$
on $(\widehat{M}_{\boldsymbol{\nu}}, 2\,\widehat{\omega}_{\boldsymbol{\nu}})$; in fact,
by the previous remark and the construction of 
$\widehat{M}_{\boldsymbol{\nu}}$ as a quotient,
$\varphi^{\widehat{M}_{\boldsymbol{\nu}}}=\psi^{\widehat{M}_{\boldsymbol{\nu}}}$.

Thus $\varphi^X$ and $\psi^X$ are different contact lifts to $X$ of the same symplectic
action of $T^{d-1}_{st}$ on $\widehat{M}_{\boldsymbol{\nu}}$;
hence they correspond to different Hamiltonian structures for the latter action, whose moment
maps differ by a translation in ${\mathfrak{t}^{d-1}_{st}}^\vee$.

Let us consider the
action of $\widehat{T}^d:=T^{d-1}_{st}\times S^1$ 
on $\widehat{M}_{\boldsymbol{\nu}}$ given by the product of 
$\varphi^{\widehat{M}_{\boldsymbol{\nu}}}=\psi^{\widehat{M}_{\boldsymbol{\nu}}}$
and $\rho^{\widehat{M}_{\boldsymbol{\nu}}}$, where the latter is
the action of $S^1$ on $\widehat{M}_{\boldsymbol{\nu}}$ obtained by descending $\rho^X$.
Let us adopts the previous choices of Hamiltonian structures
(for the second factor, we use the same Hamiltonian struture as 
in the proof of Theorem \ref{thm:case r=1}, see Propositions
\ref{prop:twisted lift hat} and \ref{prop:hamiltonian on Mhat}). 
The corresponding moment maps then differ by a translation
by a constant in ${\mathfrak{t}^{d-1}_c}^\vee\times \{0\}$; hence so do the corresponding 
moment polytopes, say $\widehat{\Delta}_{\boldsymbol{\nu}}^\psi$ and 
$\widehat{\Delta}_{\boldsymbol{\nu}}^\varphi$.

The previous considerations may be extended to the case where
$\boldsymbol{\delta}\neq \mathbf{0}$, and therefore
$\varphi^{\widehat{M}_{\boldsymbol{\nu}}}\neq 
\psi^{\widehat{M}_{\boldsymbol{\nu}}}$.
In fact, if
$\boldsymbol{\delta}\neq \mathbf{0}$, then $\varphi^X$ and $\psi^X$ differ by the composition
of a morphism $T^{d-1}_{st}\rightarrow T^1$, say
of the form 
$e^{\imath\,\boldsymbol{\vartheta}}\mapsto 
e^{\imath\,\langle \mathbf{a},\boldsymbol{\vartheta}\rangle\,
\widetilde{\boldsymbol{\nu}}}$ 
where $\mathbf{a}\in \mathbb{Z}^{d-1}$,
with $\gamma^X$. Hence, passing to the
quotient, in view of (\ref{eqn:char mediatra})
the induced actions $\varphi^{\widehat{M}_{\boldsymbol{\nu}}}$ and 
$\psi^{\widehat{M}_{\boldsymbol{\nu}}}$ will now differ by the composition of a character 
$T^{d-1}_{st}\rightarrow S^1$ of the form 
$e^{\imath\,\boldsymbol{\vartheta}}\mapsto e^{-\imath\,\delta\,\langle \mathbf{a},\boldsymbol{\vartheta}\rangle}$ with $\rho^{\widehat{M}_{\boldsymbol{\nu}}}$.
Identifying the coalgebra of $T^{d-1}_{st}$ with $\imath\,\mathbb{R}^{d-1}$, 
the corresponding moment maps $\Phi^{\psi}$ and $\Phi^{\varphi}$ for 
$\psi^{\widehat{M}_{\boldsymbol{\nu}}}$ and $\varphi^{\widehat{M}_{\boldsymbol{\nu}}}$ 
are related by a relation of the form
$\Phi^{\psi}=\Phi^{\varphi}-\delta\,\mathbf{a}\,\Gamma$,
where $\Gamma:\widehat{M}_{\boldsymbol{\nu}}\rightarrow \imath\,\mathbb{R}$
is the moment map for $\rho^{\widehat{M}_{\boldsymbol{\nu}}}$ (recall from 
Proposition \ref{prop:hamiltonian on Mhat} that $\Gamma(\hat{m})
=\imath\,\Phi^{\tilde{\boldsymbol{\nu}}}(m)^{-1}$ if $m\leftarrow x\rightarrow \hat{m}$).

It follows that the two cones are related by a transformation in $\imath\,\mathbb{R}^{d}\times \imath\,\mathbb{R}$
of the form $\imath\,(\mathbf{x},y)\mapsto \imath\,(\mathbf{x}-y\,\delta\,\mathbf{a},y)$,
followed perhaps by a translation.

\bigskip

\section{The case of arbitrary $r$}

We shall now now remove the restriction that $r=1$, and allow any value
$1\le r\le d$. Before dealing directly with the geometric situation, we shall dwell on
some handy technical results.

\subsection{Preliminaries on transversality of polytopes}
\label{sctn:transverse polytopes}

\begin{defn}
\label{defn:transverse polytopes}
Let $V$ be a finite dimensional real vector space, $\Gamma\subset V$ a convex
polytope, $W\subseteq V$ an affine subspace.
We shall say that $\Gamma$ and $W$ \textit{meet transversely}, of that they are transverse to each other,
if $W$ is transverse to the relative interior $F^0$ of every face $F$ of $\Gamma$.
\end{defn}

In the hypothesis of Definition \ref{defn:transverse polytopes},
let us set $\Gamma_W:=\Gamma\cap W$. Clearly, $\Gamma_W$ is a convex polytope in 
$W$.

Let $\mathcal{F}(\Gamma)$ be the collection of faces of $\Gamma$ and
$\mathcal{G}(\Gamma)=\{F_1,\ldots,F_s\}\subseteq \mathcal{F}(\Gamma)$ be the subset of its facets.
For each $j=1,\ldots,s$ let $\ell_j\in V^\vee$ be an inward normal covector
to $F_j$, so that
$$
\Gamma=\bigcap _{j=1}^s\left\{ v\in V\,:\,\ell_j (v)\ge \lambda_j      \right\}
$$
for certain $\lambda_j\in \mathbb{R}$;
the $j$-th facet is thus
\begin{equation}
\label{eqn:jth facet Gamma}
F_j:=\Gamma\cap \big\{p\in V\,:\,\ell_j(p)=\lambda_j\}.
\end{equation}
If $L\in \mathcal{F}(\Gamma)$, there exists
a unique subset $I_L\subseteq \{1,\ldots,s\}$ such that
$L=\bigcap_{j\in I_L}F_j$.

We are interested in \textit{simple} polytopes (meaning that 
exactly $n$ facets of $\Gamma$ meet at each vertex, where $n=\dim_{\mathbb{R}}(V)$
\cite{lt}; 
if $\Gamma$ is simple, then every codimension-$k$
face $L\in \mathcal{F}(\Gamma)$ is the intersection of exactly $k$ facets, 
that is, $|I_L|$ equals the codimension of
$L$.

\begin{prop}
\label{prop:transverse polytope faces}
In the setting of Definition \ref{defn:transverse polytopes}, suppose that
$\Gamma$ and $W$ meet transversely. The following holds.

\begin{enumerate}
\item If $F\subseteq L$ are faces of $\Gamma$ and 
$W\cap F^0\neq \emptyset$, then 
$W\cap L^0\neq \emptyset$. 
\item If $W\cap \Gamma\neq \emptyset$, then $W\cap \Gamma^0\neq \emptyset$
\end{enumerate}

\end{prop}

More precisely, regarding 1. we shall show that for any
$p\in W\cap F^0$ and any open neighborhood $W'$ of $p$ in 
$W$ one has $W'\cap L^0\neq \emptyset$. Similarly, regarding 2. we shall show that for
any $p\in \Gamma_W$ and any open neighborhood $W'$ of $p$ in 
$W$ one has $W'\cap \Gamma^0\neq \emptyset$.

\begin{proof} 
[Proof of 1.]
Since $W\subseteq V$ is an affine subspace, it is a translate of
a vector subspace $\widehat{W}\subseteq V$. Suppose $p\in F^0\cap W$. Then
$W=p+\widehat{W}$, and by transversality the map 
$\rho:(w,q)\in \widehat{W}\times F^0\mapsto w+q\in V$ is submersive, hence open.
We have $p=\rho(\mathbf{0},p)$, hence the image of an arbitrary small neighborhood
of $(\mathbf{0},p)$ in $\widehat{W}\times F^0$ contains an open neighborhood of $p$.

Since $p\in F\subseteq L$, we can find points $p'\in L^0$ arbitrarily close to $p$.
For any such $p'$, therefore, there exist $w\in \widehat{W}$, $w\sim \mathbf{0}$, and
$q\in F^0$, $q\sim p$, such that
$p'=w+q$.

\begin{claim} 
\label{claim:GammaL0}
With the previous choices,
$w+p\in W\cap L^0$.
\end{claim}

\begin{proof}
[Proof of Claim \ref{claim:GammaL0}]
Clearly, $w+p\in W$ by construction. Let us prove that $w+p\in L^0$, i.e. that
$\ell_j(w+p)=\lambda_j$ if $j\in I_L$ and $\ell_j(w+p)>\lambda_j$ if $j\not\in I_L$.

Since $F\subseteq L$, we have $I_L\subseteq I_F$. 

If $j\in I_L$, we have 
$\ell_j(p')=\ell_j(q)=\ell(p)=\lambda_j$ since $p,\,p',\,q\in L$. Therefore,
$\ell_j(w)=\ell_j(p'-q)=\lambda_j-\lambda_j=0$. Thus
$\ell_j(w+p)=0+\lambda_j=\lambda_j$ for every $j\in I_L$, so that $w+p\in L$.

If $j\in I_F\setminus I_L$, we have 
$\ell_j(p')>\lambda_j$ since $p'\in L^0$, and 
$\ell_j(p)=\ell_j(q)=\lambda_j$ since $p,\,q\in F$. Therefore, 
$\ell_j(w)=\ell_j(p'-q)=\ell_j(p')-\ell_j(q)>0$, and so
$\ell_j(w+p)=\ell_j(w)+\lambda_j>\lambda_j$.

Let $\delta:=\min\{\ell_j(p)-\lambda_j\,:\,j\not\in I_F\}$; thus $\delta>0$, since 
$p\in F^0$.  
If $j\not\in I_F$, therefore,  
$\ell_j(p+w)\ge \lambda_j+\delta/2>\lambda_j$ as $w\sim \mathbf{0}$.
 
\end{proof}

Thus for every $p\in W\cap F^0$ we have found points $p+w$ 
arbitrarily close to $p$ in $W\cap L^0$, and this completes the proof of 1..

\end{proof}

\begin{proof}
[Proof of 2.]
This is a slight modification of the previous argument.
Suppose $p\in \Gamma\cap W$. If $p\in \Gamma^0$, there is nothing to prove. Otherwise,
$p\in F^0$ for some face $F\in \mathcal{F}(\Gamma)$. We can find $p'\in \Gamma^0$
arbitrarily close to $p$, and therefore - by the previous considerations - for any such $p'$
there exist $w\in \widehat{W}$ with $w\sim \mathbf{0}$ and $q\in F^0$ with $q\sim p$
such that $p'=w+q$. 

If $j\in I_F$, $\ell_j(p)=\ell_j(q)=\lambda_j$ since $p,\,q\in F$, and $\ell_j(p')>\lambda_j$ since $p'\in \Gamma^0$.
Therefore $\ell_j(w)=\ell_j(p')-\ell_j(q)>0$. Hence 
$\ell_j(p+w)=\lambda_j+\ell_j(w)>\lambda_j$ for all $j\in I_F$.

Let $\delta:=\min\{\ell_j(p)-\lambda_j\,:\,j\not\in I_F\}$; then $\delta>0$ since $p\in F^0$. Since
$w\sim \mathbf{0}$, $\ell_j(p+w)\ge \lambda_j+\delta/2>\lambda_j$ for all $j\not\in I_F$.

Thus $\ell_j(p+w)>\lambda_j$ for every $j=1,\ldots,r$, i.e. $p+w\in W\cap \Gamma^0$.

\end{proof}

\begin{cor}
\label{cor:interior intersection}
Under the hypothesis of Proposition \ref{prop:transverse polytope faces},
if $L$ is a face of $\Gamma$ and $\Gamma_{W}\cap L\neq \emptyset$,
then $\Gamma_{W}\cap L^0\neq \emptyset$.
\end{cor}

\begin{proof}
If $p\in W\cap L$, there is a face $F\in \mathcal{F}(\Gamma)$ with
$F\subseteq L$ and $p\in W\cap F^0$. Hence $W\cap L^0\neq \emptyset$
by Proposition \ref{prop:transverse polytope faces}.
\end{proof}

\begin{prop}
\label{prop:transverse polytope facets}
Under the hypothesis of Proposition \ref{prop:transverse polytope faces}, the following holds:
\begin{enumerate}
\item the facets of $\Gamma_W$ are the non-empty intersections of $W$ with the facets of $\Gamma$;
\item $\Gamma_W^0=W\cap \Gamma^0$;
\item if $\Gamma$ is simple, then the codimension-$k$ faces of $\Gamma_W$ are the non-empty intersections of $W$ with the
codimension-$k$ faces of $\Gamma$;
\item if $F$ is a face of $\Gamma$ such that $F_W:=F\cap W\neq \emptyset$,
then the relative interior of $F_W$ is $F_W^0=F^0\cap W$;

\item if $\Gamma$ is simple, then so is $\Gamma_W$.
\end{enumerate}

\end{prop}

\begin{proof}
[Proof of 1.]
Let $F_j$ be the $j$-th facet of $\Gamma$ as in (\ref{eqn:jth facet Gamma}), 
and suppose $F_j\cap W\neq \emptyset$. 
Then $F_j^0\cap W\neq \emptyset$ by Corollary \ref{cor:interior intersection}.
Since $W$ meets $F_j^0$ transversely by assumption, $F_j^0\cap W$ has codimension one in $W$
and $\ell_j$ restricts to a non-constant 
affine linear functional on $W$. Thus, if $p\in F_j^0$ then every neighborhood of 
$p$ in $W$ intersects both $\Gamma^0$ and $\Gamma^c$. It follows that $F_j\cap W$
is a facet of $\Gamma_W$. 

Conversely, let $F$ be a facet of $\Gamma_W$, and let $p\in F^0$. Since 
$p\not\in \Gamma^0$ (for else $p\in \Gamma_W^0$), there
exists $j\in \{1,\ldots,s\}$ such that $p\in F_j$, and therefore $F_j^0\cap W\neq \emptyset$.
By the above $F_j\cap W$ is a facet of $\Gamma_W$.
Since a small neighborhood of $p$ in $W$ meets no facet of $\Gamma_W$ other
that $F$, we may slightly perturb $p$ in $F^0$ and assume that $p\in F^0\cap F_j^0$ and
therefore $F^0\cap W'=F_j^0\cap W'$ for some open neighborhood $W'$ of $p$ in $W$.
This forces $F=F_j\cap W$.
\end{proof}

\begin{proof}
[Proof of 2.]
Since every face is the intersection of the facets containing it, by 1. we have
$$
\Gamma_W^0=(W\cap \Gamma) \setminus \bigcup_{j=1}^s (W\cap F_j)
=W\cap \left(\Gamma\setminus \bigcup_{j=1}^s F_j\right)=
W\cap \Gamma^0.
$$
\end{proof}

\begin{proof}
[Proof of 3. and 4.]
If $F$ is a codimension-$k$ face of $\Gamma$ such that $F\cap W\neq \emptyset$, let us choose 
$p\in F^0\cap W\neq \emptyset$. Since $\Gamma$ is simple, 
$I_F=\{i_1,\ldots,i_k\}\subseteq \{1,\ldots,s\}$ and there is a neighborhood 
$W'$ of $p$ in $W$ such that 
$$W'\cap F=W'\cap F^0=W'\cap \bigcap_{j=1}^kF_{i_j}.$$ 
By transversality, $F^0\cap W$ has codimension
$k$ in $W$ and furthemore each $F_{i_j}$ is a facet of $\Gamma_W$ by 1.. 
Thus $F_W:=F\cap W$ is a face of $\Gamma_W$, since it is a non-empty intersection of facets,
and it has codimension-$k$ in $W$, since it has a non-empty open subset which has codimension $k$.
Furthermore, it is given by the intersection of the $k$ facets $F_{i_j}\cap \Gamma_W$ 
of $\Gamma_W$.
It then follows that 
$$
F_W^0=\bigcap_{j=1}^k(W\cap F_{i_j})\setminus \bigcup _{i\not\in I_F}(W\cap F_i)=
W\cap \left(\bigcap_{j=1}^kF_{i_j}\setminus \bigcup _{i\not\in I_F}F_i\right)=
W\cap F^0.
$$

Conversely, suppose that $K\subset \Gamma_W$ is a face, and suppose $p\in K^0$.
Since $p\not\in \Gamma^0$, there exists a unique face $F$ of $\Gamma$ such that $p\in F^0$.
By the previous discussion, $F_W=F\cap W$ is also a face of $\Gamma_W$, and 
$p\in F^0\cap W=F_W^0$.
Since distinct faces of $\Gamma_W$ have disjoint relative interiors,
$K=F_W$.
\end{proof}

\begin{proof}
[Proof of 5.]
By 3., every codimension-$k$ face of $\Gamma_W$ is the intersection of $k$ facets, and
this means that $\Gamma_W$ is simple.
\end{proof}

\subsection{The reduction $\overline{M}_{\boldsymbol{\nu}}$ and the circle bundle 
$Y_{\boldsymbol{\nu}}$}
\label{sctn:MnuYnu}

\begin{defn}
Given a subspace $\mathfrak{a}\subseteq {\mathfrak{t}^r}^\vee$, 
we shall denote by $\mathfrak{a}^\perp\subseteq \mathfrak{t}^r$ its
annihilator in $\mathfrak{t}^r$. Given 
$\boldsymbol{\rho}\in{\mathfrak{t}^r}^\vee$ we shall also set
$\boldsymbol{\rho}^\perp:=\mathrm{span}(\boldsymbol{\rho})^\perp$.

\end{defn}

\begin{defn}
A vector subspace $\mathfrak{c}\subseteq \mathfrak{t}^d$ is \textit{integral}
if $\mathfrak{c}\cap L(T)$ is a full-rank lattice in $\mathfrak{c}$ or, equivalently,
if $\mathfrak{c}$ is the Lie subalgebra of a closed embedded torus in $T^d$.

\end{defn}

Thus 
$\boldsymbol{\nu}^\perp\subseteq
\mathfrak{t}^r\subseteq \mathfrak{t}^d$ is a vector subspace of dimension $r-1$; since
$\boldsymbol{\nu}\in L(T)^\vee$, furthermore, $\boldsymbol{\nu}^\perp$ is integral. 
Let 
$T^{r-1}_{\boldsymbol{\nu}^\perp}\leqslant T^r\leqslant T^d$ be the (closed) torus with
Lie algebra $\boldsymbol{\nu}^\perp$; equivalently, if 
$\chi_{\boldsymbol{\nu}}(e^{\boldsymbol{\xi}}):=e^{2\,\pi\,\imath\,\langle\boldsymbol{\nu},\boldsymbol{\xi}\rangle}$
is the character of $T^r$ defined by $\boldsymbol{\nu}$, then $T^{r-1}_{\boldsymbol{\nu}^\perp}=
\ker(\chi_{\boldsymbol{\nu}})^0$ (the connected component of the kernel). 

We are interested
in $\widehat{M}_{\boldsymbol{\nu}}=X_{\boldsymbol{\nu}}/T^r$, the action being $\mu^X$ (notation as in (\ref{eqn:Mnudefn}) and (\ref{eqn:defn Mnu})). The latter quotient may be performed in stages. Namely, 
under Basic Assumption \ref{asmpt:basic}, $T^r$ acts in a locally free
manner on $X_{\boldsymbol{\nu}}$, whence \textit{a fortiori} so does $T^{r-1}_{\boldsymbol{\nu}^\perp}$. 
We first form the partial quotient $Y_{\boldsymbol{\nu}}
:=X_{\boldsymbol{\nu}}/T^{r-1}_{\boldsymbol{\nu}^\perp}$, where 
$T^{r-1}_{\boldsymbol{\nu}^\perp}$ acts on $X_{\boldsymbol{\nu}}$ by $\mu^X$; 
next, $\mu^X$ descends to a residual locally free action of $T^1_{\boldsymbol{\nu}}:=T^r/T^{r-1}_{\boldsymbol{\nu}^\perp}$
on $Y_{\boldsymbol{\nu}}$,
\begin{equation}
\label{eqn:muonYnu}
\mu^{Y_{\boldsymbol{\nu}}}:T^1_{\boldsymbol{\nu}}\times Y_{\boldsymbol{\nu}}
\rightarrow Y_{\boldsymbol{\nu}}.
\end{equation} 
Then
\begin{equation}
\label{eqn:partialquotient/quotientgroup}
\widehat{M}_{\boldsymbol{\nu}}=X_{\boldsymbol{\nu}}/T^r=Y_{\boldsymbol{\nu}}/T^1_{\boldsymbol{\nu}}.
\end{equation}

$Y_{\boldsymbol{\nu}}$
inherits a natural contact struture.
Let $\jmath_{\boldsymbol{\nu}}:X_{\boldsymbol{\nu}} \hookrightarrow X$ be the inclusion.
Then $\jmath_{\boldsymbol{\nu}}^*(\alpha)$ is $T^{r-1}_{\boldsymbol{\nu}^\perp}$-invariant.
Furthermore, (writing $\Phi$ for $\Phi\circ \pi$ with abuse of notation)
by definition of $X_{\boldsymbol{\nu}}$ we have
$\Phi\circ \jmath_{\boldsymbol{\nu}}=\lambda\,\boldsymbol{\nu}$ 
for some $\mathcal{C}^\infty$ function
$\lambda:X_{\boldsymbol{\nu}}\rightarrow \mathbb{R}_+$. Hence,
if $\boldsymbol{\xi}\in \mathfrak{t}^{r-1}_{\boldsymbol{\nu}^\perp}=\boldsymbol{\nu}^\perp$
then $\boldsymbol{\xi}_X^\Phi$ satisfies
$$
\left.\boldsymbol{\xi}_X^\Phi\right|_{X_{\boldsymbol{\nu}}}=\left.\boldsymbol{\xi}_M^\sharp\right|_{X_{\boldsymbol{\nu}}}-\lambda\,\langle\boldsymbol{\nu},\boldsymbol{\xi}\rangle\,
\left.\partial_\vartheta\right|_{X_{\boldsymbol{\nu}}}
=\left.\boldsymbol{\xi}_M^\sharp\right|_{X_{\boldsymbol{\nu}}}  \in \ker\left(\jmath_{\boldsymbol{\nu}}^*(\alpha)\right).
$$
In other words,
$[\iota(\boldsymbol{\xi}_X^\Phi)\,\alpha]\circ \jmath_{\boldsymbol{\nu}}=0$. 

We conclude the following.
Let $q_{\boldsymbol{\nu}}:X_{\boldsymbol{\nu}}\rightarrow Y_{\boldsymbol{\nu}}$ be the
projection. Thus we have arrows
$$
Y_{\boldsymbol{\nu}}\stackrel{q_{\boldsymbol{\nu}}}{\longleftarrow}
X_{\boldsymbol{\nu}}\stackrel{\jmath_{\boldsymbol{\nu}}}{\hookrightarrow}
X.
$$
\begin{lem}
There exists a differential 1-form (in the orbifold sense)
$\alpha_{\boldsymbol{\nu}}\in \Omega^1(Y_{\boldsymbol{\nu}})$, such that 
$q_{\boldsymbol{\nu}}^*(\alpha_{\boldsymbol{\nu}})=
\jmath_{\boldsymbol{\nu}}^*(\alpha)$.
\label{lem:quotient connectionnu}
\end{lem}

Under the stronger condition that
$T^{r-1}_{\boldsymbol{\nu}^\perp}$ acts freely on $M_{\boldsymbol{\nu}}$, the quotient 
\begin{equation}
\label{eqn:overlineMnu}
\overline{M}_{\boldsymbol{\nu}}:=M_{\boldsymbol{\nu}}/T^{r-1}_{\boldsymbol{\nu}^\perp}
\end{equation}
is smooth; furthemore, the action of $T^{r-1}_{\boldsymbol{\nu}^\perp}$
on $X_{\boldsymbol{\nu}}$ induced by $\mu^X$ is
also free, and therefore $Y_{\boldsymbol{\nu}}$ is non-singular.
In addition, $\rho^X$ (the action on $X$ generated by $-\partial_\theta$)
descends to a free $S^1$-action $\rho^{Y_{\boldsymbol{\nu}}}:S^1\times Y_{\boldsymbol{\nu}}
\rightarrow Y_{\boldsymbol{\nu}}$, and we also have
\begin{equation}
\label{eqn:overlineMnu1}
\overline{M}_{\boldsymbol{\nu}}=Y_{\boldsymbol{\nu}}/S^1.
\end{equation}
In addition,
$\alpha_{\boldsymbol{\nu}}$ in Lemma \ref{lem:quotient connectionnu}  is a connection form 
for $\rho^{Y_{\boldsymbol{\nu}}}$.

Furthermore, $\overline{M}_{\boldsymbol{\nu}}$ inherits a complex structure 
$\overline{J}_{\boldsymbol{\nu}}$ and a compatible symplectic structure $\overline{\omega}_{\boldsymbol{\nu}}$;
the triple $(\overline{M}_{\boldsymbol{\nu}},\,2\,\overline{\omega}_{\boldsymbol{\nu}},\,\overline{J}_{\boldsymbol{\nu}})$ is a 
Hodge manifold. More precisely, $(\overline{M}_{\boldsymbol{\nu}},
\,2\,\overline{\omega}_{\boldsymbol{\nu}})$ is the Marsden-Weinstein 
reduction of $(M,2\,\omega)$ under the restriction of the Hamiltonian action
$(\mu^M,\,\Phi)$ to $T^{r-1}_{\boldsymbol{\nu}^\perp}\leqslant T^r$, and
$\overline{J}_{\boldsymbol{\nu}}$ is determined from $J$ as in \cite{gs-gq}.

Let $\mathfrak{t}^1_{\boldsymbol{\nu}}$ denote the Lie algebra of 
$T^1_{\boldsymbol{\nu}}$, so that 
$$
\mathfrak{t}^1_{\boldsymbol{\nu}}\cong 
\mathfrak{t}^r/\boldsymbol{\nu}^\perp,\quad
{\mathfrak{t}^1_{\boldsymbol{\nu}}}^\vee 
\cong \mathrm{span}(\boldsymbol{\nu}).$$
By definition of $M_{\boldsymbol{\nu}}$, restricting $\Phi$ yields a map
$\Phi':M_{\boldsymbol{\nu}}\rightarrow \mathrm{span}(\boldsymbol{\nu})\cong 
{\mathfrak{t}^1_{\boldsymbol{\nu}}}^\vee$, which descends to a 
non-vanishing 
$T^1_{\boldsymbol{\nu}}$-equivariant map
$\overline{\Phi}:\overline{M}_{\boldsymbol{\nu}}\rightarrow {\mathfrak{t}^1_{\boldsymbol{\nu}}}^\vee$.

Let $\overline{\pi}_{\boldsymbol{\nu}}:Y_{\boldsymbol{\nu}}\rightarrow \overline{M}_{\boldsymbol{\nu}}$
and $\widehat{\pi}_{\boldsymbol{\nu}}:Y_{\boldsymbol{\nu}}\rightarrow \widehat{M}_{\boldsymbol{\nu}}$
be the projections. From the previous discussion and the theory in \cite{gs-gq} one obtains the
following.

\begin{prop}
\label{prop:reduction 1 dimensional case}
There is a positive holomorphic line bundle $(A_{\boldsymbol{\nu}},\,h_{\boldsymbol{\nu}})$ 
on $\overline{M}_{\boldsymbol{\nu}}$, such that:
\begin{enumerate}
\item $Y_{\boldsymbol{\nu}}$ is the unit circle bundle in $A_{\boldsymbol{\nu}}^\vee$;
\item $\alpha_{\boldsymbol{\nu}}$ is the normalized connection form associated to the unique compatible
covariant deriviative on $A_{\boldsymbol{\nu}}$;
\item $\mathrm{d}\alpha_{\boldsymbol{\nu}}
=2\,\overline{\pi}_{\boldsymbol{\nu}}^*(\overline{\omega}_{\boldsymbol{\nu}})$;
\item $\mu^{Y_{\boldsymbol{\nu}}}$ in (\ref{eqn:muonYnu}) descends to an action
$\mu^{\overline{M}_{\boldsymbol{\nu}}}:T^1_{\boldsymbol{\nu}}\times \overline{M}_{\boldsymbol{\nu}}
\rightarrow \overline{M}_{\boldsymbol{\nu}}$, which is holomorphic for $\overline{J}_{\boldsymbol{\nu}}$
and symplectic for $\overline{\omega}_{\boldsymbol{\nu}}$;
\item $\mu^{\overline{M}_{\boldsymbol{\nu}}}$ is Hamiltonian for $2\,\overline{\omega}_{\boldsymbol{\nu}}$,
with  moment map $\overline{\Phi}$;
\item $\mu^{Y_{\boldsymbol{\nu}}}$ is the contact lift of $(\mu^{\overline{M}_{\boldsymbol{\nu}}},
\overline{\Phi})$. 
\end{enumerate}

\end{prop} 

In other words, the description of $\widehat{M}_{\boldsymbol{\nu}}$ can be abstractly reduced to the
case $r=1$, with $M$ replaced by $\overline{M}_{\boldsymbol{\nu}}$ and 
$X$ by $Y_{\boldsymbol{\nu}}$. We need to describe how to transfer the toric structure to this
quotient picture.

\subsection{The toric structure of $\overline{M}_{\boldsymbol{\nu}}$}

W aim to verify that the toric setting is preserved in the reduction process 
of Proposition 
\ref{prop:reduction 1 dimensional case}. To this end, let us consider the saturation
$$
\tilde{M}_{\boldsymbol{\nu}}:=\mathbb{T}^{r-1}_{\boldsymbol{\nu}^\perp}\cdot M_{\boldsymbol{\nu}};
$$
thus $\tilde{M}_{\boldsymbol{\nu}}$ is the set of (semi)-stable points in $M$ for the
action of the complexification $\mathbb{T}^{r-1}_{\boldsymbol{\nu}^\perp}$ of
$T^{r-1}_{\boldsymbol{\nu}^\perp}$, with respect to the linearization 
induced by $\Phi$ (hence to the lift $\mu^X$).

Therefore, $m\in \tilde{M}_{\boldsymbol{\nu}}$ if and only if there exists 
a $T^{r-1}_{\boldsymbol{\nu}^\perp}$-invariant holomorphic section $\sigma$ 
of $A^{\otimes k}$, for some
$k\ge 1$, such that $\sigma (m)\neq 0$. Equivalently, $m\in \tilde{M}_{\boldsymbol{\nu}}$  if and only
if for some, and therefore for any, $x\in \pi^{-1}(m)\subseteq X$ there exists a CR function
$\tilde{\sigma}\in H(X)_k$ which is $T^{r-1}_{\boldsymbol{\nu}^\perp}$-invariant 
under $\mu^X$, and satisfies $\tilde{\sigma}(x)\neq 0$.
Here, $T^{r-1}_{\boldsymbol{\nu}^\perp}$-invariance means that 
\begin{equation}
\label{eqn:invariance spelledout}
\tilde{\sigma}^g=\tilde{\sigma}\quad
\forall\,g\in T^{r-1}_{\boldsymbol{\nu}^\perp} \quad \text{where}\quad \tilde{\sigma}^g:=
\tilde{\sigma}\circ \mu^X_{g^{-1}}.
\end{equation}

It is convenient to emphasize the holomorphic structure.
Recall that $\widetilde{\gamma}^M:\mathbb{T}^d\times M\rightarrow M$ 
denotes the complexification of $\gamma^M:T^d\times M\rightarrow M$;
similarly, we have complexified bundle actions 
$\widetilde{\gamma}^{A^\vee_0}:\mathbb{T}^d\times A^\vee_0\rightarrow A^\vee_0$ 
(the complexification of $\gamma^X$) and 
$\widetilde{\mu}^{A^\vee_0}:\mathbb{T}^r\times A^\vee_0\rightarrow A^\vee_0$
(the complexification of $\mu^X$).
Accordingly, we have associated linear representations of $\mathbb{T}^r$ and $\mathbb{T}^d$
on each space of global holomorphic sections $H^0(M,A^{\otimes k})$, $k=0,1,2,\ldots$. In fact,
$H^0(M,A^{\otimes k})$ is canonically isomorphic with the space of holomorphic functions 
on $A^\vee_0$ that are homogeneous of degree $k$, $\mathcal{H}_k(A^\vee_0)\subset \mathcal{O}(A^\vee_0)$,
and given $\hat{\sigma}\in \mathcal{H}_k(A^\vee_0)$ and $g\in \mathbb{T}^r$ we set
\begin{equation}
\label{eqn:invariance Tr}
\hat{\sigma}^g:=\hat{\sigma}\circ \widetilde{\mu}^{A^\vee_0}_{g^{-1}}.
\end{equation}
The correspondences
$$
\hat{\sigma}\in \mathcal{H}_k(A^\vee_0)\mapsto
\tilde{\sigma}\in H(X)_k\mapsto \sigma\in H^0(M,A^{\otimes k})
$$
are natural and equivariant isomorphisms.
Therefore, $m\in \tilde{M}_{\boldsymbol{\nu}}$  if and only
if for some, and therefore for any, $\ell\in A^\vee_0$ lying over $m$ there exists 
$\hat{\sigma}\in \mathcal{H}_k(A^\vee_0)$ which is $\mathbb{T}^{r-1}_{\boldsymbol{\nu}^\perp}$-invariant 
under (\ref{eqn:invariance Tr}), and satisfies $\hat{\sigma}(\ell)\neq 0$.

\begin{lem}
\label{lem:tildeMinvariance}
$\tilde{M}_{\boldsymbol{\nu}}$ is $\mathbb{T}^d$-invariant, that is, 
$
\widetilde{\gamma}^M_t(\tilde{M}_{\boldsymbol{\nu}})=\tilde{M}_{\boldsymbol{\nu}}
\, \forall\,t\in \mathbb{T}^d.
$
\end{lem}

\begin{proof}
All actions involved commute. Suppose $m\in \tilde{M}_{\boldsymbol{\nu}}$ and let
$\hat{\sigma}\in \mathcal{H}_k(A^\vee_0)$ satisfy $\hat{\sigma}^g=\hat{\sigma}$
for all $g\in \mathbb{T}^r$,
and be such that 
$\tilde{\sigma}(\ell)\neq 0$ for some (hence any) $\ell\in A^\vee_0$ lying over $m$. 
Then for any $t\in \mathbb{T}^d$ we have
\begin{equation}
\label{eqn:tildesigmat}
0\neq \hat{\sigma}(\ell)=\hat{\sigma}\circ \tilde{\gamma}^{A^\vee_0}_{t^{-1}}
\circ \tilde{\gamma}^{A^\vee_0}_t(\ell).
\end{equation}
Clearly, $\hat{\sigma}\circ \widetilde{\gamma}^{A^\vee_0}_{t^{-1}}\in \mathcal{H}_k(A^\vee_0)$; furthermore, 
by the assumed invariance of $\hat{\sigma}$ we have for every
$g\in \mathbb{T}^{r-1}_{\boldsymbol{\nu}^\perp}$
\begin{equation}
\label{eqn:invariancepreserved}
\hat{\sigma}\circ \widetilde{\gamma}^{A^\vee_0}_{t^{-1}}\circ \widetilde{\mu}^{A^\vee_0}_{g^{-1}}
=\hat{\sigma}\circ \widetilde{\mu}^{A^\vee_0}_{g^{-1}}\circ \tilde{\gamma}^{A^\vee_0}_{t^{-1}}=
\hat{\sigma}\circ \widetilde{\gamma}^{A^\vee_0}_{t^{-1}}.
\end{equation}
Since $\widetilde{\gamma}^{A^\vee_0}_t(\ell)$ lies over $\tilde{\gamma}^{M}_t(m)$,
(\ref{eqn:tildesigmat}) and (\ref{eqn:invariancepreserved}) imply that
$\tilde{\gamma}^{M}_t(m)\in \tilde{M}_{\boldsymbol{\nu}}$.
\end{proof}

Recall that $M^0$ is the dense open subset where $\tilde{\gamma}^M$ is free and transitive.
Since $\tilde{M}_{\boldsymbol{\nu}}$ and $M^0$ are open and dense in $M$, 
$\tilde{M}_{\boldsymbol{\nu}}\cap M^0\neq \emptyset$. Therefore Lemma 
\ref{lem:tildeMinvariance} implies the following.

\begin{cor}
$\tilde{M}_{\boldsymbol{\nu}}\supseteq	 M^0$.
\end{cor}

As is well-known, we have a natural identification \cite{kir}, \cite{ness}, \cite{gs-gq}
\begin{equation}
\label{eqn:natural ssidenti}
\overline{M}_{\boldsymbol{\nu}}=M_{\boldsymbol{\nu}}/T^{r-1}_{\boldsymbol{\nu}^\perp}
\cong \tilde{M}_{\boldsymbol{\nu}}/\mathbb{T}^{r-1}_{\boldsymbol{\nu}^\perp},
\end{equation}
which will be left implicit in the following. Accordingly, we shall set
\begin{equation}
\label{eqn:M0}
\overline{M}_{\boldsymbol{\nu}}^0:=M^0/\mathbb{T}^{r-1}_{\boldsymbol{\nu}^\perp}\subseteq \overline{M}_{\boldsymbol{\nu}},
\end{equation}
an open and dense subset of $\overline{M}_{\boldsymbol{\nu}}$.

Let us define the quotient tori
\begin{equation}
\label{eqn:quotient tori}
T^{d-r+1}_q:=T^d/T^{r-1}_{\boldsymbol{\nu}^\perp},\quad
\mathbb{T}^{d-r+1}_q:=\mathbb{T}^d/\mathbb{T}^{r-1}_{\boldsymbol{\nu}^\perp};
\end{equation}
clearly $\mathbb{T}^{d-r+1}_q$ is the complexification of $T^{d-r+1}_q$.
Then there are induced quotient actions 
$$
\gamma^{\overline{M}_{\boldsymbol{\nu}}}:
T^{d-r+1}_q\times \overline{M}_{\boldsymbol{\nu}}\rightarrow
\overline{M}_{\boldsymbol{\nu}}, \quad \widetilde{\gamma}
^{\overline{M}_{\boldsymbol{\nu}}}:
\mathbb{T}^{d-r+1}_q\times \overline{M}_{\boldsymbol{\nu}}\rightarrow
\overline{M}_{\boldsymbol{\nu}},
$$
and $\widetilde{\gamma}
^{\overline{M}_{\boldsymbol{\nu}}}$ is the complexification of 
$\gamma^{\overline{M}_{\boldsymbol{\nu}}}$.
Furthermore, $T^1_{\boldsymbol{\nu}}\leqslant T^{d-r+1}_q$ is a 1-dimensional
subtorus
(notation as in 
(\ref{eqn:muonYnu}) and 
Proposition \ref{prop:reduction 1 dimensional case}), and the action 
$\mu^{\overline{M}_{\boldsymbol{\nu}}}$ in Proposition \ref{prop:reduction 1 dimensional case}
is the restriction of $\gamma^{\overline{M}_{\boldsymbol{\nu}}}$ to $T^1_{\boldsymbol{\nu}}$.

\begin{prop}
\label{prop:quotient torus}
$\widetilde{\gamma}
^{\overline{M}_{\boldsymbol{\nu}}}$ is free and transitive on
$\overline{M}_{\boldsymbol{\nu}}^0$.
\end{prop}

Before giving the proof let us interject some pieces of notation.

\begin{enumerate}
\item Let us choose a complementary subtorus
$\widehat{T}^{1}_{\boldsymbol{\nu}}\leqslant T^r$ to $T^{r-1}_{\boldsymbol{\nu}^\perp}$, so that 
$T^r\cong \widehat{T}^{1}_{\boldsymbol{\nu}}\times T^{r-1}_{\boldsymbol{\nu}^\perp}$;
projecting yields an isomorphism
$\widehat{T}^{1}_{\boldsymbol{\nu}}\cong T^{1}_{\boldsymbol{\nu}}$.
Having chosen $\widehat{T}^{1}_{\boldsymbol{\nu}}$, there is a unique 
primitive $\widetilde{\boldsymbol{\nu}}\in L(\widehat{T}^{1}_{\boldsymbol{\nu}})$ such that
$\boldsymbol{\nu}(\widetilde{\boldsymbol{\nu}})=1$. 
Correspondingly, we have isomorphisms
$T^r\cong \widehat{T}^{1}_{\boldsymbol{\nu}}\times  T^{r-1}_{\boldsymbol{\nu}^\perp}$,
$L(T^r)\cong \mathbb{Z}\,\widetilde{\boldsymbol{\nu}}\oplus L(T^{r-1}_{\boldsymbol{\nu}^\perp})$,
and dually $L(T^r)^\vee\cong \mathbb{Z}\,{\boldsymbol{\nu}}\oplus L(T^{r-1}_{\boldsymbol{\nu}^\perp})^\vee$.
\item Let us choose a complementary
subtorus 
$T^{d-r}_c\leqslant T^d$ to $T^r$, so that
\begin{equation}
\label{eqn:factorizationT1}
T^d\cong T^{d-r}_c\times \widehat{T}^{1}_{\boldsymbol{\nu}}\times  T^{r-1}_{\boldsymbol{\nu}^\perp},
\quad \mathbb{T}^d\cong \mathbb{T}^{d-r}_c\times \widehat{\mathbb{T}}^{1}_{\boldsymbol{\nu}}\times  \mathbb{T}^{r-1}_{\boldsymbol{\nu}^\perp}.
\end{equation}
Then
\begin{equation}
\label{eqn:lattice decomp nu}
L(T^d)\cong L(T^{d-r}_c)\oplus \mathbb{Z}\,\widetilde{\boldsymbol{\nu}}\oplus 
L(T^{r-1}_{\boldsymbol{\nu}^\perp})
\end{equation}
and dually
\begin{equation}
\label{eqn:lattice decomp nu dual}
L(T^d)^\vee\cong L(T^{d-r}_c)^\vee\oplus \mathbb{Z}\,{\boldsymbol{\nu}}\oplus 
L(T^{r-1}_{\boldsymbol{\nu}^\perp})^\vee.
\end{equation}
\item Projection induces isomorphisms 
\begin{equation}
\label{eqn:factorTd-r+1}
{T}^{d-r+1}_c:={T}^{d-r}_c\times 
\widehat{{T}}^{1}_{\boldsymbol{\nu}}
\cong {T}^{d-r+1}_q,\quad\mathbb{T}^{d-r+1}_c:=\mathbb{T}^{d-r}_c\times 
\widehat{\mathbb{T}}^{1}_{\boldsymbol{\nu}}
\cong \mathbb{T}^{d-r+1}_q;
\end{equation}
we shall denote by $\mathbb{T}^{d-r}_q
\leqslant \mathbb{T}^{d-r+1}_q$ the image of $\mathbb{T}^{d-r}_c$, so that 
$\mathbb{T}^{d-r+1}_q\cong \mathbb{T}^{d-r}_q\times \mathbb{T}^{1}_{\boldsymbol{\nu}}$.
\item If $t\in \mathbb{T}^{d-r+1}_c$, we shall
denote by $\overline{t}\in \mathbb{T}^{d-r+1}_q$ its image, and for any 
$m\in \tilde{M}_{\boldsymbol{\nu}}$ we shall denote by $\overline{m}\in \overline{M}_{\boldsymbol{\nu}}$
its projection. 
\end{enumerate}

\begin{proof}
Suppose $\overline{m'},\,\overline{m''}\in \overline{M}_{\boldsymbol{\nu}}^0$,
and choose $m',\,m''\in \tilde{M}_{\boldsymbol{\nu}}^0$ lying over them. Then there exist a unique
$t\in \mathbb{T}^d$ such that $m''=\tilde{\gamma}^M_t(m')$. Let us factor
$t=t_1\,t_2\,t_3$ according to (\ref{eqn:factorizationT1}), that is, 
$t_1\in \mathbb{T}^{d-r}_c$, $t_2\in \widehat{\mathbb{T}}^{1}_{\boldsymbol{\nu}}$, $t_3\in \mathbb{T}^{r-1}_{\boldsymbol{\nu}^\perp}$. Hence
$$
m''=\tilde{\gamma}^M_{t_1}\circ\tilde{\gamma}^M_{t_2}\circ \tilde{\gamma}^M_{t_3} (m)
\quad\Rightarrow\quad \overline{m''}=\widetilde{\gamma}
^{\overline{M}_{\boldsymbol{\nu}}}_{\overline{t}_1}\circ
\widetilde{\gamma}
^{\overline{M}_{\boldsymbol{\nu}}}_{\overline{t}_2} (\overline{m'})
=\widetilde{\gamma}
^{\overline{M}_{\boldsymbol{\nu}}}_{\overline{t}_1\,\overline{t}_2} (\overline{m'}).
$$
This establishes that 
$\widetilde{\gamma}
^{\overline{M}_{\boldsymbol{\nu}}}$ is transitive on
$\overline{M}_{\boldsymbol{\nu}}^0$.

Suppose $\overline{m}\in \overline{M}_{\boldsymbol{\nu}}^0$, $\overline{t}\in \mathbb{T}^{d-r+1}_q$
and $\overline{m}=\widetilde{\gamma}
^{\overline{M}_{\boldsymbol{\nu}}}_{\overline{t}} (\overline{m})$.
Let us choose $m\in \tilde{M}_{\boldsymbol{\nu}}^0$ lying over $m$ and 
$t\in \mathbb{T}_c^{d-r+1}$ lying over $\overline{t}$.  
Then there exists $t'\in \mathbb{T}^{r-1}_{\boldsymbol{\nu}^\perp}$ such that
$m=\widetilde{\gamma}
^{M}_{t\,t'} (m)$; hence $t\,t'=1$ and therefore $t=t'=1$, so $\overline{t}=1$.
In conclusion,
$\widetilde{\gamma}
^{\overline{M}_{\boldsymbol{\nu}}}$ is free on
$\overline{M}_{\boldsymbol{\nu}}^0$.
\end{proof}

\begin{cor}
$(\overline{M}_{\boldsymbol{\nu}},\,\overline{J}_{\boldsymbol{\nu}})$ is a toric projective manifold.
\end{cor}

We can similarly recover the structure of a toric symplectic manifold, as follows.
Let us choose $\tilde{\boldsymbol{\delta}}\in {\mathfrak{t}^d}^\vee$ such that 
$\boldsymbol{\delta}=\iota^t(\tilde{\boldsymbol{\delta}})$ as in (\ref{eqn:Trpsi})
and (\ref{eqn:Trpsi1}).

\begin{defn}
Given a subspace $\mathfrak{b}\subseteq \mathfrak{t}^d$,
we shall denote by $\mathfrak{b}^0\subseteq {\mathfrak{t}^d}^\vee$ its annhilator in
${\mathfrak{t}^d}^\vee$. 
\end{defn}

Hence 
${\boldsymbol{\nu}^\perp}^0\subseteq
{\mathfrak{t}^d}^\vee$ is a vector subspace of dimension $d-r+1$,
and 
\begin{equation}
\label{eqn:nuperp0nu}
{\boldsymbol{\nu}^\perp}^0=({\iota^t})^{-1}\big(\mathrm{span}(\boldsymbol{\nu})\big).
\end{equation}
%
%
%
Thus with notation as in (\ref{eqn:Trpsi1})
\begin{equation}
\label{eqn:alternative char}
M_{\boldsymbol{\nu}}=\Psi_{\tilde{\boldsymbol{\delta}}}^{-1}({\boldsymbol{\nu}^\perp}^0);
\end{equation}
hence $\left.\Psi_{\tilde{\boldsymbol{\delta}}}\right|_{M_{\boldsymbol{\nu}}}$ is
an equivariant map $M_{\boldsymbol{\nu}}\rightarrow {\boldsymbol{\nu}^\perp}^0
\cong {\mathfrak{t}^{d-r+1}_q}^\vee$. Therefore, $\left.\Psi_{\tilde{\boldsymbol{\delta}}}\right|_{M_{\boldsymbol{\nu}}}$ passes to the quotient and yields a $T^{d-r+1}_q$-equivariant map
$\overline{\Psi}_{\tilde{\boldsymbol{\delta}}}:\overline{M}_{\boldsymbol{\nu}}
\rightarrow {\mathfrak{t}^{d-r+1}_q}^\vee$, which 
is a moment map for $\gamma^{\overline{M}_{\boldsymbol{\nu}}}$.

We conclude the following.

\begin{lem}
\label{lem:sympl toric orbifold}
$(\overline{M}_{\boldsymbol{\nu}},\,
2\,\omega_{\boldsymbol{\nu}},\,\overline{\Psi}_{\tilde{\boldsymbol{\delta}}})$ is a 
symplectic toric orbifold \cite{lt}. 
Furthermore, the moment map $\overline{\Phi}$ in Proposition
\ref{prop:reduction 1 dimensional case} is induced by $\overline{\Psi}_{\tilde{\boldsymbol{\delta}}}$.

\end{lem}

\begin{rem}
Distinct choices of $\tilde{\boldsymbol{\delta}}$ determine distinct moment maps
$\overline{\Psi}_{\tilde{\boldsymbol{\delta}}}$, differing by a constant 
in ${\mathfrak{t}^r}^0\subseteq {\boldsymbol{\nu}^\perp}^0\cong {\mathfrak{t}^{d-r+1}_q}^\vee$.
\end{rem}

\subsection{The reduced moment polytope $\overline{\Delta}_{\boldsymbol{\nu}}$}
\label{sctn:momentpolytope}

We aim to clarify the relation between
the moment polytope $\overline{\Delta}_{\boldsymbol{\nu}}$ 
of $\overline{M}_{\boldsymbol{\nu}}$ and the moment polytope
$\Delta$ of $M$, and to interpret properties of $\gamma$ and $\Psi$
in terms of $\Delta$ and $\overline{\Delta}_{\boldsymbol{\nu}}$.

In view of (\ref{eqn:alternative char}) and the identification 
$ {\mathfrak{t}^{d-r+1}_q}^\vee\cong {\boldsymbol{\nu}^\perp}^0$,
\begin{equation}
\label{eqn:intersection of polytopes}
\overline{\Delta}_{\boldsymbol{\nu}}=
\overline{\Psi}_{\tilde{\boldsymbol{\delta}}}(\overline{M}_{\boldsymbol{\nu}})
=\Psi_{\tilde{\boldsymbol{\delta}}}(M_{\boldsymbol{\nu}})
=(\Delta+\tilde{\boldsymbol{\delta}})\cap {\boldsymbol{\nu}^\perp}^0.
\end{equation}

%



With notation as in (\ref{eqn:Trpsi}), (\ref{eqn:Trpsi1}),
and in view of Definition \ref{defn:transverse polytopes}, we have the following.

\begin{prop}
\label{prop:transverse polytope}
Suppose that $\Phi(m)\neq \mathbf{0}$ for every $m\in M$, and 
$\Phi^{-1}(\mathbb{R}_+\cdot \boldsymbol{\nu})\neq \emptyset$. 
The following conditions are equivalent:
\begin{enumerate}
\item $\Phi$ is transverse to $\mathbb{R}_+\cdot \boldsymbol{\nu}$;
\item $\Psi_{\tilde{\boldsymbol{\delta}}}$ is transverse to ${\boldsymbol{\nu}^\perp}^0$;
\item $\Psi$ is tansverse to  ${\boldsymbol{\nu}^\perp}^0-\tilde{\boldsymbol{\delta}}$;
\item ${\boldsymbol{\nu}^\perp}^0$ and $\Delta+\tilde{\boldsymbol{\delta}}$ meet transversely
in ${\mathfrak{t}^d}^\vee$;
\item ${\boldsymbol{\nu}^\perp}^0-\tilde{\boldsymbol{\delta}}$ and $\Delta$ meet transversely
in ${\mathfrak{t}^d}^\vee$;
\item if $T_1,\ldots,T_a\leqslant T^d$ are the (distinct) compact tori stabilizing some point
of $M_{\boldsymbol{\nu}}$, then the projection 
$\pi_q:T^d\rightarrow T^{d-r+1}_q$ restricts to a finite map $T_j\rightarrow \pi_q(T_j)$ for $j=1,\ldots,a$;
\item $T_j\cap T^{r-1}_{\boldsymbol{\nu}^\perp}$ is finite for $j=1,\ldots,a$.
\end{enumerate}

\end{prop}

The proof of Proposition \ref{prop:transverse polytope} rests on the following property
of the moment map $\Psi$ of a toric symplectic manifold (see \cite{del}, \cite{g-toric}, 
\cite{lt}). Let $\Delta$ be the moment polytope and $F$ be a face of $\Delta$.
If $\boldsymbol{\xi}\in F^0$
$m\in 
\Psi^{-1}(\boldsymbol{\xi})$, then 
\begin{equation}
\label{eqn:tgspaceface}
\mathrm{d}_m\Psi(T_mM)=T_{\boldsymbol{\xi}}F^0.
\end{equation}

\begin{proof}
To begin with, by the hypothesis and the convexity of
$\Phi(M)$ \cite{gs-conv}, 
$\Phi^{-1}(\mathbb{R}_+\cdot \boldsymbol{\nu})=\Phi^{-1}\big( \mathrm{span}(\boldsymbol{\nu}) \big)$.

The equivalence of 1. and 2. follows from (\ref{eqn:nuperp0nu}).
That 2. is equivalent to 3. and that 4. is equivalent to 5. is obvious,
as is the equivalence of 6. and 7., given that 
$T^{d-r+1}_q=T^d/T^{r-1}_{\boldsymbol{\nu}^\perp}$.
On the other hand, 7. is equivalent to $T^{r-1}_{\boldsymbol{\nu}^\perp}$ acting locally freely on
$M_{\boldsymbol{\nu}}$, and this condition is equivalent to 1..

Thus it suffices to show that 
2. is equivalent to 4. Let us adopt $\Psi_{\tilde{\boldsymbol{\delta}}}$  as moment map
for the action of $T^d$ on $(M,2\,\omega)$, so that $\Delta_{\tilde{\boldsymbol{\delta}}}:=\Delta+\tilde{\boldsymbol{\delta}}$ is the corresponding moment polytope.
By definition, $\Psi_{\tilde{\boldsymbol{\delta}}}$ is transverse to ${\boldsymbol{\nu}^\perp}^0$ if and only if
for every $m\in \Psi_{\tilde{\boldsymbol{\delta}}}^{-1}({\boldsymbol{\nu}^\perp}^0)$ we have
\begin{equation}
\label{eqn:tansvcondition}
\mathrm{d}_m\Psi_{\tilde{\boldsymbol{\delta}}}(T_mM)+{\boldsymbol{\nu}^\perp}^0
={\mathfrak{t}^d}^\vee.
\end{equation}
Let $F$ be a face of $\Delta_{\tilde{\boldsymbol{\delta}}}$ such that
$F^0\cap {\boldsymbol{\nu}^\perp}^0\neq \emptyset$.
If $\boldsymbol{\xi}\in F^0\cap {\boldsymbol{\nu}^\perp}^0$
and $m\in \Psi_{\tilde{\boldsymbol{\delta}}}^{-1}(\boldsymbol{\xi})$ then by (\ref{eqn:tgspaceface}) and
(\ref{eqn:tansvcondition}) (with $\Psi$ replaced by $\Psi_{\tilde{\boldsymbol{\delta}}}$)
$$
T_{\boldsymbol{\xi}}F^0+{\boldsymbol{\nu}^\perp}^0=
\mathrm{d}_m\Psi_{\tilde{\boldsymbol{\delta}}}(T_mM)+{\boldsymbol{\nu}^\perp}^0=
{\mathfrak{t}^d}^\vee.
$$
Hence 2. implies 4.
The argument for the reverse implication is similar.
\end{proof}

\begin{prop}
\label{prop:stabilizer subgroup}
With the notation in Proposition \ref{prop:transverse polytope}, the closed subtori
$\overline{T}_j:=\pi_q(T_j)\leqslant T^{d-r+1}_q$, $j=1,\ldots,a$, are the subgroups
that appear as stabilizer subgroups of points in $\overline{M}_{\boldsymbol{\nu}}$.
\end{prop}

\begin{proof}
If $T_j\leqslant T^d$ is the stabilizer subgroup of $m\in M_{\boldsymbol{\nu}}$, 
by equivariance of the projection
$M_{\boldsymbol{\nu}}\rightarrow \overline{M}_{\boldsymbol{\nu}}$
clearly
$\overline{T}_j\leqslant T^{d-r+1}_q$ is a closed subtorus stabilizing $\overline{m}$
(the image of $m$ in $\overline{M}_{\boldsymbol{\nu}}$).
Thus, if $S\leqslant T^{d-r+1}_q$ is the stabilizer of $\overline{m}$, then 
$\overline{T}_j\leqslant S$, and $S$ is a torus \cite{lt}. 

Given the isomorphism $T^d\cong T^{d-r+1}_c\times T^{r-1}_{\boldsymbol{\nu}^\perp}
\cong T^{d-r+1}_q\times T^{r-1}_{\boldsymbol{\nu}^\perp}$, we can
lift $S$ to a subgroup $S'=S\times \{1\}$ of $T^d$. Again by equivariance,
for every 
$s\in S'$ there exist finitely many $t\in T^{r-1}_{\boldsymbol{\nu}}$ such that
$\mu^M_{s\,t}(m)=m$. The collection $\tilde{S}$ of all such pairs
$(s,t)\in S'\times T^{r-1}_{\boldsymbol{\nu}}$ is a closed subgroup of $T^d$ of the
same dimension as $S$, stabilizing
$m$ (hence contained in $T_j$) 
and projecting onto $S$ in $T^{d-r+1}_q$; conversely, any element of $T^d$ stabilizing $m$
must have this form and therefore $\tilde{S}=T_j$. It follows that $S=\overline{T}_j$.

Hence the subgroups $\overline{T}_j\leqslant T^{d-r+1}_q$ are all the stabilizer subgroups of points in
$\overline{M}_{\boldsymbol{\nu}}$.

\end{proof}

Let $F$ be a facet of $\Delta$, so that $F+\tilde{\boldsymbol{\delta}}$
is a facet of $\Delta+\tilde{\boldsymbol{\delta}}$. Let $M_F:=\Psi^{-1}(F)=
\Psi_{\tilde{\boldsymbol{\delta}}}^{-1}(F+\tilde{\boldsymbol{\delta}})$,
$M_F^0:=\Psi^{-1}(F^0)$.
If $\boldsymbol{\upsilon}$ is an inward primitive normal vector to $F$, then $M_F^0$ is the locus
of points in $M$ having stabilizer the $1$-dimensional torus $S_F\leqslant T^d$
generated by $\boldsymbol{\upsilon}$. Furthermore, $M_F^0$ 
is open and dense in the $1$-codimensional complex
submanifold $M_F$ of $F$. 

\begin{prop}
\label{prop:MjMFtransverse}
Assume that the equivalent conditions of Proposition \ref{prop:transverse polytope}
are satisfied, and let $F$ be a facet of $\Delta$. Then:
\begin{enumerate}
\item $(F+\tilde{\boldsymbol{\delta}})\cap {\boldsymbol{\nu}^\perp}^0=
\Psi_{\tilde{\boldsymbol{\delta}}}(M_F\cap M_{\boldsymbol{\nu}})$, and in particular
$M_F\cap M_{\boldsymbol{\nu}}\neq \emptyset$ if and only if
$(F+\tilde{\boldsymbol{\delta}})\cap {\boldsymbol{\nu}^\perp}^0\neq \emptyset$;
\item if $M_F\cap M_{\boldsymbol{\nu}}\neq \emptyset$, then the intersection is transverse
in $M$.
\end{enumerate}
\end{prop}

\begin{proof}
[Proof of 1.] Suppose $m\in M_F\cap M_{\boldsymbol{\nu}}$. Then 
$\Psi(m)\in F$ (since $m\in M_F$), whence $\Psi_{\tilde{\boldsymbol{\delta}}}(m)\in 
F+\tilde{\boldsymbol{\delta}}$; 
on the other hand $\Phi (m)\in \mathbb{R}\,\boldsymbol{\nu}$ (since $m\in M_{\boldsymbol{\nu}}$),
hence $\Psi_{\tilde{\boldsymbol{\delta}}}(m)\in {\boldsymbol{\nu}^\perp}^0$. Thus,
$\Psi_{\tilde{\boldsymbol{\delta}}}(m)\in (F+\tilde{\boldsymbol{\delta}})\cap{\boldsymbol{\nu}^\perp}^0$.

Conversely, suppose $\gamma\in (F+\tilde{\boldsymbol{\delta}})\cap{\boldsymbol{\nu}^\perp}^0$.
Thus there exists $m\in M$ such that $\gamma=\Psi_{\tilde{\boldsymbol{\delta}}}(m)
\in F+\tilde{\boldsymbol{\delta}}$ (i.e., $m\in M_F$), and 
$\Psi_{\tilde{\boldsymbol{\delta}}}(m)\in {\boldsymbol{\nu}^\perp}^0$, i.e. $m\in M_{\boldsymbol{\nu}}$.
Thus $m\in M_F\cap {\boldsymbol{\nu}^\perp}^0$, whence $\gamma\in \Psi_{\tilde{\boldsymbol{\delta}}}(M_F\cap M_{\boldsymbol{\nu}})$.
\end{proof}

Before giving the proof of 2., a remark is in order. The holomorphic and 
Hamiltonian action $(\gamma^M,\,\Psi_{\tilde{\boldsymbol{\delta}}})$ of $T^d$ on $(M,2\,\omega)$ 
restricts to a holomorphic and Hamiltonian action
$(\lambda^M,\Lambda)$ of $T^{r-1}_{\boldsymbol{\nu}^\perp}$, where the moment map
$\Lambda:M\rightarrow {\mathfrak{t}^{r-1}_{\boldsymbol{\nu}^\perp}}^\vee$ is induced by
$\Psi_{\tilde{\boldsymbol{\delta}}}$ in the standard manner. 
Then $M_{\boldsymbol{\nu}}=\Lambda^{-1}(\mathbf{0})$ and the transversality hypothesis
in Proposition \ref{prop:transverse polytope} are equivalent to the condition that
$\mathbf{0}$ be a regular value of $\Lambda$, or - still equivalently - that 
$T^{r-1}_{\boldsymbol{\nu}^\perp}$ act locally freely on $M_{\boldsymbol{\nu}}$.

\begin{proof}
[Proof of 2.]
$M_F$ is a K\"{a}hler submanifold of $(M,J,2\,\omega)$. 
It is furthermore $T^d$-invariant, hence
$\Lambda$ restricts to a moment map 
$\Lambda_F:M_F\rightarrow {\mathfrak{t}^{r-1}_{\boldsymbol{\nu}^\perp}}^\vee$ for the
action of $T^{r-1}_{\boldsymbol{\nu}^\perp}$ on $M_F$. Transversality of $M_F$ and 
$M_{\boldsymbol{\nu}}$ is then equivalent to $\mathbf{0}$ being a regular value for
$\Lambda_F$, hence to $T^{r-1}_{\boldsymbol{\nu}^\perp}$ acting locally freely on
$\Lambda_F^{-1}(0)$. 
However $T^{r-1}_{\boldsymbol{\nu}^\perp}$ does act
locally freely on $\Lambda_F^{-1}(0)=M_{\boldsymbol{\nu}}\cap M_F$, since it acts locally
freely on all of $M_{\boldsymbol{\nu}}$.
 
\end{proof}

Assume that the conditions in Proposition \ref{prop:transverse polytope}
are satisfied, and let $F$ be a facet of $\Delta$ such that 
$F_{\boldsymbol{\nu}}:=(F+\tilde{\boldsymbol{\delta}})\cap {\boldsymbol{\nu}^\perp}^0\neq \emptyset$.
Then $F_{\boldsymbol{\nu}}$ is a facet of $\overline{\Delta}_{\boldsymbol{\nu}}$.
Furthermore, if 
$\overline{M}_F:=(M_F\cap M_{\boldsymbol{\nu}})/T^{r-1}_{\boldsymbol{\nu}^\perp}\subseteq
\overline{M}_{\boldsymbol{\nu}}$, we can draw the following conclusion from the previous discussion.

\begin{cor}
$\overline{M}_F$ is a complex suborbifold of $\overline{M}_{\boldsymbol{\nu}}$, and 
$\overline{M}_F=\overline{\Psi}_{\tilde{\boldsymbol{\delta}}}^{-1}(F_{\boldsymbol{\nu}})$.
\end{cor}

\subsection{Smoothness conditions on $\overline{\Delta}_{\boldsymbol{\nu}}$}
\label{sctn:smoothnessDelta}

Proposition \ref{prop:transverse polytope} characterizes the transversality of $\Phi$ to
$\mathbb{R}\,\boldsymbol{\nu}$ in terms of the mutual position of $\Delta$ and 
${\boldsymbol{\nu}^\perp}^0$ in ${\mathfrak{t}^d}^\vee$. 
This condition ensures that $M_{\boldsymbol{\nu}}$ is a submanifold,
that $T^{r-1}_{\boldsymbol{\nu}}$ acts locally freely on it, and therefore that
$\overline{M}_{\boldsymbol{\nu}}$ is a K\"{a}hler orbifold.
Since our present focus is on the case where $\overline{M}_{\boldsymbol{\nu}}$ is a K\"{a}hler 
manifold, we want to similarly characterize this stronger condition using $\Delta$ and 
${\boldsymbol{\nu}^\perp}^0$. 

By the discussion in \S \ref{sctn:transverse polytopes}, (\ref{eqn:intersection of polytopes}), 
and Proposition \ref{prop:transverse polytope},
$\overline{\Delta}_{\boldsymbol{\nu}}$ is the convex polytope in ${\boldsymbol{\nu}^\perp}^0$
having as facets the non-empty intersections of ${\boldsymbol{\nu}^\perp}^0$ with the 
facets of $\Delta+\tilde{\boldsymbol{\delta}}$.
Equivalently, it is the convex hull of the intersection of
${\boldsymbol{\nu}^\perp}^0$ with the $(d-r+1)$-codimensional (i.e., $(r-1)$-dimensional)
faces of $\Delta+\tilde{\boldsymbol{\delta}}$. The connected component of such 
intersection are precisely verteces of $\overline{\Delta}_{\boldsymbol{\nu}}$; furthermore, if
$F\subseteq \Delta$ is an $(r-1)$-dimensional face, then 
$(F+\tilde{\boldsymbol{\delta}})\cap {\boldsymbol{\nu}^\perp}^0=(F^0+\tilde{\boldsymbol{\delta}})\cap {\boldsymbol{\nu}^\perp}^0$, since by transversality ${\boldsymbol{\nu}^\perp}^0$ must have empty
intersection with any face of lesser dimension.

Let $\mathcal{G}(\Delta)=\{F_1,\ldots,F_k\}$ be the collection of facets of
$\Delta$, so that the collection of facets of $\Delta+\tilde{\boldsymbol{\delta}}$ is
$\mathcal{G}(\Delta+\tilde{\boldsymbol{\delta}})=\{F_1+\tilde{\boldsymbol{\delta}},\ldots,F_k+\tilde{\boldsymbol{\delta}}\}$.
Thus for every $j=1,\ldots,s$ there exist unique $\boldsymbol{\upsilon}_j\in L(T^d)$ primitive
and $\lambda_j\in \mathbb{R}$ such that
\begin{equation}
\label{eqn:Delta inters delta}
\Delta+\tilde{\boldsymbol{\delta}}=\bigcap_{j=1}^k\{\ell \in \mathfrak{t}^\vee\,:\,
\ell (\boldsymbol{\upsilon}_j)\ge \lambda_j+\delta_j\},
\quad \delta_j:=\tilde{\boldsymbol{\delta}}(\boldsymbol{\upsilon}_j).
\end{equation}

Let us assume that the $F_j$'s have been so numbered that
$(F_j+\tilde{\boldsymbol{\delta}})\cap {\boldsymbol{\nu}^\perp}^0
\neq \emptyset$ for $j=1,\ldots,l$, and 
$(F_j+\tilde{\boldsymbol{\delta}})\cap {\boldsymbol{\nu}^\perp}^0
= \emptyset$ for $l+1\le j\le k$. 
Hence by (\ref{eqn:Delta inters delta}) and the previous discussion
(with the usual identification ${\mathfrak{t}^{d-r+1}_q}^\vee\cong {\boldsymbol{\nu}^\perp}^0$)
\begin{equation}
\label{eqn:Delta inters delta nuperp0}
{\mathfrak{t}^{d-r+1}_q}^\vee\supset
\overline{\Delta}_{\boldsymbol{\nu}}\cong\bigcap_{j=1}^l\{\gamma \in {\boldsymbol{\nu}^\perp}^0\,:\,
\gamma (\boldsymbol{\upsilon}_j)\ge \lambda_j+\delta_j\}.
\end{equation}

In view of (\ref{eqn:factorizationT}), (\ref{eqn:lattice decomp nu}), 
and (\ref{eqn:factorTd-r+1})
there exist unique
$\boldsymbol{\upsilon}_j'\in L(T^{d-r}_c)$, $\rho_j\in \mathbb{Z}$,
$\boldsymbol{\upsilon}_j''\in L(T^{r-1}_{\boldsymbol{\nu}^\perp})$
such that
\begin{equation}
\label{eqn:decompj}
\boldsymbol{\upsilon}_j=\boldsymbol{\upsilon}_j'+\rho_j\,\widetilde{\boldsymbol{\nu}}+
\boldsymbol{\upsilon}_j''.
\end{equation}
Therefore (\ref{eqn:decompj}) may be rewritten as
\begin{equation}
\label{eqn:Delta inters delta nuperp01}
\overline{\Delta}_{\boldsymbol{\nu}}\cong
\bigcap_{j=1}^l\{\gamma \in {\boldsymbol{\nu}^\perp}^0\,:\,
\gamma (\boldsymbol{\upsilon}_j'+\rho_j\,\widetilde{\boldsymbol{\nu}})\ge \lambda_j+\delta_j\}.
\end{equation}

\begin{defn}
For $I=\{i_1,\ldots,i_a\}\subseteq \{1,\ldots,l\}$, let 
$\mathfrak{s}_I:=\mathrm{span}(\boldsymbol{\upsilon}_j\,:\,j\in I\}\subseteq \mathfrak{t}^d$ and let 
$S_I\leqslant T^d$ be the closed subtorus with Lie subalgebra $\mathfrak{s}_I$.
\end{defn}

Suppose that $I$ is such that
$F:=F_{i_1}\cap \ldots\cap F_{i_a}$ is a face of $\Delta$; then, since
$\Delta$ is a Delzant polytope, the sequence of normal vectors
$(\boldsymbol{\upsilon}_{i_1},\ldots,\boldsymbol{\upsilon}_{i_a})$ is a primitive system
in $L(T^d)$ (meaning that it can be extended to a lattice basis).
Furthermore, $S_I$ is the stabilizer subgroup of any $m\in \Psi^{-1}(F^0)$.

In particular, let $F$ be a codimension-$a$ face of $\Delta$ such that 
$(F+\tilde{\boldsymbol{\delta}})\cap {\boldsymbol{\nu}^\perp}^0
\neq \emptyset$ (whence $(F^0+\tilde{\boldsymbol{\delta}})\cap {\boldsymbol{\nu}^\perp}^0
\neq \emptyset$ by Corollary \ref{cor:interior intersection}).
Then there is a unique $I_F=\{i_1,\ldots,i_a\}\subseteq \{1,\ldots,l\}$, such that
$F$ is the intersection of the facets $F_{i_1},\ldots,F_{i_a}$, and so
$(\boldsymbol{\upsilon}_{i_1},\ldots,\boldsymbol{\upsilon}_{i_a})$ is a primitive system.

\begin{lem}
\label{lem:primitive systems}
Under the previous assumption, and with notation
(\ref{eqn:decompj}),
the following conditions are equivalent:
\begin{enumerate}
\item $T^{r-1}_{\boldsymbol{\nu}^\perp}$ acts freely on 
$\Psi_{\tilde{\boldsymbol{\delta}}}^{-1}\big(
(F^0+\tilde{\boldsymbol{\delta}})\cap {\boldsymbol{\nu}^\perp}^0\big)
$;
\item $(\boldsymbol{\upsilon}_{i_1}'+\rho_{i_1}\,\widetilde{\boldsymbol{\nu}},
\ldots,\boldsymbol{\upsilon}_{i_a}'+\rho_{i_a}\,\widetilde{\boldsymbol{\nu}})$ is a
primitive system. 
 
\end{enumerate}

\end{lem}

\begin{proof}
Suppose $\gamma\in (F^0+\tilde{\boldsymbol{\delta}})\cap {\boldsymbol{\nu}^\perp}^0
$, and choose $m\in M$ such that 
$\Psi_{\tilde{\boldsymbol{\delta}}}(m)=\gamma$. Then $\Psi(m)\in F^0$, and so the stabilizer subgroup
of $m$ is $S_{I_F}$. Hence $m$ has trivial stabilizer in $T^{r-1}_{\boldsymbol{\nu}^\perp}$ 
if and only if $T^{r-1}_{\boldsymbol{\nu}^\perp}\cap S_{I_F}$ is trivial. 

Suppose that 1. holds, and let $\vartheta_j\in \mathbb{R}$, $j=1,\ldots,a$, 
be such that $$\exp\left(\sum_{j=1}^a\vartheta_j\,
(\boldsymbol{\upsilon}_{i_j}'+\rho_{i_j}\,\widetilde{\boldsymbol{\nu}})\right)=1.$$
Then 
$$\exp\left(\sum_{j=1}^a\vartheta_j\,
\boldsymbol{\upsilon}_{i_j}\right)=\exp\left(\sum_{j=1}^a\vartheta_j\,
\boldsymbol{\upsilon}_{i_j}''\right)\in 
S_{I_F}\cap  T^{r-1}_{\boldsymbol{\nu}^\perp}
=(1).$$
Thus necessarily $\vartheta_j\in 2\,\pi\,\mathbb{Z}$ because $(\boldsymbol{\upsilon}_{i_j})_j$
is a primitive system. Hence 2. holds.

Conversely, assume that 2. holds. Suppose that $t\in  S_{I_F}\cap  T^{r-1}_{\boldsymbol{\nu}^\perp}$.
There exist $\vartheta_j$, $j=1,\ldots,a$, and 
$\boldsymbol{\xi}\in \mathfrak{t}^{r-1}_{\boldsymbol{\nu}^\perp}$ such that 
\begin{eqnarray*}
\lefteqn{t=\exp\left(\sum_{j=1}^a\,\vartheta_j\,\boldsymbol{\upsilon}_j\right)=\exp(\boldsymbol{\xi})}\\
& \Rightarrow&\exp\left(\sum_{j=1}^a\vartheta_j\,(\boldsymbol{\upsilon}_{i_j}'+\rho_{i_j}\,\widetilde{\boldsymbol{\nu}})\right)=\exp\left( \boldsymbol{\xi}
-\sum_{j=1}^a\,\vartheta_j\,\boldsymbol{\upsilon}_j''\right)\\
&\Rightarrow&\exp\left(\sum_{j=1}^a\vartheta_j\,(\boldsymbol{\upsilon}_{i_j}'+\rho_{i_j}\,\widetilde{\boldsymbol{\nu}})\right)\in 
\left(T^{d-r}_c\times \widehat{T}^{1}_{\boldsymbol{\nu}}\right)\cap  T^{r-1}_{\boldsymbol{\nu}^\perp}
=(1)\\
&\Rightarrow&\vartheta_j\in 2\pi\,\mathbb{Z},\quad \forall\,j=1,\ldots,a\quad
\Rightarrow\quad t=1,
\end{eqnarray*}
where we have made use of (\ref{eqn:factorizationT}).
Hence 2. implies 1..

\end{proof}

This can be strengthened as follows.

\begin{prop}
\label{prop:free action Tr-1}
Under the previous assumptions, the following conditions are equivalent:
\begin{enumerate}
\item $T^{r-1}_{\boldsymbol{\nu}^\perp}$ acts freely on $M_{\boldsymbol{\nu}}$;
\item for every $(r-1)$-dimensional face $F$ of $\Delta$ such that
$(F+\tilde{\boldsymbol{\delta}})\cap {\boldsymbol{\nu}^\perp}^0
\neq \emptyset$, with $I_F=\{i_1,\ldots,i_{d-r+1}\}\subseteq \{1,\ldots,l\}$,
the sequence $$(\boldsymbol{\upsilon}_{i_1}'+\rho_{i_1}\,\widetilde{\boldsymbol{\nu}},
\ldots,\boldsymbol{\upsilon}_{i_{d-r+1}}'+\rho_{i_{d-r+1}}\,\widetilde{\boldsymbol{\nu}})$$
is a primitive system;
\item for every $(b+r-1)$-dimensional face $F$ of $\Delta$ (with $b\ge 0$) such that
$(F+\tilde{\boldsymbol{\delta}})\cap {\boldsymbol{\nu}^\perp}^0
\neq \emptyset$, with $I_F=\{i_1,\ldots,i_{d-b-r+1}\}\subseteq \{1,\ldots,l\}$,
the sequence $$(\boldsymbol{\upsilon}_{i_1}'+\rho_{i_1}\,\widetilde{\boldsymbol{\nu}},
\ldots,\boldsymbol{\upsilon}_{i_{d-b-r+1}}'+\rho_{i_{d-b-r+1}}\,\widetilde{\boldsymbol{\nu}})$$
is a primitive system.
\end{enumerate}
\end{prop}

\begin{proof}
That 1. implies 2. follows immediately from Lemma \ref{lem:primitive systems}.
Suppose that 2. holds. Let $m\in M_{\boldsymbol{\nu}}$. If $m\in M^0$ (i.e.,
$\Psi(m)\in \Delta^0$), then $T^d$ acts freely at $m$, hence so does $T^{r-1}_{\boldsymbol{\nu}^\perp}$.
Otherwise, $\Psi(m)\in F^0$ for a unique face $F$ of $\Delta$, whence
$\Psi_{\tilde{\boldsymbol{\delta}}}(m)\in 
(F^0+\tilde{\boldsymbol{\delta}})\cap {\boldsymbol{\nu}^\perp}^0$.
Applying again Lemma \ref{lem:primitive systems}, we conclude that $T^{r-1}_{\boldsymbol{\nu}^\perp}$
acts freely at $m$. Thus if 2. holds then $T^{r-1}_{\boldsymbol{\nu}^\perp}$ acts freely
at every $m\in \Psi_{\tilde{\boldsymbol{\delta}}}^{-1}({\boldsymbol{\nu}^\perp}^0)
=M_{\boldsymbol{\nu}}$, i.e. 1. holds.
That 3. implies 2. is obvious, since 2. is formally the special case of 3. with $b=0$. 
Suppose that 2. holds, and let $F$ be a $(b+r-1)$-dimensional face of
$\Delta$ as in the statement 
of 3; then $F_{\boldsymbol{\nu}}:=(F+\tilde{\boldsymbol{\delta}})\cap {\boldsymbol{\nu}^\perp}^0$ is a 
$b$-dimensional face of $\overline{\Delta}_{\boldsymbol{\nu}}$. Therefore $F_{\boldsymbol{\nu}}$
contains a vertex $\gamma$ of $\overline{\Delta}_{\boldsymbol{\nu}}$. Hence 
there exists an $(r-1)$-dimensional face $F'$ of $\Delta$, as in the statement of 2., such that
$\{\gamma\}=(F'+\tilde{\boldsymbol{\delta}})\cap {\boldsymbol{\nu}^\perp}^0$.
The sequence of normal vectors corresponding to $F'$ contains the sequence corresponding to
$F$, and since a subsystem of a primitive system is necessarily also primitive, we conclude that 
3. holds.

\end{proof}

\subsection{Proof of Theorem \ref{thm:case r general}}

We can now build on the previous discussion, we give the proof of the Theorem.

\begin{proof}
[Proof of Theorem \ref{thm:case r general}]
Under the given assumptions on $\Delta$ and ${\boldsymbol{\nu}^\perp}^0$,
$\overline{M}_{\boldsymbol{\nu}}$ is a toric manifold, acted upon by
${T}^{d-r+1}_q\cong {T}^{d-r+1}_c={T}^{d-r}_c\times 
\widehat{{T}}^{1}_{\boldsymbol{\nu}}$ in (\ref{eqn:factorTd-r+1}),
and with associated moment 
polytope $\overline{\Delta}_{\boldsymbol{\nu}}$.
Furthermore, by 
(\ref{eqn:overlineMnu1}) and the discussion in \S \ref{sctn:MnuYnu},
$Y_{\boldsymbol{\nu}}$ is the unit circle bundle associated to the positive line
bundle $(A_{\boldsymbol{\nu}},\,h_{\boldsymbol{\nu}})$ 
on $\overline{M}_{\boldsymbol{\nu}}$. In addition, 
$\widehat{M}_{\boldsymbol{\nu}}=Y_{\boldsymbol{\nu}}/T^1_{\boldsymbol{\nu}}$ by
(\ref{eqn:partialquotient/quotientgroup}), where $T^1_{\boldsymbol{\nu}}$
acts on $Y_{\boldsymbol{\nu}}$ by the contact lift 
$\mu^{Y_{\boldsymbol{\nu}}}$ of the Hamiltonian action $(\mu^{\overline{M}_{\boldsymbol{\nu}}},
\overline{\Phi})$ (Proposition \ref{prop:reduction 1 dimensional case}).

We are therefore
in the situation of Theorem \ref{thm:case r=1}, with the folowing replacements:
$M$ by $\overline{M}_{\boldsymbol{\nu}}$;
$T^d$ by ${T}^{d-r+1}_q\cong {T}^{d-r+1}_c={T}^{d-r}_c\times 
\widehat{{T}}^{1}_{\boldsymbol{\nu}}
$;
$\Delta$ by $\overline{\Delta}_{\boldsymbol{\nu}}$;
$X$ by $Y_{\boldsymbol{\nu}}$;
$\Psi$ by 
$\overline{\Psi}_{\tilde{\boldsymbol{\delta}}}$; 
$T^r$ by $T^1_{\boldsymbol{\nu}}\cong \widehat{{T}}^{1}_{\boldsymbol{\nu}}$; $\Phi$ by 
$\overline{\Phi}$; $T^{d-1}_c$ by ${T}^{d-r}_c$.
Furthermore, the constants $\lambda_j$ are replaced by
$\lambda_j+\delta_j$ for $j=1,\ldots,l$ in view of (\ref{eqn:Delta inters delta nuperp0}), 
and the scalar $\delta$ is taken to vanish by Lemma \ref{lem:sympl toric orbifold}
(once $\Psi$ has been replaced by $\overline{\Psi}_{\tilde{\boldsymbol{\delta}}}$, no further
translation is required).

The statement of Theorem \ref{thm:case r general} is now an immediate consequence of
Theorem \ref{thm:case r=1}.
\end{proof}


\begin{thebibliography}{Dillo99}

















\bibitem[D]{del} T.
Delzant, {\em 
Hamiltoniens p\'{e}riodiques et images convexes de l'application moment}. 
Bull. Soc. Math. France
\textbf{116}
 (1988), no. 3, 315–-339

\bibitem[F]{ful}
W. Fulton, {\em Introduction to toric varieties}. Annals of Mathematics Studies, 
\textbf{131}. The William H. Roever Lectures in Geometry. Princeton University Press, Princeton, NJ, 1993. xii+157 pp. ISBN: 0-691-00049-2 



\bibitem[G1]{g-JDG}
V. Guillemin, {\em 
Kaehler structures on toric varieties}. J. Differential Geom. 
\textbf{40} (1994), no. 2, 285–-309



\bibitem[G2]{g-toric}
V. Guillemin, {\em Moment maps and combinatorial invariants of Hamiltonian $T^n$-spaces}. Progress in Mathematics, 
\textbf{122}. Birkh\"{a}user Boston, Inc., Boston, MA, 1994. viii+150 pp. ISBN: 0-8176-3770-2


\bibitem[GS1]{gs-gq} V. Guillemin, S. Sternberg {\em Geometric quantization and multiplicities of group representations}, Invent. Math. \textbf{67} (1982), no. 3, 515–-538

\bibitem[GS2]{gs-conv}
V. Guillemin, S. Sternberg {\em  Convexity properties of the moment mapping}. Invent. Math. 
\textbf{67} (1982), no. 3, 491-–513

\bibitem[GS3]{gs-stp}
V. Guillemin, S. Sternberg {\em Symplectic techniques in physics}. Second edition. Cambridge University Press, Cambridge, 1990. xii+468 pp. ISBN: 0-521-38990-9



\bibitem[Ka]{kaw} T. Kawasaki, {\em 
Cohomology of twisted projective spaces and lens complexes}.
Math. Ann. 
\textbf{206} (1973), 243–-248



\bibitem[Ki]{kir}
F. Kirwan,
{\em Cohomology of quotients in symplectic and algebraic geometry}. Mathematical Notes, 
\textbf{31}. Princeton University Press, Princeton, NJ, 1984. i+211 pp. ISBN: 0-691-08370-3


\bibitem[Ko]{ko}
B. Kostant, {\em Quantization and unitary representations. 
I. Prequantization}, Lectures in modern analysis and applications, III, pp. 87--208. 
Lecture Notes in Math., Vol. \textbf{170}, 
Springer, Berlin, 1970

\bibitem[LT]{lt}
E. Lerman, S. Tolman {\em Hamiltonian torus actions on symplectic orbifolds and toric varieties}, 
Trans. Amer. Math. Soc. \textbf{349} (1997), no. 10, 4201–-4230



\bibitem[N]{ness} L. Ness, {\em
A stratification of the null cone via the moment map}.
With an appendix by David Mumford.
Amer. J. Math. 
\textbf{106} (1984), no. 6, 1281–-1329

\bibitem[P1]{pao-IJM} R. Paoletti, {\em Asymptotics of Szeg\"{o} kernels under Hamiltonian
torus actions}, Israel Journal of Mathematics \textbf{191} (2012), no. 1, 363--403
DOI: 10.1007/s11856-011-0212-4



\bibitem[P2]{pao-loa} R. Paoletti, {\em Lower-order asymptotics for Szeg\"{o} and Toeplitz kernels under Hamiltonian circle actions},
Recent advances in algebraic geometry, 321--369, 
London Math. Soc. Lecture Note Ser., \textbf{417}, Cambridge Univ. Press, Cambridge, 2015

\bibitem[P3]{pao-u2} R. Paoletti, {\em 
Conic reductions for Hamitonian actions of $U(2)$ and its maximal torus},
arXiv:2002.08105, to appear in Rendiconti del Circolo Matematico di Palermo Series 2

\bibitem[P4]{pao-JGP} R. Paoletti, {\em
Polarized orbifolds associated to quantized Hamiltonian torus actions}. J. Geom. Phys. 
\textbf{170} (2021), Paper No. 104363






\bibitem[Sj]{sj-holoslice} 
R. Sjamaar, {\em Holomorphic slices, symplectic reduction and multiplicities of representations},
Ann. of Math. (2) \textbf{141} (1995), no. 1, 87–-129



\end{thebibliography}
\end{document}